\documentclass[11pt]{amsart}
\usepackage{amssymb, amsfonts, amsmath, stackrel, enumitem, tikz, tikz-cd, verbatim, xypic}
\usetikzlibrary{intersections, calc, matrix, decorations.pathreplacing}
\usepackage[outline]{contour} \contourlength{2pt}
\usepackage{color}   
\usepackage{hyperref}
\hypersetup{
    colorlinks=false, 
    linktoc=all,     
    linkcolor=blue,  
}
\usepackage{geometry}
\allowdisplaybreaks

\newtheorem{theorem}{Theorem}[section]
\newtheorem{lemma}[theorem]{Lemma}
\newtheorem{corollary}[theorem]{Corollary}
\newtheorem{proposition}[theorem]{Proposition}
\theoremstyle{definition}
\newtheorem{definition}[theorem]{Definition}
\newtheorem{example}[theorem]{Example}
\newtheorem{remark}[theorem]{Remark}
\numberwithin{equation}{section}
\numberwithin{figure}{section}

\newtheorem*{ack}{Acknowledgments}

\numberwithin{equation}{section}
\newcommand{\C}{\mathbb{C}}
\newcommand{\N}{\mathbb{N}}

\newcommand{\Z}{\mathbb{Z}}

\newcommand{\intfib}[1]{\int_{\Delta^{#1}}}
\newcommand{\Cech}{{\v{C}ech} }  
\newcommand{\Tot}{{\bf Tot}}
\newcommand{\tot}{{\bf tot}}

\newcommand{\Chern}{\mathrm{Ch}}
\newcommand{\uu}{[u]}
\newcommand{\uul}{[u]^{\bu\leq 0}}  \newcommand{\ul}{[u]^{\bu\leq 0}}

\newcommand{\Ab}{{\mc{A}b}}
\newcommand{\sAb}{\mc{A}b^{\Del^{op}}}
\newcommand{\sSet}{\mc{S}et^{\Del^{op}}}
\newcommand{\Chain}{\mc {C}h}
\newcommand{\Set}{\mc{S}et}
\newcommand{\Cat}{\mc{C}at}
\newcommand{\OM}{{\bf \Omega}}
\newcommand{\om}{\omega}
\newcommand{\tr}{\mathrm{tr}}
\newcommand{\HVB}{{\bf{HVB}}}
\newcommand{\lsSet}{\lSet^{\Del^{op}}}
\newcommand{\lSet}{{\mc{S}et_{\bf l}}}
\newcommand{\lCat}{{\mc{C}at_{\bf l}}}
\newcommand{\CMan}{\mathbb{C}\mc{M}an}
\newcommand{\Ob}{{\bf Obj}}
\newcommand{\Mor}{{\bf Mor}}
\newcommand{\Cov}{{\mathcal C}ov}
\newcommand{\image}{\mathrm{Im}}

\newcommand{\sgn}{\mathrm{sgn}}
\newcommand{\Del}{{\bf \Delta}}
\newcommand{\Open}{{Open}}

 \newcommand{\Ohol}{\Om_{hol}}
\newcommand{\Nerve}{{\mathcal N}}
\newcommand{\Norm}{N}
\newcommand{\DK}{DK}
\newcommand{\DKSet}{\underline{DK}}

\newcommand{\HVBnabla}{{\mathcal{HVB}^\nabla}}
\newcommand{\mc}{\mathcal }
\newcommand{\Om}{\Omega}
\newcommand{\CS}{{\bf{Ch}}}
\newcommand{\bu}{\bullet}

\newcommand{\ot}{\otimes}
\newcommand{\al}{\alpha}
\newcommand{\stp}{\mathfrak s}
\newcommand{\trunc}{T}
\newcommand{\quot}{Q}
\newcommand{\tabspac}{0.8cm}
\newcommand{\CN}{\mathrm{\check{N}}}   
   
\newcommand{\vC}{{\check{C}}}   
\newcommand{\Delt}[1]{\Delta^{#1}}
\newcommand{\F}{{\mathbf F}}

\author[C. Glass]{Cheyne Glass}
\address{Cheyne Glass, St.~Joseph's College Long Island, Department of Mathematics, 155 W. Roe Blvd., Patchogue, NY 11772}
  \email{cmiller5@sjcny.edu}

\author[M. Miller]{Micah Miller}
\address{Micah Miller, Borough of Manhattan Community College The City University of New York, Department of Mathematics, 199 Chambers Street, New York, NY 10007}
  \email{mmiller@bmcc.cuny.edu}

\author[T. Tradler]{Thomas Tradler}
  \address{Thomas Tradler, New York City College of Technology The City University of New York, Department of Mathematics, 300 Jay Street, Brooklyn, NY 11201}
  \email{ttradler@citytech.cuny.edu}

\author[M. Zeinalian]{Mahmoud Zeinalian}
  \address{Mahmoud Zeinalian, Lehman College, The City University of New York, Department of Mathematics, 250 Bedford Park Blvd W, Bronx, NY 10468}
  \email{mahmoud.zeinalian@lehman.cuny.edu}

\title[The Hodge Chern character of holomorphic connections...]{The Hodge Chern character of holomorphic connections as a map of simplicial presheaves}


\begin{document}
\maketitle
\begin{abstract} We define a map of simplicial presheaves, the Chern character, that assigns to every sequence of composable non connection preserving isomorphisms of vector bundles with holomorphic connections an appropriate sequence of holomorphic forms. We apply this Chern character map to the \Cech nerve of a good cover of a complex manifold and assemble the data by passing to the totalization to obtain a map of simplicial sets. In simplicial degree $0$, this map gives a formula for the Chern character of a bundle in terms of the clutching functions. In simplicial degree $1$, this map gives a formula for the Chern character of bundle maps.  In each simplicial degree beyond $1$, these invariants, defined in terms of the transition functions, govern the compatibilities between the invariants assigned in previous simplicial degrees. In addition to this, we also apply this Chern character to complex Lie groupoids to obtain invariants of bundles on them in terms of the simplicial data. For group actions, these invariants land in suitable complexes calculating various Hodge equivariant cohomologies. In contrast, the de Rham Chern character formula involves additional terms and will appear in a sequel paper. In a sense, these constructions build on a point of view of ``characteristic classes in terms of transition functions" advocated by Raoul Bott \cite{BB, B1, B2, S}, which has been addressed over the years in various forms and degrees (e.g. \cite{S, BSS, KT, TT1, OTT1}), concerning the existence of formulae for the Hodge and de Rham characteristic classes of bundles solely in terms of their clutching functions. 
\end{abstract}
\tableofcontents

\section{Introduction}\label{SEC:Intro}

Let $\HVB(U)=\Nerve ( \HVBnabla(U))$ be the simplicial presheaf that assigns to a complex manifold $U$ the nerve of the category whose objects are holomorphic vector bundles endowed with holomorphic connections, and whose morphisms are holomorphic bundle isomorphisms that ignore the connections. Let $\Ohol^\bu(U)[u]^{\bullet \leq 0}$ be the non-positively graded complex obtained by first tensoring holomorphic differential forms $(\Ohol^\bu(U), d= 0)$, with the polynomial ring in $u$ of degree $-2$, and then quotienting out by all elements in positive degrees.  Then, $\OM(U)=\DKSet(\Ohol^\bu(U)[u]^{\bu \leq 0})$ is the underlying simplicial set of the Dold-Kan functor applied to $\Ohol^\bu(U)[u]^{\bu \leq 0}$. Simply, this is the simplicial set whose $k$-simplices are decorations of all $i$-dimensional faces of the standard $k$-simplex with sequences of forms, all even for $i$ even, and all odd for $i$ odd, in such a way that the alternating sum of all forms sitting on the $(i-1)$-dimensional faces of any $i$-dimensional face add up to $0$. The assignment $U\mapsto \Ohol^\bu(U)[u]^{\bu \leq 0}$ defines a simplicial presheaf $\OM$.

We construct a map of simplicial presheaves $\CS: \HVB \to \OM$, as follows. In simplicial degree $0$, we assign to a holomorphic bundle and a holomorphic connection $(E, \nabla)$ the decoration of the standard $0$-simplex by the sum $dim(E)+0\cdot u +0\cdot u^2+\dots$, where $dim(E)$ is the dimension of the fiber of $E$ and the $j^{th}$ zero denotes the zero $2j$-form.

In simplicial degree $1$, we assign to a bundle isomorphism $g: (E_{0}, \nabla_0) \to (E_{1}, \nabla_1)$ that ignores the holomorphic connections $\nabla_0$ and $\nabla_1$, the decoration of the standard $1$-simplex obtained by the trace of the bundle endomorphism valued Maurer-Cartan form $(g^{-1}dg)u$. Here, $dg$ represents the derivative of $g$ obtained by pre and post composition with the operators $\nabla_0$ and  $\nabla_1$ on the domain and the range.

In simplicial degree $2$, we assign to a pair of compossible morphisms $(E_{0}, \nabla_0)\to (E_{1}, \nabla_1)\to (E_{2}, \nabla_2)$ the labelling of the $7$ faces of the standard $2$-simplex by formulae, where the $0$-codimension face, which is the body of the triangle, is labeled by the trace of the product of the left and right invariant Maurer-Cartan forms. Similarly, we decorate higher simplices with appropriate forms, all of which are compatibly encompassed in the following statement.

\setcounter{section}{2} \setcounter{theorem}{3} \begin{theorem}
The above map is a map $\CS: \HVB \to \OM$ is a map of simplicial presheaves.
\end{theorem}

We note that in this paper we have chosen a situation with vanishing differential, i.e., $(\Ohol^\bu(U), d= 0)$. In a sequel paper, this construction is completed to a map of presheaves whose target is similarly built out of holomorphic forms but now with the differential $\partial$ instead of the zero differential. These discussions are closely related by appropriate Hodge-to-de Rham spectral sequences. In addition, a complete analog of this story in the smooth category, wherein flat connections on smooth vector bundles play the role of holomorphic connections on holomorphic bundles, follows naturally from this description. 

In section \ref{SEC:Totalization}, we apply the simplicial presheaf $\HVB$ to the \Cech nerve simplicial manifold $\CN\mc U$ of a cover $\mc U$ of a complex manifold $X$ to obtain a cosimplicial simplicial set. The totalization of this cosimplicial simplicial set is a simplicial set whose vertices are the vector bundles on $X$ endowed with non-matching holomorphic connections on each open set of the cover $\mc U$. The edges are bundle isomorphisms, which do not necessarily respect the locally chosen connections, etc. 

Similarly, we can evaluate the simplicial presheaf $\OM$ on the \Cech nerve of $\mc U$ and pass to the totalization to obtain a simplicial set.  The vertices of the simplicial set are closed elements of the \Cech complex of holomorphic forms with the zero internal differential.  We refer to this \Cech complex as the \v{C}ech-Hodge complex, in contrast to the \v{C}ech-de Rham complex, which has the $\partial$ operator on the holomorphic forms. The edges of the totalization are witnesses to two such closed elements in the \v{C}ech-Hodge complex representing the same Hodge cohomology class, and similarly for higher simplices, with further elements witnessing how a sum of witnesses in the previous simplicial degree is realized as a coboundary. 

We then look at the map induced by $\CS$ on the totalization. On the $0$-simplices, the totalized map gives a combinatorial formula for the Chern character of a bundle in the \v{C}ech-Hodge complex, in terms of its transition functions. Over the 1-simplices, we obtain a formula for the Hodge-Chern-Simons invariant of bundle isomorphisms in the \Cech complex, with respect to the domain and range connections, and in terms of the transition functions of the bundle. Similar invariants are obtained from the higher simplices.

We note that that totalization has an interesting effect. Before totalization, the map on the vertices was rather trivial, encoding only the rank of the vector bundle. After totalization, the map on vertices becomes a cocycle representative of the total Chern character of the holomorphic bundle in the \v{C}ech-Hodge complex, which is quite non trivial. The following statement summarizes the above.

\setcounter{section}{3} \setcounter{theorem}{2} \begin{corollary}
Given a complex manifold $M$ with a cover $\mc U=\{U_i\}_{i\in I}$, the map $\CS(\CN\mc U):\HVB(\CN\mc U)\to \OM(\CN\mc U)$ is a map of cosimplicial simplicial sets, and thus induces a simplicial set map on the totalization, i.e., a map
\begin{equation*}
\Tot(\CS(\CN\mc U)):\Tot(\HVB(\CN\mc U))\to \Tot(\OM(\CN\mc U))\stackrel \sim\longrightarrow \DKSet(\vC^\bu(\mc U, \Ohol^\bu)\ul).
\end{equation*}
\end{corollary}

When transition functions take values in $G=GL(n, \C)$, there is a more direct description of the Hodge Chern character analyzed in the diagram below, which we describe in section \ref{SEC:product-bundles}.

\setcounter{section}{4} \setcounter{theorem}{16} \begin{theorem} 
There is a commutative diagram of simplicial sets,
\[
\xymatrix{ 
\CMan^{\Del^{op}}(\CN \mc U^{[\bu]},BG) \ar^{\quad \beta}[r] \ar_{\gamma}[d] & \Tot(\HVB(\CN\mc U)) \ar^{\Tot(\CS(\CN\mc U))}[dd]  
\\ 
(\vC^\bu(\mc U^{[\bu]},\Ohol^\bu))^{even}_{closed} \ar_{\iota}[d] &
\\
\DKSet(\vC^\bu(\mc U, \Ohol^\bu)\ul)  & \Tot(\OM(\CN\mc U)) \ar^{}[l] }
\]
\end{theorem}

The above picture relates to the fact that on complex Lie group with a linear $n$-dimensional representation, there is a sequence of forms living on the Cartesian products of $G^{\times p}$, for every $p=0, 1, 2, \dots$ given by 
\begin{equation*}
\Chern:=(n, \tr( g\partial (g^{-1})), \tr( g_1g_2 \partial (g_2^{-1})\partial (g_1^{-1})),\dots).
\end{equation*}

These forms assemble into a single closed element in the complex of Hodge forms on the stack $BG=[*/G]$ represented by the following simplicial manifold:

\[\begin{tikzcd}
  * \arrow[from=r,  transform canvas={yshift=-0.8ex}]
         \arrow[from=r, transform canvas={yshift=0.8ex}]
          \arrow[r]
    & G
        \arrow[from=r,  transform canvas={yshift=-1.6ex}]
        \arrow[r, transform canvas={yshift=-0.8ex}]
        		 \arrow[from=r]
        \arrow[r, transform canvas={yshift=0.8ex}]
            	  \arrow[from=r,  transform canvas={yshift=1.6ex}]
& G \times G 
 \arrow[r, swap, "s_i",  transform canvas={yshift=-0.8ex}]
         \arrow[from=r, swap, "d_j",  transform canvas={yshift=0.8ex}]
        & \cdots.
\end{tikzcd}\]

As mentioned, this simplicial presheaf point of view leads to a map of simplicial sets whose value on the vertices reproduces the Chern character formulae. The value on the $1$-simplices is the Chern character of a map of bundles and higher dimensional simplices of the totalization are new invariants that should be thought of as an infinite hierarchy of Chern character type invariants for composable sequences of bundle maps. We give a description of the sequence of holomorphic invariants in the \v{C}ech-Hodge cochain complex that correctly mirror the sequence of the Chern-Simons invariants present in the smooth picture for bundles. 

In section \ref{SEC:Group-action}, a further application to equivariant theories, and more generally bundles on simplicial manifolds, is worked out. Applying $\CS$ to the stack $[M/G]$, we obtain an induced map of simplicial sets as follows.
\begin{equation*}
\Tot(\CS([M/G])):\Tot(\HVB([M/G]))\to \Tot(\OM([M/G])).
\end{equation*}

We describe $\Tot(\HVB([M/G]))$ more explicitly.

\setcounter{section}{5} \setcounter{theorem}{2} \begin{proposition}
The simplices of $\Tot(\HVB([M/G]))$ have the following interpretation.
\begin{enumerate}
\item
A $0$-cell in $\Tot(\HVB([M/G]))$ consists precisely of a $G$-equivariant bundle, $E$, with connection, $\nabla$, where $\nabla$ is not required to satisfy any condition with respect to the $G$-action.
\item
An $n$-cell in $\Tot(\HVB([M/G]))$ consists precisely of a sequence of $G$-equivariant bundles, $E^{(0)},\dots, E^{(n)}$,  and $G$-equivariant maps, $\al_0,\dots, \al_{n-1}$,
\[ E^{(0)} \xrightarrow{\alpha_0} E^{(1)} \xrightarrow{\alpha_1} \dots \xrightarrow{\alpha_{n-1}} E^{(n)}\] 
where each bundle $E_i \to M$ has a connection $\nabla_i$, which are not required to satisfy any conditions with respect to the $G$-action or the bundle maps.
\end{enumerate}
\end{proposition}

The following corollary states that we can use the map $\Tot(\CS([M/G]))$ from equation \eqref{EQU:Tot(CS(M/G))} as a measure for the connection $\nabla$ to be $G$-invariant.

\setcounter{section}{5} \setcounter{theorem}{4} \begin{corollary}
Let $(E, M, \pi, \rho, \varphi)$ be a $G$-equivariant bundle with connection $\nabla$, which, by proposition \ref{PROP:Tot(HVB(M/G))}\eqref{PROP:Tot(HVB(M/G))_0}, we may interpret as a $0$-simplex in $\Tot(\HVB([M/G]))_0$. If the connection $\nabla$ is $G$-invariant, then $\Tot(\CS([M/G]))$ applied to this is zero in all positive holomorphic form degrees.
\end{corollary}

There is also an infinity homotopy coherent version of all of this, where vector bundles are replaced by derived families whose clutching functions fit together only up to an infinite system of coherent homotopies. This relates to the work of \cite{TT1, OTT1}, which in fact motivated us and was the starting point of our project. Here, we have avoided discussing this homotopy coherent generalization because the strict case is sufficiently rich by itself. The homotopy coherent story, which will be discussed in a forthcoming paper, will be employed to obtain invariants of the derived automorphisms of coherent sheaves on complex manifolds. One foreseeable direction is to develop a commutative diagram of spaces where after applying $\pi_0$ results in the classical Grothendieck-Riemann-Roch (GRR) commutative square. This will extend the differential geometric discussion of GRR established by Toledo and Tong \cite{TT1} and with O'Brian \cite{OTT2} to the entire K-theory spectrum. It will also extend the GRR from ordinary manifolds to the equivariant setting and more generally to simplicial manifolds in an appropriate sense.

\begin{ack}
We would like to thank Domingo Toledo for email correspondences and informing us of his paper with Tong on Green's work \cite{TT2}. While Green's work does not enter this paper, it will be relevant to our forthcoming work on a homotopy coherent version of the discussion here. We also would like to thank Dennis Sullivan for numerous valuable conversations about the local formulae for characteristic classes. Mahmoud Zeinalian would also like to acknowledge a conversation with Julien Grivaux regarding the work of Toledo and Tong, as well as Green's work, on simplicial vector bundles and the Chern character. We would like to add that the forthcoming results of his student, Timothy Hosgood, on simplicial connections and the Chern character are entirely independent of our work. Mahmoud Zeinalian would like to thank the Max Planck Institute and Universit\'e Paris 13 for their hospitality during his visits.
\end{ack}

\setcounter{section}{1} \setcounter{theorem}{0} 
\section{A map of simplicial presheaves  $\CS$}\label{SEC:Prestacks}

In this section, we define two simplicial presheaves and a map between them. First, $\HVB(U)=\Nerve ( \HVBnabla(U))$ is the nerve of holomorphic vector bundles on a complex manifold $U$ and, second, $\OM(U)=\DK(\Ohol^\bu(U)\ul)$ is the Dold-Kan dual of holomorphic forms on $U$. We produce a map $\CS:\HVB\to \OM$ of simplicial presheaves between those, which we will show in section \ref{SEC:Totalization} to be closely related to the Atiyah class.

\begin{definition}
We start by defining the functor $\HVBnabla:\CMan^{op}\to \lCat$ of holomorphic vector bundles with connection of a complex manifold. For a complex manifold $U\in \Ob(\CMan)$, denote by $\HVBnabla(U)\in \lCat$ the (large) groupoid whose objects are finite dimensional holomorphic vector bundles $E\to U$ over $U$ together with holomorphic connection $\nabla_E$ on $E$, and whose morphisms $f\in \HVBnabla(E_0,E_1)$ consist of holomorphic bundle isomorphisms $f:E_0\to E_1$, which need not to respect the connections $\nabla_{E_0}$ and $\nabla_{E_1}$ in any way. Any map of complex manifolds $\varphi:U\to U'$ induces a functor $\HVBnabla(\varphi):\HVBnabla(U')\to \HVBnabla(U)$ via pullback, so that we have a functor $\HVBnabla:\CMan^{op}\to \lCat$. Composing $\HVBnabla$ with the nerve $\Nerve:\lCat\to \lsSet$ thus gives a simplicial presheaf, i.e., a functor $\HVB:=\Nerve\circ \HVBnabla:\CMan^{op}\to \lsSet$.
\end{definition}

Next, for a complex manifold $U\in \CMan$, we consider the algebra $\Ohol^\bu(U)$ of holomorphic differential forms on $U$, and thus have a cochain complex $\Ohol^\bu(U)\uul$, which becomes a simplicial set after applying the Dold-Kan functor. 

\begin{definition}\label{DEF:OM}
For a complex manifold $U\in \Ob(\CMan)$, consider the (non-negatively graded) cochain complex of holomorphic forms $\Ohol^\bu(U)$ on $U$ with zero differential $d=0$. By definition \ref{DEF:Chain-Cat}, $\Ohol^\bu(U)\ul=\quot(\Ohol^\bu(U))$ is a chain complex with zero differential, and, by theorem \ref{THM:Dold-Kan}, the Dold-Kan functor yields a simplicial abelian group $DK(\Ohol^\bu(U)\ul)$, which we think of as a simplicial set, $\OM(U)=\DKSet(\Ohol^\bu(U)\ul)$. Since holomorphic forms pull back via a holomorphic map $\varphi:U\to U'$, this assignment defines a simplicial presheaf $\OM:\CMan^{op}\to \sSet$ by $\OM:=\DKSet(\Ohol^\bu(\cdot)\ul):\CMan^{op}\to \sSet$.
\[
\OM:\CMan^{op}
\stackrel{\Ohol^\bu(-)}{\longrightarrow}\Chain^+
\stackrel{\quot}{\longrightarrow}\Chain^-
\stackrel{\DK}{\longrightarrow}\sAb
\stackrel{\mc F}{\longrightarrow}\sSet
\]
\end{definition}

The main goal of this section is to obtain map of simplicial presheaves  from $\HVB$ to $\OM$, i.e., a natural transformation $\CS: \HVB \to \OM$.
\begin{definition}\label{DEF:CS-map}
We define the Chern character map $\CS: \HVB \to \OM$ by defining for each complex manifold $U\in \Ob(\CMan)$ a map of simplicial set $\CS(U)_\bu:\Nerve(\HVBnabla(U))_\bu\to \DKSet(\Ohol^\bu(U)\ul)_\bu$ explicitly for each simplicial degree $k$, as follows.
\begin{enumerate}
\item[$k=0$:]
A $0$-simplex in the nerve $\Nerve(\HVBnabla(U))_0$  is an object of $\HVBnabla(U)$, i.e., a holomorphic vector bundle $E\to U$ with holomorphic connection $\nabla_E$. To this data, we need to assign a $0$-simplex in $\DKSet(\Ohol^\bu(U)\ul)_0$.  This amounts to associating to $(E \to U, \nabla)$ a polynomial of holomorphic forms $\om_0+\om_2 u+\om_4 u^2+\dots\in \Ohol^\bu(U)\uu^0$, where each $\om_i\in \Ohol^i(U)$. We define $\CS(U)_0$ by mapping $E$ to the constant function $\dim(E)$, the dimension of the fiber of $E$, without any higher $u$-terms. As a short hand, we write $\CS(U)_0(E)$ by labeling the $0$-simplex by $\dim(E)$:
\begin{equation*}
\begin{tikzpicture}[scale=0.5]
\node (E) at (0, 0) {};  \fill (E) circle (3pt) node[right] {$\dim(E)$};
\end{tikzpicture}
\end{equation*}
\item[$k=1$:]
A $1$-simplex in the nerve $\Nerve(\HVBnabla(U))_1$ consists of two holomorphic vector bundles $E_0\to U$ and $E_1\to U$ with connections $\nabla_{E_0}$ and $\nabla_{E_1}$ and a bundle isomorphism $f:E_0\to E_1$, which may not respect the connections. To this data, we assign a $1$-simplex in $\DKSet(\Ohol^\bu(U)\ul)_1$, which is a chain map from $\dots \rightarrow0\rightarrow<e_{0,1}>\rightarrow <e_0,e_1>\rightarrow 0\rightarrow\dots$ (in the notation from example \ref{REM:NZDeltak}) to $\Ohol^\bu(U)\ul$. Assign to $e_0$ and $e_1$ the dimensions $\dim(E_0)=\dim(E_1)$, thought of as elements in $\Ohol^\bu(U)\uu^{0}$, and assign to $e_{0,1}$ the trace $\tr(f^{-1}\nabla_{1,0}(f))\cdot u\in \Ohol^\bu(U)\uu^{-1}$.  Here, $\nabla_{1,0}$ is the induced connection of $\nabla_{E_0}$ and $\nabla_{E_1}$ on $Hom(E_0,E_1)$.  Note that $\tr(f^{-1}\nabla_{1,0}(f))\cdot u\in \Ohol^\bu(U)\uu^{-1}$ has no higher powers of $u$. Informally, we write the chain map $\CS(U)_1(E_0\stackrel f\to E_1):N(\Z\Delt{1})\to \Ohol^\bu(U)\ul$ by labeling the interval as follows:
\begin{equation*}
\begin{tikzpicture}[scale=0.5]
\node (E0) at (0, 0) {}; \fill (E0) circle (4pt) node[above] {$\dim(E_0)$};
\node (E1) at (12, 0) {}; \fill (E1) circle (4pt) node[above] {$\dim(E_1)$};
\draw [->] (E0) -- node[above] {$ \tr(f^{-1}\nabla_{1,0}(f))\cdot u$} (E1);
\end{tikzpicture}
\end{equation*}
 \item[$k\geq 2$:]
 A $k$-simplex in the nerve $\Nerve(\HVBnabla(U))_k$ is a sequence of holomorphic vector bundles $E_0,\dots, E_k$ with holomorphic connections $\nabla_{E_0}, \dots, \nabla_{E_k}$, and holomorphic bundle isomorphisms $E_0\stackrel {f_1}\to E_1\stackrel {f_2}\to\dots\stackrel {f_k}\to E_k$ not necessarily respecting the connections. For $0\leq p<q\leq k$ we denote by $\tilde f_{q,p}:E_p\to E_q$ the composition $\tilde f_{q,p}:=f_{q}\circ \dots\circ f_{p+1}$, i.e., $E_p\stackrel {f_{p+1}}\to E_{p+1}\stackrel {f_{p+2}}\to \dots  \stackrel {f_q}\to E_q$. Now, to a $k$-simplex in the nerve we assign a $k$-simplex in $\DKSet(\Ohol^\bu(U)\ul)_k$, which is a chain map $N(\Z\Delt{k})$ to $\Ohol^\bu(U)\ul$. The generator $e_i$ of $N(\Z\Delt{k})$, where $i=0,\dots, k$, gets assigned $\dim(E_i)\in \Ohol^\bu(U)\uu^0$. For $\ell>0$, the generator $e_{i_0,\dots, i_\ell}$ with $0\leq i_0<\dots<i_\ell\leq k$ gets assigned to the following element in $\Ohol^\bu(U)\uu^{-\ell}$,
 \begin{equation}\label{EQU:tr(f-1nablaf-nablag)}
 \tr\Big(\tilde f_{i_{\ell},i_{0}} ^{-1}  
 \circ\nabla_{i_{\ell},i_{\ell-1}}(\tilde f_{i_{\ell},i_{\ell-1}})  \circ\dots\circ \nabla_{i_2,i_1}(\tilde f_{i_2,i_1})\circ\nabla_{i_1,i_0}(\tilde f_{i_1,i_0})\Big)\cdot u^{\ell},
 \end{equation}
where $\nabla_{q,p}$ is the induced connection on $Hom(E_p,E_q)$ via the connections $\nabla_{E_p}$ and $\nabla_{E_q}$. Informally, we picture the chain map $\CS(U)_k(E_0\stackrel {f_1}\to E_1\stackrel {f_2}\to\dots\stackrel {f_k}\to E_k):N(\Z\Delt{k})\to \Ohol^\bu(U)\ul$ by labeling the cells of a $k$-simplex with the terms from \eqref{EQU:tr(f-1nablaf-nablag)}. For example, for the $k=2$, and the $2$-simplex  $E_0\stackrel f \to E_1\stackrel g \to E_2$ in the nerve $\Nerve(\HVBnabla(U))$, we get
\begin{equation*}
\begin{tikzpicture}[scale=0.5]
\node (E0) at (0, 0) {}; \fill (E0) circle (4pt) node[below] {$\dim(E_0)$};
\node (E1) at (10, 4) {}; \fill (E1) circle (4pt) node[above] {$\dim(E_1)$};
\node (E2) at (20, 0) {}; \fill (E2) circle (4pt) node[below] {$\dim(E_2)$};
\draw [->] (E0) -- node[above left] {$\tr(f^{-1}\nabla_{1,0}(f))\cdot u$} (E1);
\draw [->] (E1) -- node[above right] {$\tr(g^{-1}\nabla_{2,1}(g))\cdot u$} (E2);
\draw [->] (E0) -- node[below] {$\tr((gf)^{-1}\nabla_{2,0}(gf))\cdot u$} (E2);
\node (C) at (10,1.5) {$ \tr((gf)^{-1}\nabla_{2,1}(g)\nabla_{1,0}(f))\cdot u^2$};
\end{tikzpicture}
\end{equation*}
\end{enumerate}
In the next theorem we show that this assignment is well-defined.
\end{definition}

\begin{theorem}\label{THM:CS-natural-transformation}
The assignments from definition \ref{DEF:CS-map} give a map of simplicial presheaves $\CS: \HVB \to \OM$, i.e., a natural transformation of functors $\CMan^{op}\to \lsSet$.
\end{theorem}
\begin{proof}
First, we show that the assignment defined by \eqref{EQU:tr(f-1nablaf-nablag)} is well-defined, i.e., it indeed gives a chain map $N(\Z\Delt{k})\to \Ohol^\bu(U)\ul$. The differential in $N(\Z\Delt{k})$ is $d(e_{i_0,\dots,i_\ell})=\sum_{j=0}^\ell (-1)^j e_{i_0,\dots,\widehat{i_j},\dots, i_\ell}$, while the differential in $\Ohol^\bu(U)\ul$ vanishes, $d=0$, by our choice of taking the zero differential in $\Ohol^\bu(U)$, cf. definition \ref{DEF:OM}. We thus have to show that the images of $\sum_{j=0}^\ell (-1)^j e_{i_0,\dots,\widehat{i_j},\dots, i_\ell}$ also vanish. This image is given by
\begin{multline*}
u^{\ell-1}\cdot \Big[  \tr\Big(\tilde f_{i_{\ell},i_{1}} ^{-1}  
 \circ\nabla_{i_{\ell},i_{\ell-1}}(\tilde f_{i_{\ell},i_{\ell-1}})  \circ\dots\circ \nabla_{i_2,i_1}(\tilde f_{i_2,i_1})\Big)
 \\
 +\sum_{j=1}^{\ell-1}(-1)^j\cdot  \tr\Big(\tilde f_{i_{\ell},i_{0}} ^{-1}  
 \circ\nabla_{i_{\ell},i_{\ell-1}}(\tilde f_{i_{\ell},i_{\ell-1}})  \circ\dots\circ\nabla_{i_{j+1},i_{j-1}}(\tilde f_{i_{j+1},i_{j-1}})\circ\dots\circ\nabla_{i_1,i_0}(\tilde f_{i_1,i_0})\Big)
\\
 +(-1)^\ell \cdot \tr\Big(\tilde f_{i_{\ell-1},i_{0}} ^{-1}  
 \circ\nabla_{i_{\ell-1},i_{\ell-2}}(\tilde f_{i_{\ell-1},i_{\ell-2}})  \circ\dots\circ \nabla_{i_1,i_0}(\tilde f_{i_1,i_0})\Big)\Big].
\end{multline*}
Using the Leibniz property $\nabla_{i_{j+1},i_{j-1}}(\tilde f_{i_{j+1}, i_{j-1}})= \nabla_{i_{j+1},i_{j-1}}(\tilde f_{i_{j+1},i_{j}}\circ \tilde f_{i_{j},i_{j-1}})= \nabla_{i_{j+1},i_{j}}(\tilde f_{i_{j+1},i_{j}})\circ \tilde f_{i_{j},i_{j-1}}+ \tilde f_{i_{j+1},i_{j}}\circ \nabla_{i_{j},i_{j-1}}(\tilde f_{i_{j},i_{j-1}})$, together with $\tilde f_{i_\ell,i_1}^{-1}=\tilde f_{i_1,i_0}\circ \tilde f_{i_\ell,i_0}^{-1}$ and $\tilde f_{i_{\ell-1},i_{0}} ^{-1}=\tilde f_{i_{\ell},i_{0}} ^{-1}\circ \tilde f_{i_{\ell},i_{\ell-1}}$, and the cyclic property of the trace, shows, that the above terms indeed vanish.

Finally, we note that $\CS$ is a map of simplicial presheaves, i.e., a natural transformation. For a morphism $\varphi:U\to U'$ the induced simplicial sets are all given by pullback via $\varphi$, and equation \eqref{EQU:tr(f-1nablaf-nablag)} respects pullbacks.
\end{proof}

\section{Chern character induced via totalization}\label{SEC:Totalization}

Not every holomorphic vector bundle $E\to M$ over a complex manifold $M$ admits a holomorphic connection, and thus will not be an object in the category $\HVBnabla(M)$ of holomorphic vector bundles with holomorphic connections. However, we can cover the underlying complex manifold by open sets such that each restriction of the bundle to an open set has a holomorphic connection. By taking the limit of such a cover, we obtain a Chern character map associated to $E$. In fact, when taking limits, the holomorphic Chern character as defined by Atiyah can be recovered as the $0$-simplex part of this Chern character map, while higher simplices naturally yield higher Chern-Simons forms.

\subsection{Totalization of $\CS$}

We begin by describing the category of covers $\Cov_M$ of a complex manifold $M\in \Ob(\CMan)$.

\begin{definition}\label{DEF:Cov+NU}
Let $M\in \Ob(\CMan)$ be a complex manifold, and denote by $\Open_M:=\{U\subset M: U\text{ is open}\}$ the set of all open subsets of $M$. By definition, an (open) cover $\mc U$ of $M$ consists of an index set $I$, and a map $\alpha:I\to \Open_M$ such that $\bigcup_{i\in I}\alpha(i)=M$. We also write this as $\mc U=\{U_i\}_{i\in I}$ for $U_i=\al(i)$. Next, we make the covers of $M$ into a category $\Cov_M$ by letting the objects of $\Cov_M$ consists of covers of $M$, while a morphism $\Cov_M(\mc U, \mc U')$ consists of a map $f:I\to I'$ such that $\al=\al'\circ f$,
\[
\begin{tikzpicture}[scale=0.5]
\node (E0) at (0,2) {$I$}; 
\node (E1) at (0,0) {$I'$};
\node (E2) at (5,1) {$\Open_M$}; 
\draw [->] (E0) -- node[left] {$f$} (E1);
\draw [->] (E1) -- node[below] {$\al'$} (E2);
\draw [->] (E0) -- node[above] {$\al$} (E2);
\end{tikzpicture}
\]

There is a functor $\CN:\Cov_M\to \CMan^{\Del^{op}}$ called the \v{C}ech nerve of a cover, which we define now. Let $\mc U=\{U_i\}_{i\in I}$ be a cover, and denote by $U_{i_0,\dots, i_k}:=U_{i_0}\cap\dots \cap U_{i_k}$. Then, define the simplicial manifold $\CN\mc U$ by setting the $k$-simplices of $\CN\mc U$ to be the disjoint union of the $k$-fold intersections, i.e., $\CN\mc U_k:=\coprod_{i_0,\dots, i_k\in I} U_{i_0,\dots, i_k}$. Then $\CN\mc U:\Del^{op}\to \CMan$ is a simplicial complex manifold with face maps $d_j:\coprod_{i_0,\dots, i_k\in I} U_{i_0,\dots, i_k}\to \coprod_{i'_0,\dots, i'_{k-1}\in I} U_{i'_0,\dots, i'_{k-1}}$ induced by the inclusions $U_{i_0,\dots, i_k}\stackrel {inc}\hookrightarrow U_{i_0,\dots, \widehat{i_j},\dots, i_k}$ and degeneracies $s_j:\coprod_{i_0,\dots, i_k\in I} U_{i_0,\dots, i_k}\to \coprod_{i'_0,\dots, i'_{k+1}\in I} U_{i'_0,\dots, i'_{k+1}}$ induced by the identity maps $U_{i_0,\dots, i_k}\stackrel {id} \rightarrow  U_{i_0,\dots, i_j,i_j,\dots, i_k}$. Indeed, all the simplicial identities follow by a direct check. Below, we will slightly abuse notation by considering $\CN \mc U$ both as $\CN\mc U:\Del^{op}\to \CMan$ and $\CN\mc U:\Del\to \CMan^{op}$.
\end{definition}

\begin{proposition}\label{PROP:cosimpl-simpl-set}
Let $M\in \Ob( \CMan)$ be a complex manifold, and let $\mc U=\{U_i\}_{i\in I}$ be an open cover of $M$. Composing $\CN\mc U:\Del\to \CMan^{op}$ with $\HVB$ yields a cosimplicial simplicial set $\HVB(\mc\CN\mc U):=\HVB\circ \mc\CN\mc U:\Del\to \CMan^{op}\to \lsSet$. Similarly, composing $\CN\mc U:\Del\to \CMan^{op}$ with $\OM$ yields a cosimplicial simplicial set $\OM(\mc\CN\mc U):=\OM\circ \mc \CN\mc U:\Del\to \CMan^{op}\to \sSet$. Furthermore,  composing $\CS$ with $\CN(\mc U)$ yields a map $\CS(\CN\mc U):\HVB(\CN\mc U)\to \OM(\CN\mc U)$ of cosimplicial simplicial sets.
\end{proposition}
\begin{proof}
This follows from $\CS:\HVB\to \OM$ being a natural transformation by theorem \ref{THM:CS-natural-transformation} composed with $\CN\mc U$, resulting in $\CS\circ \CN\mc U:\HVB\circ \CN\mc U\to \OM\circ \CN\mc U$, which is a natural transformation of functors $\Del\to \lsSet$.
\end{proof}

By proposition \ref{PROP:cosimpl-simpl-set}, both $\HVB(\CN\mc U)$ and $\OM(\CN\mc U)$ are cosimplicial simplicial sets. We may thus apply the totalization. The relevant definitions for the totalization can be found in appendix \ref{APPX:totalization}.

\begin{corollary}
Given a complex manifold $M$ with a cover $\mc U=\{U_i\}_{i\in I}$, the map $\CS(\CN\mc U):\HVB(\CN\mc U)\to \OM(\CN\mc U)$ is a map of cosimplicial simplicial sets, and thus induces a simplicial set map on the totalization, i.e., a map
\begin{equation}\label{EQU:TOT(CS)}
\Tot(\CS(\CN\mc U)):\Tot(\HVB(\CN\mc U))\to \Tot(\OM(\CN\mc U)).
\end{equation}
\end{corollary}

Applying the totalization to the cosimplicial simplicial set $\HVB(\CN\mc U)$ gives, by definition, a simplicial set $\Tot(\HVB(\CN\mc U))$. The $0$-simplices of this simplicial set are given by arbitrary holomorphic vector bundles $E$ on $U_i$ together with choices of local holomorphic connections on each open set $U_i$ of the cover $\mc U$, as stated more precisely in the next proposition.
\begin{proposition}\label{PROP:Tot(Hol(U))0}
Let $\mc U=\{U_i\}_{i\in I}$ be an open cover of a complex manifold $M$. Then, the $0$-simplices of  $\Tot(\HVB(\CN\mc U))$ are given by a choice of holomorphic bundles $E_i\to U_i$ with holomorphic connections $\nabla_i$, and holomorphic bundle isomorphisms $g_{i,j}:E_j|_{U_{i,j}}\to E_i|_{U_{i,j}}$ (not necessarily respecting the connections) satisfying the cocycle condition $g_{i,j}|_{U_{i,j,k}}\circ g_{j,k}|_{U_{i,j,k}}=g_{i,k}|_{U_{i,j,k}}$ on $U_{i,j,k}$, as well as $g_{i,i}=id_{E_i}$.
\end{proposition}
\begin{proof}
First note that, by definition, $\Tot(\HVB(\CN\mc U))$ is a simplicial set, which is determined by a product $\prod_{[\ell]\in \Ob(\Del)} (\HVB(\CN\mc U_\ell))^{\Delt{\ell}}$ of simplicial sets, whose $k$-simplices consist of simplicial set maps $\prod_{[\ell]\in \Ob(\Del)} {\lsSet}(\Delt{\ell}\times \Delt{k},\HVB(\CN\mc U_\ell))$. Thus, a $0$-simplex is given by a sequence of simplicial set maps $\Delt{\ell}\times \Delt{0}\to \HVB(\CN\mc U_\ell)$ for $\ell=0,1,2,\dots$. Since each such map is determined by its image on the unique maximal non-degenerate $\ell$-simplex, this amounts to a sequence of elements in $\HVB(\CN\mc U_\ell)_\ell$, i.e., a holomorphic vector bundle $E$ with holomorphic connection on $\CN\mc U_0=\coprod_i U_i$, two holomorphic vector bundles $\tilde E_0, \tilde E_1$ with holomorphic connection over $\CN\mc U_1=\coprod_{i_0,i_1}U_{i_0,i_1}$ and a morphism ${\tilde f}:\tilde E_0\to \tilde E_1$ not respecting the connections, three holomorphic vector bundles $\tilde {\tilde {E_0}}, \tilde {\tilde {E_1}}, \tilde {\tilde {E_2}}$ with holomorphic connection over $\CN\mc U_2=\coprod_{i_0,i_1,i_2}U_{i_0,i_1,i_2}$ and morphisms $\tilde {\tilde {E_0}}\stackrel {\tilde{ \tilde {f_0}}} \to \tilde {\tilde {E_1}}\stackrel {\tilde {\tilde {f_1}}} \to \tilde {\tilde {E_2}}$  not respecting the connections, etc. However, in the totalization, this data is not independent.

First, consider $\varphi:[0]\to [1], \varphi(0)=r$, where $r$ is $0$ or $1$. Use $\phi$ from \eqref{DEF:Tot} to map the $\ell=1$ component $\tilde E_0\stackrel {\tilde f} \to \tilde E_1$ to the $\varphi:[0]\to [1]$ component, which gives the bundle $\tilde E_{r}$ interpreted as a simplicial set morphism $\Delt{0}\times \Delt{0}\to \HVB(\CN\mc U_1)$. On the other hand, using $\psi$ from \eqref{DEF:Tot} to map the $\ell=0$ component $E$ to the $\varphi:[0]\to [1]$ component gives $(\coprod_{i_0,i_1} U_{i_0,i_1}\stackrel {inc_r}\to \coprod_i U_i)^*(E)$, i.e., the pullback of $E$ under the inclusions $inc_r:U_{i_0,i_1}\hookrightarrow U_{i_r}$. Since these coincide in the equalizer, we see that $\tilde E_r$ is just the pullback of $E$ under the inclusion $inc_r$. Writing $E=\coprod_i E_i$ over $\coprod_i U_i$, we see that $\tilde E_0=\coprod_{i_0,i_1} E_{i_0}|_{U_{i_0,i_1}}$ and $\tilde E_1=\coprod_{i_0,i_1} E_{i_1}|_{U_{i_0,i_1}}$. Similar arguments show that all higher $\tilde{\tilde{E_i}},\tilde{\tilde{\tilde{E_i}}}\dots$, are pullbacks of $E$ under inclusions mapping $U_{i_0,\dots i_m}\hookrightarrow U_{i_r}$, obtained by considering the component $\rho:[0]\to [m], 0\mapsto r$.

Next, considering components $\rho:[1]\to [m], \rho(0)=r, \rho(1)=s$, for some $0\leq r\leq s\leq m$, shows that all morphisms $\tilde{\tilde {f_j}},\dots$, are induced by pullbacks of ${\tilde f}:\tilde E_0\to \tilde E_1$ via inclusions. In particular, if we write the map ${\tilde f}:\coprod_{i_0,i_1}E_{i_0}|_{U_{i_0,i_1}}\to \coprod_{i_0,i_1}E_{i_1}|_{U_{i_0,i_1}}$ in $(i_0,i_1)$-components as $\tilde f=\coprod_{i_0,i_1} g_{i_1,i_0}$, where $g_{i_1,i_0}:E_{i_0}|_{U_{i_0,i_1}}\to E_{i_1}|_{U_{i_0,i_1}}$, then the $2$-simplex $\tilde {\tilde {E_0}}\stackrel {\tilde{ \tilde {f_0}}} \to \tilde {\tilde {E_1}}\stackrel {\tilde {\tilde {f_1}}} \to \tilde {\tilde {E_2}}$ on $\mc U_2$ from above is given by ${\tilde{ \tilde {f_0}}} =\coprod_{i_0,i_1,i_2}g_{i_1,i_0}|_{U_{i_0,i_1,i_2}}:\coprod_{i_0,i_1,i_2} E_{i_0}|_{U_{i_0,i_1,i_2}}\to \coprod_{i_0,i_1,i_2} E_{i_1}|_{U_{i_0,i_1,i_2}}$ and ${\tilde{ \tilde {f_1}}} =\coprod_{i_0,i_1,i_2}g_{i_2,i_1}|_{U_{i_0,i_1,i_2}}:\coprod_{i_0,i_1,i_2} E_{i_1}|_{U_{i_0,i_1,i_2}}\to \coprod_{i_0,i_1,i_2} E_{i_2}|_{U_{i_0,i_1,i_2}}$, while the composition ${\tilde{ \tilde {f_1}}} \circ {\tilde{ \tilde {f_0}}} =\coprod_{i_0,i_1,i_2}g_{i_2,i_0}|_{U_{i_0,i_1,i_2}}:\coprod_{i_0,i_1,i_2} E_{i_0}|_{U_{i_0,i_1,i_2}} \to \coprod_{i_0,i_1,i_2} E_{i_2}|_{U_{i_0,i_1,i_2}}$. Therefore, the functions $\{g_{i,j}\}_{i,j}$ satisfy the cocycle condition $g_{i_2,i_1}\circ g_{i_1,i_0}=g_{i_2,i_0}$ on triple intersections $U_{i_0,i_1,i_2}$, and we may thus interpret the $\{g_{i,j}\}_{i,j\in I}$ as transition functions for a global holomorphic vector bundle on $M$, so that on the cover $\mc U$ we have locally chosen holomorphic connections.

Finally, we note that there are no further higher conditions, since the higher restrictions on the $\ell$-simplices in $\HVB(\CN\mc U_\ell)_\ell$ coming from \eqref{DEF:Tot} are implied by the cocycle condition; see figure \ref{fig:E0-bundle}.
\begin{figure}[h]
\[
\scalebox{.95}{
\begin{tikzpicture}[scale=1]
\node (E0) at (1.5,1.5) {$E_i$}; 

\node (E0') at (3,1.5) {$E_j|_{U_{i,j}}$};
\node (E1') at (5,1.5) {$E_i|_{U_{i,j}}$}; 
\draw [->] (E0') -- node[above] {$g_{i,j}$} (E1');

\node (E0'') at (7,0.5) {$E_k|_{U_{i,j,k}}$};
\node (E1'') at (8,2) {$E_j|_{U_{i,j,k}}$}; 
\node (E2'') at (9,0.5) {$E_i|_{U_{i,j,k}}$}; 
\draw [->] (E0'') -- node[left] {$g_{j,k}|_{U_{i,j,k}}$} (E1'');
\draw [->] (E1'') -- node[right] {$g_{i,j}|_{U_{i,j,k}}$} (E2'');
\draw [->] (E0'') -- node[below] {$g_{i,k}|_{U_{i,j,k}}$} (E2'');

\node (E0''') at (11,0.5) {$E_\ell|_{U_{i,j,k,\ell}}$};
\node (E1''') at (12,2) {$E_k|_{U_{i,j,k,\ell}}$}; 
\node (E2''') at (14,0.5) {$E_j|_{U_{i,j,k,\ell}}$}; 
\node (E3''') at (15,2) {$E_i|_{U_{i,j,k,\ell}}$}; 
\draw [->] (E0''') -- node[left] {$g_{k,\ell}|_{\dots}$} (E1''');
\draw [->] (E0''') -- node[below] {$g_{j,\ell}|_{\dots}$} (E2''');
\draw [->] (E0''') -- node[above] {$g_{i,\ell}|_{\dots}$} (E3''');
\draw [->] (E1''') -- node[below] {$g_{j,k}|_{\dots}$} (E2''');
\draw [->] (E1''') -- node[above] {$g_{i,k}|_{\dots}$} (E3''');
\draw [->] (E2''') -- node[right] {$g_{i,j}|_{\dots}$} (E3''');

\node (F) at (16.5,1.5) {$\dots$};
\end{tikzpicture}
}
\]
\caption{Higher relations in $\Tot(\HVB(\CN \mc U))_0$ are induced by the cocycle condition} \label{fig:E0-bundle}
\end{figure}

This completes the proof.
\end{proof}

Although not all holomorphic vector bundles admit a holomorphic connection, this is certainly true locally.
\begin{lemma}\label{LEM:E-given-local}
If $\pi:E\to M$ is a holomorphic vector bundle over $M$, then there exists a cover $\mc U=\{U_i\}_{i\in I}$ of $M$ such that for each $i\in I$, the restriction $E|_{U_i}\to U_i$ has a holomorphic connection. In particular, each holomorphic vector bundle with such a choice of cover gives a $0$-simplex in $\Tot(\HVB(\CN\mc U))$.
\end{lemma}
\begin{proof}
Choose a local trivialization $\{\psi_i: D_i\times \C^d\to E\}_{i\in I}$ via trivial holomorphic bundles $D_i\times \C^d\to D_i$, where $D_i\subset \C^m$ is an open disk. Then, the holomorphic connection $\partial=\sum_k dz_k \frac{\partial}{\partial z_k}$ on $D_i\times \C^d\to D_i$ can be transported to a holomorphic connection on $E|_{\image(\psi_i)}\to U_i:=\pi(E|_{\image(\psi_i)})$ via pullback by $\psi_i^{-1}$.

The last statement follows by proposition \ref{PROP:Tot(Hol(U))0}.
\end{proof}

The importance of the above map of simplicial sets is that for $0$-simplices we recover the Atiyah's Chern character, cf. \cite{A}.
\begin{proposition}
The map from \eqref{EQU:TOT(CS)} on $0$-simplices coincides with the Chern character of a holomorphic vector bundle as defined by O'Brian, Toledo, Tong in \cite{OTT1} applied to the strict case. More precisely, for a $0$-simplex given by the local data $(\{E_i\to U_i,\nabla_i\}_{i\in I},\{g_{i,j}\}_{i,j\in I})$ from proposition \ref{PROP:Tot(Hol(U))0}, $\Tot(\CS(\CN\mc U))_0$ maps this to the $0$-simplex in $\Tot(\OM(\CN\mc U))_0$, given by the following sequence of holomorphic forms on $\coprod_{i_0,\dots, i_\ell} U_{i_0,\dots, i_\ell}$ for $\ell\geq 0$:
\begin{equation}\label{EQU:Tot(CS(U))0}
\coprod_{i_0,\dots, i_\ell}
 \tr\Big((g_{i_{\ell},i_{\ell-1}}\circ\dots \circ g_{i_1,i_0})^{-1}  
 \circ\nabla_{i_{\ell},i_{\ell-1}}(g_{i_{\ell},i_{\ell-1}})  \circ\dots\circ\nabla_{i_1,i_0}(g_{i_1,i_0})\Big)\cdot u^\ell
\end{equation}
\end{proposition}
\begin{proof}
By the proof of proposition \ref{PROP:Tot(Hol(U))0}, the $0$-simplex of $\Tot(\HVB(\CN\mc U))$ is a sequence of $\ell$ composable morphisms
\[
\coprod_{i_0,\dots, i_\ell}E_{i_0}|_{U_{i_0,\dots, i_\ell}}
\stackrel {\coprod g_{i_1,i_0}} \longrightarrow
\coprod_{i_0,\dots, i_\ell}E_{i_1}|_{U_{i_0,\dots, i_\ell}}
\stackrel {\coprod g_{i_2,i_1}} \longrightarrow
\dots
\stackrel {\coprod g_{i_\ell,i_{\ell-1}}} \longrightarrow
\coprod_{i_0,\dots, i_\ell}E_{i_\ell}|_{U_{i_0,\dots, i_\ell}},
\]
for $\ell\geq 0$, which do not (necessarily) respect the connections $\nabla_i$ restricted to $U_{i_0,\dots, i_\ell}$. By definition \ref{DEF:CS-map}, and, in particular, equation \eqref{EQU:tr(f-1nablaf-nablag)}, $\CS$ maps this to \eqref{EQU:Tot(CS(U))0} on the top non-degenerate $\ell$-simplex of $\Tot(\OM(\CN\mc U))_0$.
\end{proof}

\begin{remark}
Note that the map \eqref{EQU:TOT(CS)} is producing not only the Chern character via the Atiyah class on the $0$-simplices of $\Tot(\HVB(\CN\mc U))$, but a host of Chern-Simons type invariants for holomorphic bundles on the higher simplices. We will revisit these invariants in a future paper.
\end{remark}

\subsection{Totalization of cosimplicial non-positively and non-negatively graded complexes}

There is an even more explicit relationship between the formulae in \eqref{EQU:Tot(CS(U))0} and the map constructed in \cite{OTT1}. To see this we will interpret the $0$-simplices of $\Tot(\OM(\CN\mc U))$ as a \v{C}ech complex with values in holomorphic forms.  We first need to make some general statements about the \v{C}ech cochain complex.

\begin{definition}\label{DEF:Presheaf-Cech}
Let $\mc A$ be a presheaf of non-negatively graded cochain complexes on a manifold $M$, and let $\mc U=\{U_i\}_{i\in I}$ be an open cover of $M$.  We write $\mathcal A_{i_0, \cdots, i_n} =\mathcal A(U_{i_0, \cdots, i_n})$ and $\mathcal A^k_{i_0, \cdots, i_n} =\mathcal A^k(U_{i_0, \cdots, i_n})$ for the degree $k$ component, and write $d_{\mc A}$ for the internal differential of $\mc A$. From this data, there are two ways to obtain a cochain complex.

First, we define the \v{C}ech cochain complex $\vC^\bu(\mc U, \mc A)$ of $\mc A$ for the cover $\mc U$ by setting
\[
\vC^n(\mc U, \mc A)= \prod_{i_0, \cdots, i_n\in I} \mathcal A_{i_0, \cdots, i_n},
\]
where an element $a\in \mc A^k_{i_0,\dots, i_n}$ is of total degree $|a|=n+k$. The \v{C}ech differential $\delta:\vC^\bullet(\mathcal U, \mc A)\to \vC^{\bullet+1}(\mathcal U, \mc A)$ acts on an element $c=\{ c_{i_0 , \ldots , i_n}\}_{i_0,\dots, i_n\in I}\in \vC^\bullet(\mathcal U, \mc A)$ with $c_{i_0 , \ldots , i_n}\in \mc A_{i_0,\dots, i_n}$ via 
\begin{equation}\label{EQU:chdel}
(\delta(c))_{i_0,\dots,i_{n+1}}:=\sum_{j=0}^{n+1} (-1)^j\cdot  {c_{i_0 \cdots \widehat i_j \cdots i_{n+1}}}|_{U_{i_0,\dots,i_{n+1}}}.
\end{equation}
Since $\delta^2=d_{\mc A}^2=d_{\mc A}\delta-\delta d_{\mc A}=0$, we can take the total differential 
\begin{equation}\label{EQU:D-on-C(U,A)}
D(c)=\delta(c)-(-1)^{|c|}d_{\mc A}(c)
\end{equation}
 on $\vC^n(\mc U, \mc A)$ wich satisfies $D^2=0$. Furthermore, for two covers $\mc U=\{U_i\}_{i\in I}$ and $\mc U'=\{U'_{i'}\}_{i'\in I'}$, and a morphism of covers $f\in \Cov_M(\mc U, \mc U')$, there is an induced cochain map $\vC^\bu(\mc U', \mc A)\to \vC^\bu(\mc U, \mc A)$, $\{c_{i'_0,\dots, i'_n}\}_{i'_0,\dots, i'_n\in I'}\mapsto \{c_{f(i_0),\dots, f(i_n)}\}_{i_0,\dots, i_n\in I}$, since $U_{i_0,\dots,i_n}=U'_{f(i_0), \dots, f(i_n)}$. Thus, we have a functor $\vC^\bu(.,\mc A):\Cov_M^{op}\to \Chain^+$.

Alternatively, there is a cosimplicial non-negatively graded cochain complex $A:\Del\to \Chain^+$ given by the assignment 
\begin{eqnarray*}
A=A^{\bu,\bu}: \Del &\rightarrow& \Chain^+ \\
 {[}n{]}  &\mapsto & A^{n,\bu}:=\prod_{i_0, \cdots, i_n\in I} \mc A_{i_0 \cdots i_n} 
\end{eqnarray*}
In particular, $A^{n,\bu}$ in degree $k$ is $A^{n,k}:= \prod_{i_0, \cdots, i_n} \mc A^k_{i_0 \cdots i_n}$. We may take the total complex of $A$, denoted $\tot(A)$. Recall from \eqref{EQU:Def:tot}, that the total complex of $A$ is defined as $\tot (A) = \bigoplus_n A^{n, \bullet} [n]$, where $A^{n, \bullet}[n]$ is the cochain complex $A^{n, \bullet}$ shifted up by $n$ and the differential is as in \eqref{EQU:d-in-tot}. For two covers $\mc U=\{U_i\}_{i\in I}$ and $\mc U'=\{U'_{i'}\}_{i'\in I'}$, and a morphism of covers $f\in \Cov_M(\mc U, \mc U')$, there is an identity map $\mc A(U'_{f(i_0),\dots, f(i_n)})\stackrel =\to \mc A(U_{i_0,\dots, i_n})$, which induces cochain maps $\prod_{i'_0, \cdots, i'_n\in I'} \mc A_{i'_0 \cdots i'_n}\to \prod_{i_0, \cdots, i_n\in I} \mc A_{i_0 \cdots i_n}$, $\{c_{i'_0,\dots, i'_n}\}_{i'_0,\dots, i'_n\in I'}\mapsto \{c_{f(i_0),\dots, f(i_n)}\}_{i_0,\dots, i_n\in I}$, which assemble to a map of cosimplicial non-negatively graded cochain complexes. Thus, $\tot(A)$ is also a functor $\Cov_M^{op}\to \Chain^+$.
\end{definition}

The next lemma shows that the two constructions in definition \ref{DEF:Presheaf-Cech} are naturally equivalent.

\begin{lemma}\label{lemma:TotandCech}
Let $\mc U=\{U_i\}_{i\in I}$ be an open cover on a manifold $M$, $\mc A$ a presheaf of non-negatively graded cochain complexes on $M$, and $A$ be the cosimplicial non-negatively graded cochain complex associated to $\mc A$.  Then there is an isomorphism $\tot(A) \to \vC^\bullet(\mc U, \mc A)$ from the totalization to the \v{C}ech cochain complex . 

Moreover, the isomorphisms $\tot(A)\to \vC^\bullet(\mc U, \mc A)$ yield a natural equivalence of functors $\Cov_M^{op}\to \Chain^+$.
\end{lemma} 

\begin{proof}
An element of degree $k$ in $\prod_\ell A^{\ell, \bullet}[\ell]$ is a collection of elements $c^{j,k-j} \in A^{j,k-j}[j]$, where $j \geq 0$.   An element of total degree $k$ in the \v{C}ech complex $\vC^\bullet(\mc U, \mc A)$ is a collection of elements $c^{j,k-j} \in \vC^i (\mc U, \mc A)$, where $c^{j, k-j}$ associates to an open set $U_{i_0 \cdots i_j}$ an element $c^{j, k-j} \in \mc A^{k-j}_{i_0 \cdots i_j}$.  By definition, $\mc A^{k-j}_{i_0 \cdots i_j}$ is a factor in $A^{j, k-j}$. Since the differential $D$ in \eqref{EQU:D-on-C(U,A)} and $d$ in \eqref{EQU:d-in-tot} differ by a factor $(-1)^{|c|+1}$, the cochain isomorphism $\tot(A) \to \vC^\bullet(\mc U, \mc A)$ is given by $c^{j,k-j}\mapsto (-1)^{\frac{|c^{j,k-j}|\cdot (|c^{j,k-j}|+1)}{2}}\cdot c^{j,k-j}$. This proves the first statement.

For the second statement, note that since a morphism of covers acts on the indices of the collections in $\prod_{i_0,\dots, i_n} \mc A_{i_0,\dots, i_n}$ and $\vC^\bu(\mc U, \mc A)$ in the same way (as described in definition \ref{DEF:Presheaf-Cech}), these isomorphisms induce a natural transformation. 
\end{proof}

Given a cosimplicial non-negatively graded cochain complex $A \in (\Chain^+)^{\Del}$, we get a non-positively graded cochain complex by applying the functor $\quot$ and taking totalization, $\Tot(\quot A)$; cf. appendix \ref{APPX:totalization}.  Alternatively, we can take the total complex and apply the functor $\quot$, $Q (\tot (A))$.  The following lemma shows that these two cochain complexes are equivalent. 
\begin{lemma}\label{lemma:Totandq}
Let $A: \Del \rightarrow \Chain^+$ be a cosimplicial non-negatively graded cochain complex.  Then, 
\begin{equation*}
\Tot(\bigoplus_\ell \quot(A^{\ell, \bullet})) \cong \quot ( \tot ( A))
\end{equation*}
\end{lemma} 

\begin{proof}
First note that by lemma \ref{lemma:TotComplex},  $\tot(A)=\prod_{\ell}A^{\ell, \bullet}[\ell] $ is the equalizer
\begin{eqnarray*}
eq: \prod_{[\ell]} Hom^\bullet(N(\Z \Delt{\ell}), A^{\ell, \bullet}) \rightrightarrows \prod_{[m] \rightarrow [n]} Hom^\bullet(N(\Z \Delt{m}), A^{n, \bullet}) 
\end{eqnarray*}
Since $\quot$ is a right adjoint, it commutes with limits.  The equalizer is a limit, so the right hand side of the equation can be re-written
\begin{eqnarray*}
\quot \left( \prod_{\ell} A^{\ell, \bullet}[\ell] \right)  &=& \quot \left(eq: \prod_{[\ell]} Hom^\bullet(N(\Z \Delt{\ell}), A^{\ell, \bullet}) \rightrightarrows \prod_{[m] \rightarrow [n]} Hom^\bullet(N(\Z \Delt{m}), A^{n, \bullet}) \right) \\
&=&eq: \prod_{[\ell]} \quot Hom^\bullet(N(\Z \Delt{\ell}), A^{\ell, \bullet}) \rightrightarrows \prod_{[m] \rightarrow [n]} \quot  Hom^\bullet(N(\Z \Delt{m}), A^{n, \bullet}) 
\end{eqnarray*}
Use the Hom-Tensor adjunction to rewrite  $Hom^\bullet(N(\Z \Delt{\ell}), A^{\ell, \bullet})$ as $\Norm^\bullet(\Z\Delt{\ell} )\otimes A^{\ell, \bullet}$, where $\Norm^\bullet(\Z\Delt{\ell})$ is the normalized cochain complex on $\Delt{\ell}$.  Then 
\begin{equation}\label{eqnarray:Totq}
\quot(Hom^\bullet(\Norm(\Z\Delt{\ell}, A^{\ell,\bu}))) = \frac{ \Norm^\bullet(\Z\Delt{\ell}) \otimes A^{\ell, \bullet}\ot \Z[u]}{ (\Norm^\bullet(\Z\Delt{\ell}) \otimes A^{\ell, \bullet}\ot \Z[u])^{\bu > 0}}
\end{equation}

We compare this expression to 
\begin{equation*}
\Tot(\bigoplus_\ell \quot(A^{\ell, \bullet}) )= eq: \prod_{[\ell]} (\quot(A^{\ell, \bullet}))^{\Delt{\ell}} \rightrightarrows \prod_{[m] \rightarrow[n]} (\quot(A^{n, \bullet}))^{\Delt{m}}
\end{equation*}
By definition $(\quot(A^{\ell, \bullet}))^{\Delt{\ell}} $ is equal to $q(Hom^\bu(\Norm(\Z\Delt{\ell}), \quot(A^{\ell, \bullet})))$; cf. example \ref{example:negativechaincomplex} item (6) on page \pageref{ITEM:6-}. This, using the Hom-Tensor adjunction, we can write as 
\begin{equation}\label{eqnarray:qTot}
\quot(\Chain(\Norm(\Z\Delt{\ell}, \quot(A^{\ell, \bullet}))) =\frac{ \Norm^\bullet(\Z\Delt{\ell}) \otimes \frac{A^{\ell, \bullet} \ot \Z [u]  }{( A^{\ell, \bullet} \ot \Z [u] )^{\bu>0}}}
{ \left(\Norm^\bullet(\Z\Delt{\ell}) \otimes \frac{A^{\ell, \bullet} \ot \Z [u]  }{( A^{\ell, \bullet} \ot \Z [u] )^{\bu>0}}\right)^{\bu>0}}
\end{equation}
We see that (\ref{eqnarray:Totq}) and (\ref{eqnarray:qTot}) are equal, which proves the lemma.
\end{proof}

Given a cosimplicial object in $\Chain^-$, denoted by $A$, we can apply totalization in $(\Chain^-)^{\Del}$ to it get an object in $\Chain^-$, and then apply the Dold-Kan functor to get a simplicial abelian group.  Alternatively, we can apply the Dold-Kan functor to every $A^{\ell, \bullet}$ to get a cosimplicial simplicial abelian group, and then apply totalization in $(\sAb)^\Del$ to get a simplicial abelian group.  The next lemma says that these simplicial abelian groups are weakly equivalent.

\begin{lemma}\label{lemma:TotandDK}
Let $A :\Del \rightarrow \Chain^-$ be a cosimplicial non-positively graded cochain complex.  Then there is a weak equivalence of simplicial abelian groups 
\begin{equation*}
 \Tot(\bigoplus_{\ell} \DK(A^{\ell, \bullet}))\rightarrow \DK(\Tot(\bigoplus_\ell A^{\ell,\bullet} ))
\end{equation*}
\end{lemma}
\begin{proof}
First note that the functor $\DK$ is a right adjoint, so it commutes with all limits.  Since totalization is an equalizer of two maps, we get the following equalities:
\begin{eqnarray*}
\DK(\Tot \bigoplus_\ell A^{\ell, \bullet}) &=& \DK\left(eq:\prod_{[\ell]} (A^{\ell, \bullet})^{\Delt{\ell}} \rightrightarrows \prod_{[m] \rightarrow [n]}( A^{n, \bullet})^{\Delt{m}} \right) \\
&=& eq: \prod_{[\ell]}\DK( (A^{\ell, \bullet})^{\Delt{\ell}}) \rightrightarrows \prod_{[m] \rightarrow [n]} \DK(( A^{n, \bullet})^{\Delt{m}} )
\end{eqnarray*}
By definition, see equation \eqref{EQU:DK-def}, the $n$-simplices of $\DK((A^{\ell, \bullet})^{\Delt{\ell}} )$ is the set of morphisms in $\Chain^-$ from $\Norm(\Z\Delt{n})$ to $(A^{\ell, \bullet})^{\Delt{\ell}}$.  Using the adjunctions between Hom and $\otimes$, 
we see that 
\begin{align}\label{eqnarray:DKTot}
\DK((A^{\ell, \bullet})^{\Delt{\ell}} )_n &=   {\Chain^-}(\Norm (\Z\Delt{n}), (A^{\ell, \bullet})^{\Delt{\ell}}) \\
\nonumber 
&= \Chain^-(\Norm(\Z\Delt{n}),q(Hom^\bu(\Norm(\Z\Delt{\ell}),A^{\ell,\bu}))) \\
\nonumber
& =  \Chain^- (\Norm(\Z\Delt{n} ) \otimes \Norm(\Z \Delt{\ell}), A^{\ell, \bullet}).
\end{align}

On the other hand, consider   
\begin{equation*}
\Tot (\bigoplus_\ell \DK(A^{\ell, \bullet})) = eq: \prod_{[\ell]} (\DK(A^{\ell, \bullet}) )^{\Delt{\ell}} \rightrightarrows \prod_{[m]\rightarrow[n]} (\DK(A^{n, \bullet}))^{\Delt{m}}
\end{equation*}
By definition $(\DK(A^{\ell, \bullet}))^{\Delt{\ell}}$ is a simplicial abelian group.  Its $n$-simplices, by the definition of the simplicial model category structure in example \ref{example:simplicialabeliangroups}, are equal to 
\begin{align*}
((\DK(A^{\ell, \bullet}))^{\Delt{\ell}})_n&=Map(\Z\Delta^\ell,\DK(A^{\ell,\bu}))_n
\\
&=
\sAb (\Z\Delt{n} \otimes \Z\Delt{\ell}, \DK(A^{\ell, \bullet}))\\
&= \Chain^- (\Norm(\Z (\Delt{n} \times \Delt{\ell}) ), A^{\ell, \bullet})
\end{align*}
where the last equality follows from the adjunction between $\Norm$ and $\DK$. We now use the Eilenberg-Zilber map $EZ:\Norm(\Z\Delt{n})\otimes \Norm(\Z\Delt{\ell})\to \Norm(\Z(\Delt{n}\times \Delt{\ell}))$, cf. \cite[1.6.11]{Loday},
\begin{multline}\label{EQU:EZ-map}
EZ(e_{j_0,\dots,j_p}\otimes e_{i_0,\dots, i_q})
\\
:=\sum_{(p,q)\text{-shuffles }(\mu,\nu)} \sgn(\mu,\nu) \cdot (s_{\nu_q}\dots s_{\nu_1}(e_{j_0,\dots,j_p}),s_{\mu_p}\dots s_{\mu_1}(e_{i_0,\dots, i_q})),
\end{multline}
where we used notation from example \ref{REM:NZDeltak}. We note that $EZ$ is a quasi-isomorphism with quasi-inverse the Alexander-Whitney map (cf. \cite[1.6.12]{Loday}). Thus, we get a map
\begin{equation}\label{eqnarray:EZ}
\Chain^- (\Norm(\Z (\Delt{n} \times \Delt{\ell}) ), A^{\ell, \bullet})\stackrel{(-)\circ EZ}\longrightarrow
\Chain^- (\Norm(\Z\Delt{n} ) \otimes \Norm(\Z \Delt{\ell}), A^{\ell, \bullet})   
\end{equation}
This is exactly what we had in \eqref{eqnarray:DKTot}, which completes our proof.
\end{proof}

\begin{lemma}\label{lemma:TotandForget}
Let $\mc F:\Ab^{\Del^{op}} \rightarrow \Set^{\Del^{op}}$ be the forgetful functor and let $A^{\bullet, \bullet}:\Del \rightarrow \Ab^{\Del^{op}}$ be a cosimplicial simplicial abelian group.  Then 
\begin{eqnarray*}
\mc F(\Tot(\bigoplus_\ell A^{\ell, \bullet})) \cong \Tot(\bigoplus_\ell \mc F (A^{\ell, \bullet})) 
\end{eqnarray*}
\end{lemma}

\begin{proof}
The proof proceeds similarly to the previous two lemmas, since the forgetful functor $\mc F$ is a right adjoint, just as the functors $\quot$ and $\DK$ were.
\end{proof}

Combining the previous three lemmas, we obtain the following diagram of functors
\begin{equation}\label{eqnarray:commutative}
\xymatrix{ (\Chain^+) ^\Del \ar^{\quot}[r] \ar^{\tot}[d] & (\Chain^-)^\Del \ar^{\DK}[r]\ar^{\Tot}[d] & (\Ab^{\Del^{op}})^\Del \ar^{\Tot}[d] \ar^{\mc F} [r] & (\Set^{\Del^{op}} )^{\Del} \ar^{\Tot}[d] \\ 
\Chain^+ \ar^{\quot}[r]& \Chain^-\ar^{\DK}[r]  & \Ab^{\Del^{op}} \ar^{\mc F}[r] & \Set^{\Del^{op}} }
\end{equation}
The left and right squares strictly commute, while the middle square induces a commutative square in the homotopy categories of these model categories.

We want to apply the above to the holomorphic forms on a complex manifold $M$. 
Let $M\in \CMan$ with an open cover $\mc U=\{U_i\}_{i\in I}$, and let $\CN \mc U:\Del^{op}\to \CMan$ be the \v{C}ech nerve, which is the simplicial manifold whose $k$-simplices are $\CN \mc U_k=\coprod_{i_0,\dots, i_k\in I}U_{i_0,\dots, i_k}$. Thus, applying holomorphic forms (with zero differential) gives a cosimplicial non-negatively graded cochain complex $\Ohol^\bu(\CN \mc U):\Del\stackrel {\CN\mc U}\to \CMan^{op}\stackrel{\Ohol^\bu(.)}\to \Chain^+$. Now, denote by $\Ohol^\bu$ the sheaf of holomorphic forms (with zero differential). By definition \ref{DEF:Presheaf-Cech}, there is a cosimplicial cochain complex $\Ohol^{\bu,\bu}:\Del\to \Chain^+$, $[n]\mapsto \prod_{i_0,\dots, i_n}\Ohol^\bu(U_{i_0,\dots, i_n})$. Then these two cosimplicial non-negatively graded cochain complexes coincide:

\begin{proposition}\label{PROP:Tot(Ohol)=Cech}
In the notation above, the above two cosimplicial non-negatively graded cochain complexes are isomorphic:
\begin{equation}\label{EQU:Ohol=Ohol-bullet-bullet}
\Ohol^\bu(\CN \mc U) \cong \Ohol^{\bu,\bu}.
\end{equation}

After taking the totalization, we get isomorphisms of cochain complexes
\begin{equation}\label{EQU:Tot(NU)=Cech}
\tot(\Ohol^\bu(\CN \mc U)) \cong \tot(\Ohol^{\bu,\bu})\cong \vC^\bu(\mc U, \Ohol^\bu),
\end{equation}
where the differential on $\vC^\bu(\mc U, \Ohol^\bu)$ is $\delta$ from \eqref{EQU:chdel}.

Furthermore, there is a weak equivalence of simplicial sets
\begin{equation}\label{EQU:Tot(DK(NU))=DK(Cech)}
\Tot(\OM(\CN \mc U))=\Tot(\DKSet (\Ohol^\bu(\CN \mc U)\ul))
\stackrel \sim\longrightarrow \DKSet(\vC^\bu(\mc U, \Ohol^\bu)\ul)
\end{equation}
\end{proposition}
\begin{proof}
For \eqref{EQU:Ohol=Ohol-bullet-bullet}, note that the cosimplicial non-negatively graded cochain complexes map $[n]$ to the cochain algebra $\Ohol^\bu(\coprod_{i_0,\dots, i_n\in I}U_{i_0,\dots, i_n})\cong \prod_{i_0,\dots, i_n\in I}\Ohol^\bu(U_{i_0,\dots, i_n})$ with the zero differential. Equation \eqref{EQU:Tot(NU)=Cech} follows from \eqref{EQU:Ohol=Ohol-bullet-bullet} and lemma \ref{lemma:TotandCech} applied to $\Ohol^\bu$, where the total differential from \eqref{EQU:D-on-C(U,A)} $D=\delta$ on $\vC^\bu(\mc U, \Ohol^\bu)$, since we have set the cochain differential to be zero. Finally, \eqref{EQU:Tot(DK(NU))=DK(Cech)} follows via lemmas \ref{lemma:Totandq}, \ref{lemma:TotandDK}, and \ref{lemma:TotandForget}, or, in other words, follow $\Ohol^\bu(\CN \mc U)\cong \Ohol^{\bu,\bu}$ around the diagram in \eqref{eqnarray:commutative}:
\begin{multline*}
\Tot(\DKSet( \Ohol^\bu(\CN \mc U)\ul))
\cong \Tot(\mc F(\DK( \quot(\Ohol^{\bu,\bu}))))
\cong \mc F(\Tot(\DK( \quot(\Ohol^{\bu,\bu}))))
\\
\to \mc F(\DK(\Tot (\quot(\Ohol^{\bu,\bu}))))
\cong \DKSet(\quot(\tot (\Ohol^{\bu,\bu})))
\cong\DKSet(\vC^\bu(\mc U, \Ohol^\bu)\ul).
\end{multline*}
This completes the proof.
\end{proof}

\begin{corollary}\label{COR:eval-circ-Tot(Ch)}
Using equations \eqref{EQU:TOT(CS)} and \eqref{EQU:Tot(DK(NU))=DK(Cech)}, we thus have a map 
\begin{equation}\label{EQU:eval-circ-Tot(Ch)}
\Tot(\HVB(\CN\mc U))\stackrel{\Tot(\CS(\CN\mc U))}\longrightarrow \Tot(\OM(\CN\mc U))\stackrel{\eqref{EQU:Tot(DK(NU))=DK(Cech)}}\longrightarrow  \DKSet(\vC^\bu(\mc U, \Ohol^\bu)\ul).
\end{equation}
\end{corollary}

\subsection{Computing $\Tot(\CS(\CN \mc U))$}

In equation \eqref{EQU:eval-circ-Tot(Ch)}, we obtained a map $\Tot(\HVB(\CN\mc U))\to  \DKSet(\vC^\bu(\mc U, \Ohol^\bu)\ul)$. In this section, we give an explicit description of this map. We first state a more explicit description of $n$-simplices of $\Tot(\HVB(\CN\mc U))$, extending the statement from proposition \ref{PROP:Tot(Hol(U))0}.

\begin{proposition}\label{PROP:Tot(Hol(U))n}
Let $\mc U=\{U_i\}_{i\in I}$ be an open cover of a complex manifold $M$. Then, the $n$-simplices of $\Tot(\HVB(\CN\mc U))$ are given by a choice of $n+1$ many holomorphic bundles $E^{(0)}_i\to U_i, \dots, E^{(n)}_i\to U_i$ (for each $i\in I$) together with holomorphic connections $\nabla^{(0)}_i, \dots, \nabla^{(n)}_i$, respectively, and holomorphic bundle isomorphisms $g^{(p)}_{i,j}:E^{(p)}_j|_{U_{i,j}}\to E^{(p)}_i|_{U_{i,j}}$ (not necessarily respecting the connections) satisfying the cocycle condition $g^{(p)}_{i,j}|_{U_{i,j,k}}\circ g^{(p)}_{j,k}|_{U_{i,j,k}}=g^{(p)}_{i,k}|_{U_{i,j,k}}$ on $U_{i,j,k}$, as well as $g_{i,i}^{(p)}=id_{E_i^{(p)}}$. Moreover, there are bundle isomorphisms $f_i^p:E^{(p-1)}_i\to E^{(p)}_i$ over $U_i$ (also not necessarily respecting the connections), satisfying $f^p_i|_{U_{i,j}}\circ g_{i,j}^{(p-1)}=g_{i,j}^{(p)}\circ f^p_j|_{U_{i,j}}$.
\end{proposition}
\begin{proof}
$\Tot(\HVB(\CN\mc U))$ is a simplicial subcomplex of $\prod_{[\ell]\in \Ob(\Del)} (\HVB(\CN\mc U_\ell))^{\Delt{\ell}}$, which is a simplicial set whose $n$-simplices consist of elements $\prod_{[\ell]\in \Ob(\Del)} {\lsSet}(\Delt{\ell}\times \Delt{n},\HVB(\CN\mc U_\ell))$. Thus, an $n$-simplex is given by a sequence of simplicial set maps $\Delt{\ell}\times \Delt{n}\to \HVB(\CN\mc U_\ell)$ for $\ell=0,1,2,\dots$ satisfying certain conditions. 

First, for fixed $p\in \{0,\dots, n\}$, consider the map $\rho_p:[0]\to [n], \rho_p(0)=p$. Then, an $n$-simplex of $\Tot(\HVB(\CN\mc U))$ gives rise to $0$-simplex of $\Tot(\HVB(\CN\mc U))$, via the composition $\Delt{\ell}\times \Delt{0}\stackrel {id\times \rho_p}\to \Delt{\ell}\times \Delt{n}\to \HVB(\CN\mc U_\ell)$. By proposition \ref{PROP:Tot(Hol(U))0} this $0$-simplex is given by a sequence of vector bundles $E^{(p)}_i\to U_i$ with a holomorphic connection  $\nabla^{(p)}$ and bundle maps $g^{(p)}_{i,j}$ satisfying the cocyle condition.
\[
\scalebox{.95}{
\begin{tikzpicture}[scale=1]
\node (E0) at (1.5,1.5) {$E^{(p)}_i$}; 

\node (E0') at (3,1.5) {$E^{(p)}_j|_{U_{i,j}}$};
\node (E1') at (5,1.5) {$E^{(p)}_i|_{U_{i,j}}$}; 
\draw [->] (E0') -- node[above] {$g^{(p)}_{i,j}$} (E1');

\node (E0'') at (7,0.5) {$E^{(p)}_k|_{U_{i,j,k}}$};
\node (E1'') at (8,2) {$E^{(p)}_j|_{U_{i,j,k}}$}; 
\node (E2'') at (9,0.5) {$E^{(p)}_i|_{U_{i,j,k}}$}; 
\draw [->] (E0'') -- node[left] {$g^{(p)}_{j,k}|_{U_{i,j,k}}$} (E1'');
\draw [->] (E1'') -- node[right] {$g^{(p)}_{i,j}|_{U_{i,j,k}}$} (E2'');
\draw [->] (E0'') -- node[below] {$g^{(p)}_{i,k}|_{U_{i,j,k}}$} (E2'');

\node (E0''') at (11,0.5) {$E^{(p)}_\ell|_{U_{i,j,k,\ell}}$};
\node (E1''') at (12,2) {$E^{(p)}_k|_{U_{i,j,k,\ell}}$}; 
\node (E2''') at (14,0.5) {$E^{(p)}_j|_{U_{i,j,k,\ell}}$}; 
\node (E3''') at (15,2) {$E^{(p)}_i|_{U_{i,j,k,\ell}}$}; 
\draw [->] (E0''') -- node[left] {$g^{(p)}_{k,\ell}|_{\dots}$} (E1''');
\draw [->] (E0''') -- node[below] {$g^{(p)}_{j,\ell}|_{\dots}$} (E2''');
\draw [->] (E0''') -- node[above] {$g^{(p)}_{i,\ell}|_{\dots}$} (E3''');
\draw [->] (E1''') -- node[below] {$g^{(p)}_{j,k}|_{\dots}$} (E2''');
\draw [->] (E1''') -- node[above] {$g^{(p)}_{i,k}|_{\dots}$} (E3''');
\draw [->] (E2''') -- node[right] {$g^{(p)}_{i,j}|_{\dots}$} (E3''');

\node (F) at (16.5,1.5) {$\dots$};
\end{tikzpicture}
}
\]

On the other hand, when $\ell=0$, the map $\Delt{0}\times \Delt{n}\to \HVB(\CN\mc U_0)$ is determined by its image on the maximal non-degenerate $n$-simplex of $\Delt{0}\times \Delt{n}$ in $\HVB(\CN\mc U_0)_n$, i.e., by $n$ vector bundle isomorphisms over $U_i$
$$E_i^{(0)}\stackrel {f^1_i}\to E^{(1)}_i\stackrel {f^{2}_i}\to E_i^{(2)}\stackrel {f_i^{3}}\to\dots \stackrel {f_i^{n-1}}\to E_i^{(n-1)}\stackrel {f_i^n}\to E_i^{(n)},$$
where the $E^{(p)}_i$ coincide with the ones from above, since they are the images of $\Delt{0}\times \Delt{0}\stackrel {}\to \Delt{0}\times\Delt{n}$. Now, the $g^{(p)}_{i,j}$ and $f^p_i$ commute as follows:
\[
\scalebox{.95}{
\begin{tikzpicture}[scale=1]
\node (E0) at (4,3.5) {$E^{(p-1)}_{j}|_{U_{i,j}}$};
\node (E1) at (7,3.5) {$E^{(p-1)}_{i}|_{U_{i,j}}$}; 
\draw [->] (E0) -- node[above] {$g^{(p-1)}_{i,j}$} (E1);

\node (E0') at (4,1.5) {$E^{(p)}_{j}|_{U_{i,j}}$};
\node (E1') at (7,1.5) {$E^{(p)}_{i}|_{U_{i,j}}$}; 
\draw [->] (E0') -- node[below] {$g^{(p)}_{i,j}$} (E1');

\draw [->] (E0) -- node[left] {$f^p_{j}|_{U_{i,j}}$} (E0');
\draw [->] (E0) -- node[above] {$h^p_{i,j}$} (E1');
\draw [->] (E1) -- node[right] {$f^p_{i}|_{U_{i,j}}$} (E1');
\end{tikzpicture}
}
\]
which can be seen by considering the image of two maximal non-degenerate $2$-simplices of $\Delt{1}\times\Delt{1}\stackrel{id\times \rho_p}\to\Delt{1}\times\Delt{n}\to \HVB(\CN \mc U_1)$ with $\rho_p:[1]\to [n], \rho_p(0)=p-1, \rho_p(1)=p$. The equalizer condition of the totalization shows that these two $2$-simplices have faces $f^p_i|_{U_{i,j}}, h^{p}_{i,j}, g^{(p-1)}_{i,j}$ and $g_{i,j}^{(p)}, h^p_{i,j}, f_j^p|_{U_i,j}$, respectively.

For example, in the equalizer \eqref{DEF:Tot}, the $\rho=\delta_0:[0]\to [1]$ component $\Delt{0}\times \Delt{n}\to \HVB(\CN \mc U_1)$ receives an output from $\phi$ via the component $\Delt{1}\times \Delt{n}\to \HVB(\CN \mc U_1)$, and it receives an output from $\psi$ via the component $\Delt{0}\times \Delt{n}\to \HVB(\CN \mc U_0)$, which must coincide. 
\[
\scalebox{1}{
\begin{tikzpicture}[scale=1]
\node (E0) at (0,.6) {$\lsSet(\Delt{1}\times\Delt{n},\HVB(\CN \mc U_1))$};
\node (E1) at (0,0) {$\lsSet(\Delt{0}\times\Delt{n},\HVB(\CN \mc U_0))$}; 
\node (E2) at (7,.3) {$\lsSet(\Delt{0}\times\Delt{n},\HVB(\CN \mc U_1))$};
\draw [->] (E0) -- node[above] {$\phi|_{\delta_0}$} (E2);
\draw [->] (E1) -- node[below] {$\psi|_{\delta_0}$} (E2);
\end{tikzpicture}
}
\]
Now, the image of the $1$-simplex $([1]\stackrel{\sigma_0}\to[0], [1]\stackrel{\rho_p}\to [n])\in \Delt{0}\times\Delt{n}$ is, by definition, the $1$-simplex $\coprod_i E_i^{(p-1)}\stackrel {\coprod_i f_i^p}\to \coprod_i E_i^{(p)}$ in $\HVB(\CN \mc U_0)_1$. Under $\psi$, this maps to the ($\delta_0:[0]\to [1]$)-component $\Delt{0}\times \Delt{n}\to \HVB(\CN \mc U_1)$, which maps the $1$-simplex $([1]\stackrel{\sigma_0}\to[0], [1]\stackrel{\rho_p}\to [n])$ to $\coprod_{i,j} E_j^{(p-1)}\stackrel {\coprod_{i,j} f_j^p}\to \coprod_{i,j} E_j^{(p)}$ (suitably restricted to $U_{i,j}$). On the other hand, consider the $2$-simplex $([2]\stackrel {\sigma_1}\to [1],[2]\stackrel{\lambda_p}\to [n])\in \Delt{1}\times \Delt{n}$, where $\lambda_p(0)=p-1, \lambda_p(1)=p-1, \lambda_p(2)=p$. Assume that this gets mapped to $E'_0\stackrel{g'}\to E'_1\stackrel{f'}\to E'_2$ in $\HVB(\CN \mc U_1)_2$. Note that the $0$th face of $([2]\stackrel {\sigma_1}\to [1],[2] \stackrel{ \lambda_p}\to [n])$ is in fact $([1]\stackrel{\sigma_1\circ \delta_0=\delta_0\circ \sigma_0}\to[1], [1]\stackrel{\lambda_p\circ \delta_0=\rho_p}\to [n])$, which thus gets mapped to $E'_1\stackrel {f'} \to E'_2$ in $\HVB(\CN \mc U_1)_1$. Now the map $\phi$ into the ($\delta_0:[0]\to[1]$)-component maps $\Delt{1}\times \Delt{n}\stackrel \alpha \to \HVB(\CN \mc U_1)$ to $\Delt{0}\times \Delt{n}\stackrel {\delta_0(.)\times id}\to\Delt{1}\times \Delt{n}\stackrel \alpha\to \HVB(\CN \mc U_1)$, so that it maps $([1]\stackrel{\sigma_0}\to[0], [1]\stackrel{\rho_p}\to [n])\mapsto ([1]\stackrel{\delta_0\circ \sigma_0}\to[1], [1]\stackrel{\rho_p}\to [n])\mapsto E'_1\stackrel{f'}\to E'_2$. Since the images of $\phi$ and $\psi$ coincide, we obtain that the $0$th face $E'_1\stackrel{f'}\to E'_2$ of the above $2$-simplex equals $\coprod_{i,j} E_i^{(p-1)}\stackrel {\coprod_{i,j} f_i^p}\to \coprod_i E_i^{(p)}$. A similar argument shows that $E'_0\stackrel{g'}\to E'_1$ equals $\coprod_{i,j} E_j^{(p-1)}\stackrel {\coprod_{i,j} g_{i,j}^{(p-1)}}\to \coprod_{i,j} E_i^{(p-1)}$, etc.

This shows that $f^p_i|_{U_{i,j}}\circ g_{i,j}^{(p-1)}=h^p_{i,j}=g_{i,j}^{(p)}\circ f^p_j|_{U_{i,j}}$ as claimed.

Finally, we note that there are no higher relations, since all higher cocycle conditions follow from the ones on the $1$-simplices (cf. figure \ref{fig:E1-bundle}).
\begin{figure}[h]
\[
\scalebox{.95}{
\begin{tikzpicture}[scale=1]
\node (E0) at (.5,1.5) {$E^{(0)}_i$}; 

\node (E0') at (3,1.5) {$E^{(0)}_j|_{U_{i,j}}$};
\node (E1') at (5,1.5) {$E^{(0)}_i|_{U_{i,j}}$}; 
\draw [->] (E0') -- node[above] {$g^{(0)}_{i,j}$} (E1');

\node (E0'') at (8,1) {$E^{(0)}_k|_{U_{i,j,k}}$};
\node (E1'') at (10.5,2.5) {$E^{(0)}_j|_{U_{i,j,k}}$}; 
\node (E2'') at (13,1) {$E^{(0)}_i|_{U_{i,j,k}}$}; 
\draw [->] (E0'') -- node[left] {$g^{(0)}_{j,k}|_{U_{i,j,k}}$} (E1'');
\draw [->] (E1'') -- node[right] {$g^{(0)}_{i,j}|_{U_{i,j,k}}$} (E2'');
\draw [->] (E0'') -- node[right] {$g^{(0)}_{i,k}|_{U_{i,j,k}}$} (E2'');

\node (F0) at (.5,-1.5) {$E^{(1)}_i$}; 

\node (F0') at (3,-1.5) {$E^{(1)}_j|_{U_{i,j}}$};
\node (F1') at (5,-1.5) {$E^{(1)}_i|_{U_{i,j}}$}; 
\draw [->] (F0') -- node[below] {$g^{(1)}_{i,j}$} (F1');

\node (F0'') at (8,-2) {$E^{(1)}_k|_{U_{i,j,k}}$};
\node (F1'') at (10.5,-.5) {$E^{(1)}_j|_{U_{i,j,k}}$}; 
\node (F2'') at (13,-2) {$E^{(1)}_i|_{U_{i,j,k}}$}; 
\draw [->] (F0'') -- node[left] {$g^{(1)}_{j,k}|_{U_{i,j,k}}$} (F1'');
\draw [->] (F1'') -- node[right] {$g^{(1)}_{i,j}|_{U_{i,j,k}}$} (F2'');
\draw [->] (F0'') -- node[below] {$g^{(1)}_{i,k}|_{U_{i,j,k}}$} (F2'');

\draw [->] (E0) -- node[left] {$f^1_i$} (F0);
\draw [->] (E0') -- node[left] {$f^1_j|_{U_{i,j}}$} (F0');
\draw [->] (E1') -- node[right] {$f^1_i|_{U_{i,j}}$} (F1');
\draw [->] (E0'') -- node[left] {$f^1_k|_{U_{i,j,k}}$} (F0'');
\draw [->] (E1'') -- node[below] {$f^1_j|_{U_{i,j,k}}$} (F1'');
\draw [->] (E2'') -- node[right] {$f^1_i|_{U_{i,j,k}}$} (F2'');

\node (F) at (15,0) {$\dots$};
\end{tikzpicture}
}
\]
\caption{Higher relations in $\Tot(\HVB(\CN \mc U))_n$ are induced by the conditions on the $1$-simplicies} \label{fig:E1-bundle}
\end{figure}

This completes the proof of the proposition.
\end{proof}

We now use the data from the previous proposition to express $\Tot(\CS(\CN \mc U))$.
\begin{proposition}\label{PROP:Tot-CS-n}
Using the description from proposition \ref{PROP:Tot(Hol(U))n}, the map of $n$-simplices 
$\Tot(\CS(\CN \mc U))_n:\Tot(\HVB(\CN\mc U))_n\to \DKSet(\vC^\bu(\mc U, \Ohol^\bu)\ul)_n$ is given by mapping the generator $e_{j_0,\dots, j_p}$ for $0\leq j_0<\dots<j_p\leq n$ of $N(\Z\Delt{n})$ to the cochain $c^{(j_0,\dots,j_p)}\in \vC^\bu(\mc U, \Ohol^\bu)\ul$ defined as
\begin{multline*}
(c^{(j_0,\dots,j_p)})_{i_0,\dots, i_\ell}= u^{\ell+p}\cdot (-1)^{\frac{p(p-1)}{2}}\cdot \sum_{0\leq \stp_1\leq \stp_2\leq \dots\leq \stp_p\leq \ell} (-1)^{\stp_1+\dots +\stp_p} \cdot \tr\Big[ 
\\
\Big(
(g^{(j_p)}_{i_{\ell},i_{\ell-1}} 
\dots g^{(j_p)}_{i_{\stp_p+1},i_{\stp_p}})
f^{(j_p,j_{p-1})}_{i_{\stp_p}}  
\dots f^{(j_2,j_1)}_{i_{\stp_2}}
(g^{(j_1)}_{i_{\stp_2},i_{\stp_2-1}}\dots g^{(j_1)}_{i_{\stp_1+1},i_{\stp_1}})f^{(j_1,j_0)}_{i_{\stp_{1}}}
(g^{(j_0)}_{i_{\stp_1},i_{\stp_1-1}}\dots g^{(j_0)}_{i_2,i_1} g^{(j_0)}_{i_1,i_0})\Big)^{-1}
\\ 
\cdot
\nabla(g^{(j_p)}_{i_{\ell},i_{\ell-1}})\dots  \nabla(g^{(j_p)}_{i_{\stp_p+1},i_{\stp_p}})
 \nabla(f^{(j_p,j_{p-1})}_{i_{\stp_p}} )
 \dots\dots
 \nabla(f^{(j_1,j_0)}_{i_{\stp_{1}}})
 \nabla(g^{(j_0)}_{i_{\stp_1},i_{\stp_1-1}})\dots \nabla(g^{(j_0)}_{i_2,i_1} )\nabla(g^{(j_0)}_{i_1,i_0})
\Big],
\end{multline*}
where $f^{(b,a)}_i=f^{b}_i\circ\ldots \circ f^{a+1}_i:E_i^{(a)}\to E_i^{(b)}$, $f_i^{(a,a)}=id_{E_i^{(a)}}$ appears precisely at the position $\stp_1,\dots, \stp_p$, $\nabla$ is the induced connection on the appropriate $Hom(E^{(\cdot)}_\cdot, E^{(\cdot)}_\cdot)$, and everything is suitably restricted to $U_{i_0,\dots, i_\ell}$.
\end{proposition}
\begin{proof}
We follow the sequence of maps (cf. \eqref{EQU:Tot(NU)=Cech} and \eqref{EQU:Tot(DK(NU))=DK(Cech)})
\begin{align*}
\Tot(\HVB(\CN\mc U))_n 
 \stackrel {\Tot(\CS(\CN \mc U))_n}\longrightarrow & \Tot(\DKSet(\Ohol^\bu(\CN \mc U)\ul))_n
\\
\stackrel {\,\,\,\,\text{lemma }\ref{lemma:TotandDK}\,\,\,\,}\longrightarrow & \DKSet(\Tot(\Ohol^\bu(\CN \mc U)\ul))_n 
\\
\stackrel {\,\,\text{lemmas }\ref{lemma:Totandq}, \ref{lemma:TotandCech}\,\,}\longrightarrow & \DKSet(\vC^\bu(\mc U, \Ohol^\bu)\ul)_n 
\end{align*}

An $n$-simplex in the simplicial set $\Tot(\HVB(\CN\mc U))$ consists of a sequence of $n$-simplices in $\Nerve(\HVBnabla(\CN \mc U_\ell))^{\Delt{\ell}}$ for $\ell=0,1,2,\dots$ (where $\CN \mc U_\ell=\coprod_{i_0,\dots, i_\ell} U_{i_0,\dots,i_\ell}$), i.e., in $\lsSet(\Delt{n}\times\Delt{\ell},\Nerve(\HVBnabla(\CN \mc U_\ell)))$. In particular, using the notation from example \ref{REM:NZDeltak} and proposition \ref{PROP:Tot(Hol(U))n}, the $p+q$ simplex $(s_{\nu_{q}}\dots s_{\nu_{1}}(e_{j_0,\dots,j_p}), s_{\mu_{p}}\dots s_{\mu_{1}}(e_{i_0,\dots,i_q}))$ in $\Delt{n}\times \Delt{\ell}$ gets mapped to compositions of $g^{(r)}_{ij} $ and $f^r_i$ restricted to $\CN \mc U_\ell$.
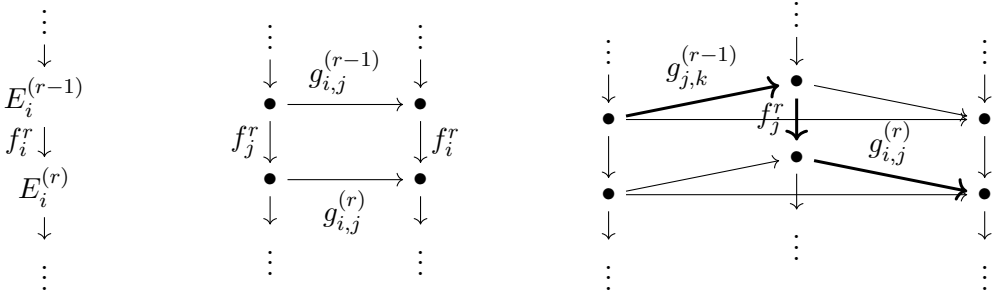
\begin{equation*}
\scalebox{1}{
\begin{tikzpicture}[scale=1]
\node (E0) at (.5,0.1) {$ E^{(r-1)}_i$};
\node (E0') at (3.5,0) {$\bu$};
\node (E1') at (5.5,0) {$\bu$}; 
\draw [->] (E0') -- node[above] {$ g^{(r-1)}_{i,j}$} (E1');
\node (E0'') at (8,-.2) {$\bu$};
\node (E1'') at (10.5,0.3) {$\bu$}; 
\node (E2'') at (13,-.2) {$\bu$}; 
\draw [line width=0.4mm,->] (E0'') -- node[above] {$ g^{(r-1)}_{j,k}$} (E1'');
\draw [->] (E1'') -- node[right] {} (E2'');
\draw [->] (E0'') -- node[right] {} (E2'');
\node (F0) at (.5,-1.1) {$ E^{(r)}_i$}; 
\node (F0') at (3.5,-1) {$\bu$};
\node (F1') at (5.5,-1) {$\bu$}; 
\draw [->] (F0') -- node[below] {$ g^{(r)}_{i,j}$} (F1');
\node (F0'') at (8,-1.2) {$\bu$};
\node (F1'') at (10.5,-.7) {$\bu$}; 
\node (F2'') at (13,-1.2) {$\bu$}; 
\draw [->] (F0'') -- node[left] {} (F1'');
\draw [line width=0.4mm,->] (F1'') -- node[above] {$ g^{(r)}_{i,j}$} (F2'');
\draw [->] (F0'') -- node[below] {} (F2'');
\draw [->] (E0) -- node[left] {$ f^r_i$} (F0);
\draw [->] (E0') -- node[left] {$ f^r_j$} (F0');
\draw [->] (E1') -- node[right] {$ f^r_i$} (F1');
\draw [->] (E0'') -- node[left] {} (F0'');
\draw [line width=0.4mm,->] (E1'') -- node[left] {$ f^r_j$} (F1'');
\draw [->] (E2'') -- node[right] {} (F2'');
\node (E0-) at (.5,1.2) {$\vdots$};  \draw [->] (E0-) -- node[above] {} (E0);
\node (E0+) at (.5,-2.2) {$\vdots$};  \draw [->] (F0) -- node[above] {} (E0+);
\node (E0'-) at (3.5,1) {$\vdots$};  \draw [->] (E0'-) -- node[above] {} (E0');
\node (E0'+) at (3.5,-2) {$\vdots$};  \draw [->] (F0') -- node[above] {} (E0'+);
\node (E1'-) at (5.5,1) {$\vdots$};  \draw [->] (E1'-) -- node[above] {} (E1');
\node (E1'+) at (5.5,-2) {$\vdots$};  \draw [->] (F1') -- node[above] {} (E1'+);
\node (E0''-) at (8,.8) {$\vdots$};  \draw [->] (E0''-) -- node[above] {} (E0'');
\node (E0''+) at (8,-2.2) {$\vdots$};  \draw [->] (F0'') -- node[above] {} (E0''+);
\node (E1''-) at (10.5,1.3) {$\vdots$};  \draw [->] (E1''-) -- node[above] {} (E1'');
\node (E1''+) at (10.5,-1.8) {$\vdots$};  \draw [->] (F1'') -- node[above] {} (E1''+);
\node (E2''-) at (13,.8) {$\vdots$};  \draw [->] (E2''-) -- node[above] {} (E2'');
\node (E2''+) at (13,-2.2) {$\vdots$};  \draw [->] (F2'') -- node[above] {} (E2''+);
\end{tikzpicture}
}
\end{equation*}
This map is exhibited in the following few examples:
\[
\begin{tabular}{l|l|l}
 $(e_r,e_0)\mapsto \coprod E_i^{(r)}$  & $(s_0 e_{r-1,r},s_1 e_{01})$ & $(s_2 s_0 e_{r-1,r},s_1 e_{0,1,2})$ \\ 
 $(e_{r-1,r},s_0 e_0)\mapsto \coprod f_i^r$ & \hspace{\tabspac} \hspace{.4cm}$\mapsto (\bu\stackrel{\coprod g^{(r-1)}_{i,j}}\to\bu\stackrel{\coprod f^r_i}\to\bu)$ & \hspace{\tabspac} $\mapsto(\bu\stackrel{\coprod g^{(r-1)}_{j,k}}\to \bu\stackrel{\coprod f^r_j}\to \bu\stackrel{\coprod g^{(r)}_{i,j}}\to \bu)$ \\  
  & $(s_1 e_{r-1,r},s_0 e_{01})$ &   \\
 & \hspace{\tabspac}\hspace{.4cm} $\mapsto (\bu\stackrel{\coprod f^r_j}\to\bu\stackrel{\coprod g^{(r)}_{i,j}}\to\bu)$ &  
\end{tabular}
\]

Now, applying $\Tot(\CS(\CN \mc U))$ means to apply $\CS$ to each simplex in the nerve $\Nerve(\HVBnabla(\CN \mc U_\ell))$, i.e., we apply \eqref{EQU:tr(f-1nablaf-nablag)} to composable morphisms. In the above examples, we thus obtain (on $U_{i}$, $U_{i,j}$, and $U_{i,j,k}$, respectively):
\[
\begin{tabular}{l|l|l}
 $(e_r,e_0)\mapsto \dim(E^{(r)})$  
 & $ (s_0 e_{r-1,r},s_1 e_{01})\mapsto u^2 \cdot \tr[$ 
 & $(s_2 s_0 e_{r-1,r},s_1 e_{0,1,2})$ 
 \\ 
 $(e_{r-1,r},s_0 e_0)$ 
 & \hspace{\tabspac} $ (f^r_i g^{(r-1)}_{i,j})^{-1}\nabla f^r_i\nabla g^{(r-1)}_{i,j}]$ 
 & \hspace{\tabspac} $\mapsto u^3\cdot \tr[ (g^{(r)}_{i,j} f^r_j g^{(r-1)}_{j,k})^{-1}$
 \\  
 \hspace{\tabspac} $\mapsto u\cdot \tr[ (f_i^r)^{-1}\nabla f_i^r]$ 
 & $(s_1 e_{r-1,r},s_0 e_{01})\mapsto u^2 \cdot \tr[$ 
 &  \hspace{\tabspac}\hspace{1.6cm}$ \nabla g^{(r)}_{i,j} \nabla f^r_j \nabla g^{(r-1)}_{j,k}]$
 \\
 & \hspace{\tabspac} $(g^{(r)}_{i,j} f^r_j)^{-1}\nabla g^{(r)}_{i,j} \nabla f^r_j]$ 
 &  
\end{tabular}
\]

Next, by lemma \ref{lemma:TotandDK}, we map this into $\DKSet(\Tot\Ohol^\bu(\CN \mc U)\ul)$, whose $n$-simplices are given by cochain maps in $ \Chain^- (\Norm(\Z\Delt{n} ) \otimes \Norm(\Z \Delt{\ell}), \Ohol^\bu(\CN \mc U_\ell)\ul)$ for $\ell=0,1,2,\dots$; cf. equation \eqref{eqnarray:DKTot}. We obtain these cochain maps by applying the Eilenberg-Zilber map \eqref{EQU:EZ-map}. More precisely, to a generator $e_{j_0,\dots, j_p}\otimes e_{i_0,\dots, i_q}$ of $\Norm(\Z\Delt{n} ) \otimes \Norm(\Z \Delt{\ell})$ we assign the sum over all $(p,q)$-shuffles.
\[
\begin{tabular}{l|l|l}
 $e_r\otimes e_0\mapsto \dim(E^{(r)})$  
 & $  e_{r-1,r}\otimes e_{01}\mapsto u^2 \cdot \{ $ 
 & $e_{r-1,r}\otimes e_{0,1,2}$ 
 \\ 
 $e_{r-1,r}\otimes e_0$ 
 & \hspace{\tabspac}\hspace{-0.3cm} $ -\tr[(f^r_i g^{(r-1)}_{i,j})^{-1}\nabla f^r_i\nabla g^{(r-1)}_{i,j}]$ 
 & \hspace{\tabspac}\hspace{-0.3cm} $\mapsto u^3\cdot \{- \tr[ (g^{(r)}_{i,j} f^r_j g^{(r-1)}_{j,k})^{-1}$
 \\  
 \hspace{\tabspac}\hspace{-0.3cm} $\mapsto u\cdot \tr[ (f_i^r)^{-1}\nabla f_i^r]$ 
 & \hspace{\tabspac}\hspace{-0.3cm}  $+ \tr[(g^{(r)}_{i,j} f^r_j)^{-1}\nabla g^{(r)}_{i,j} \nabla f^r_j]\}$ 
 &   \hspace{\tabspac}\hspace{1.4cm}$ \nabla g^{(r)}_{i,j} \nabla f^r_j \nabla g^{(r-1)}_{j,k}]$
 \\
 & 
 &  \hspace{\tabspac}$\pm$ two other $(1,2)$-shuffles$\}$
\end{tabular}
\]
Thus, $e_{j_0,\dots,j_p}\otimes e_{i_0,\dots, i_q}$ (for $p+q>0$) gets mapped
\begin{equation}\label{EQU:ej..xei...-->sum-tr}
e_{j_0,\dots,j_p}\otimes e_{i_0,\dots, i_q}\mapsto u^{p+q}\cdot \sum_{(p,q)\text{-shuffles }(\mu,\nu)} \sgn(\mu,\nu) \cdot \tr[((h_{p+q-1}\dots h_0)^{-1}\nabla h_{p+q-1}\dots \nabla h_0)],
\end{equation}
where for $m\in\{1,\dots, p\}$ the $h_{\mu_m}$ are ``vertical maps'' $f^{j_{m+1}}_i\circ\dots \circ f^{j_{m}+1}_i=f^{(j_{m+1},j_{m})}_{i}$ for some $i$, while all other $h_{\kappa}$ are ``horizontal maps'' $g^{(j_m)}_{i,i'}$ for some $m, i, i'$. 

Now, applying lemma \ref{lemma:Totandq} (and in particular lemma \ref{lemma:TotComplex}) shows that we only use the highest non-degenerate generator $e_{i_0,\dots, i_\ell}\in \Norm(\Z\Delt{\ell})_{-\ell}$ to map to $\Ohol^\bu(\CN \mc U_\ell)$ (since the images of the lower generators $e_{i_0,\dots, i_q}$ of $\Norm(\Z\Delt{\ell})$ are induced via the equalizer condition from $\Norm(\Z\Delt{q})$; see lemma \ref{lemma:TotComplex}). We thus land in
$$\DKSet(\tot(\Ohol^\bu(\CN \mc U))\ul)_n\cong \Chain^-\Big(\Norm(\Z\Delt{n}),\bigoplus_{\ell\geq 0} \Ohol^\bu\Big(\coprod_{i_0,\dots, i_\ell} U_{i_0,\dots, i_\ell}\Big)\ul\Big),$$
where a generator $e_{j_0,\dots, j_p}$ of $\Norm(\Z\Delt{n})$ maps to the $\ell$-component, by taking the image of $ e_{j_0,\dots, j_p}\otimes e_{i_0,\dots, i_\ell}$ under \eqref{EQU:ej..xei...-->sum-tr}, i.e.,
\begin{equation*}
e_{j_0,\dots,j_p}\mapsto u^{p+\ell}\cdot\sum_{(p,\ell)\text{-shuffles }(\mu,\nu)} \sgn(\mu,\nu) \cdot \tr[((h_{p+\ell-1}\dots h_0)^{-1}\nabla h_{p+\ell-1}\dots \nabla h_0)],
\end{equation*}
Now, for a given sequence of indices $i_0,\dots, i_\ell$ and a $(p,\ell)$-shuffle $(\mu,\nu)$, setting $0\leq \stp_1:=\mu_1\leq \stp_2:=\mu_2-1\leq \dots \leq \stp_p:=\mu_p-p+1\leq \ell$, it follows that the ``vertical maps'' are precisely at $h_{\mu_m}=f^{(j_m,j_{m-1})}_{i_{\stp_m}}$, while the ``horizontal maps'' are all other $h_\kappa=g_{i_\tau,i_{\tau-1}}^{(j_m)}$ for appropriate $m, \tau$ (cf. figure \ref{fig:IkIndices}). Note from \eqref{EQU:sgn(mu,nu)}, that the sign above is precisely $\sgn(\mu,\nu)=(-1)^{\stp_1+\dots+\stp_p}$. A final sign comes from the isomorphism $\tot(\Ohol^\bu(\CN \mc U))\ul\to\vC^\bu(\mc U, \Ohol^\bu)\ul$ as in lemma \ref{lemma:TotandCech}, where we multiply by a sign $(-1)^{\frac{(-p)\cdot ((-p)+1)}{2}}=(-1)^{\frac{p(p-1)}{2}}$, since the degree $|e_{j_0,\dots, j_p}|=-p$.

This shows that we get precisely the terms described in the proposition, and thus completes the proof.
\end{proof}

In particular, for $0$-simplices we have the following interpretation.
\begin{remark}
Since for any non-positively graded cochain complex $C\in \Chain^-$,  $\DKSet(C)_\bu$ is a simplicial set, whose $0$-simplices has as its underlying set $\DKSet(C)_0=C^0$, we see that
\[
\DKSet(\vC^\bu(\mc U, \Ohol^\bu)\ul)_0=(\vC^\bu(\mc U, \Ohol^\bu)[u])^0\cong \Big(\bigoplus_{\ell\geq 0} \vC^\ell(\mc U,\Ohol^\bu)\Big)^{even}
\]
Given a $0$-simplex of $\Tot(\HVB(\CN\mc U))$ via the data from proposition \ref{PROP:Tot(Hol(U))0}, this thus maps under \eqref{EQU:eval-circ-Tot(Ch)} to the \v{C}ech-de Rham forms $c\in \bigoplus_{\ell} \vC^\ell(\mc U,\Ohol^\ell)$, with
\[
c_{i_0,\dots, i_\ell}=\tr\Big((g_{i_\ell,i_0})^{-1}\nabla(g_{i_\ell,i_{\ell-1}})\dots \nabla(g_{i_1,i_0})\Big).
\]
These are, in fact, the classes that were given by O'Brian, Toledo, and Tong for the Chern character, cf. \cite[p.~244]{OTT1}. Recall that the cohomology of $\vC^\bu(\mc U,\Ohol^\bu)$ with the \v{C}ech differential (and zero as internal differential) is, by definition, the Hodge cohomology $H^\bu_{Hodge}(M)$ of $M\in \Ob(\CMan)$, i.e., that $H^\bu_{Hodge}(M):=H^\bu(\vC^\bu(\mc U,\Ohol^\bu),\delta)$.
\end{remark}

\section{Restricting to product bundles with connection $\partial$}\label{SEC:product-bundles}

In the previous section, we gave a map $\Tot(\CS(\CN\mc U)):\Tot(\HVB(\CN\mc U))\to \Tot(\OM(\CN\mc U))\stackrel \sim \to  \DKSet(\vC^\bu(\mc U, \Ohol^\bu)\ul)$. In this section, we define a variant of this maps on a new domain (namely $\CMan^{\Del^{op}}(\CN \mc U^{[\bu]},BG)$), which is capable of encoding any holomorphic vector bundle in some sense (see remark \ref{REM:compare-CMan(NU[.],BG)-Tot(HVB(NU))}), and we produce a commutative diagram
\begin{equation}\label{EQU:alternative-Tot(CS)}
\xymatrix{ 
\CMan^{\Del^{op}}(\CN \mc U^{[\bu]},BG) \ar^{}[r] \ar^{}[d] & \Tot(\HVB(\CN\mc U)) \ar^{\Tot(\CS(\CN\mc U))}[d]  \\ \DKSet(\vC^\bu(\mc U, \Ohol^\bu)\ul) & \Tot(\OM(\CN\mc U)) \ar^{}[l] }
\end{equation}
which represents the map $\Tot(\CS(\CN\mc U))$ on $\CMan^{\Del^{op}}(\CN \mc U^{[\bu]},BG)$.

\subsection{A sub-simplicial presheaf of $\HVB$}

In this section, we define the top horizontal map of \eqref{EQU:alternative-Tot(CS)}. We start by defining the cosimplicial simplicial manifold $\CN \mc U^{[\bu]}$.
\begin{definition}\label{DEF:U^[n]}
Let $\mc U=\{U_i\}_{i\in I}$ be a cover $\mc U\in \Cov_M$; see definition \ref{DEF:Cov+NU}. We first define the cosimplicial cover $\mc U^{[\bu]}:\Del\to \Cov_M$. For fixed $n$, define the the index set $I^{[n]}:=\{(i,j):i\in I, 0\leq j\leq n\}$. For convenience we will use the notation $i^{(j)}=(i,j)$ for the indices in $I^{[n]}$. Then define the cover $\mc U^{[n]}:=\{U_{i^{(j)}}\}_{i^{(j)}\in I^{[n]}}$ by letting the open set $U_{i^{(j)}}:=\al^{[n]}(i^{(j)}):=U_i$, where $\al^{[n]}:I^{[n]}\to \Open_M$ determines the cover as in definition \ref{DEF:Cov+NU}. In other words, $\mc U^{[n]}$ is obtained by taking $n+1$ many copies of the original cover $\mc U$. We can make this into a cosimplicial cover by assigning to a morphism $\rho:[n]\to [m]$ in $\Del$ the cover morphism $\mc U^{[\bu]}(\rho)\in \Cov_M(\mc U^{[n]},\mc U^{[m]})$ given by $f_\rho:I^{[n]}\to I^{[m]}, f_\rho(i^{(j)})= i^{(\rho(j))}$, for which clearly $\al^{[m]}(f_\rho(i^{(j)}))=U_i=\al^{[n]}(i^{(j)})$. Note that $\mc U^{[\bu]}(\rho\circ \rho')=\mc U^{[\bu]}(\rho)\circ \mc U^{[\bu]}(\rho')$, so that we obtain the claimed cosimplicial cover $\mc U^{[\bu]}:\Del\to \Cov_M$. 

Now, composing $\mc U^{[\bu]}:\Del\to \Cov_M$ with the \v{C}ech nerve $\CN:\Cov_M\to \CMan^{\Del^{op}}$ from defintiion \ref{DEF:Cov+NU} yields the cosimplicial simplicial complex manifold $\CN\mc  U^{[\bu]}:\Del\to  \CMan^{\Del^{op}}$.
\end{definition}

The next proposition gives a more conceptual way of thinking about $\CN \mc U^{[\bu]}$.
\begin{proposition}
There is a functor $\F:  \CMan^{\Del^{op}} \to  \left(\CMan^{\Del^{op}}\right)^{\Del}$ such that $\F(\CN \mc U)=\CN \mc U^{[\bu]}$.
\begin{proof}
Consider an object $X=X_\bu \in \CMan^{\Del^{op}}$ which assigns to each $[\ell] \in \Del$ a complex manifold $X_{\ell}$.  Then $\F(X)$ is a functor $\F(X)=\F(X)^\bu:\Del \to  \CMan^{\Del^{op}}, [n]\mapsto \F(X)^n$, where $\F(X)^n=\F(X)^n_\bu:\Del^{op}\to \CMan, [\ell]\mapsto \F(X)^n_\ell$ is defined to be
\[ \F(X)^n_\ell := \coprod\limits_{\sigma \in \Set([\ell], [n])} X_{\ell}. \] 
Here, $\Set([\ell], [n])$ denotes \emph{all} set maps from $[\ell]$ to $[n]$. For a morphism $\alpha \in \Del([k], [\ell])$, we define $\F(X)^n(\alpha): \F(X)^n_\ell\to \F(X)^n_k$ to be
\[ \begin{array}{rcl}
\coprod\limits_{\sigma \in \Set([\ell], [n])} X_{\ell} & \xrightarrow{ \F(X)^n (\alpha)}& \coprod\limits_{\tau \in \Set([k], [n])} X_{k} \\
& & \\
\left(\sigma  \in \Set([\ell], [n]) ,  x \in X_{\ell}\right) & \mapsto & \left(\sigma \circ \alpha   \in \Set([k], [n] ),  X(\alpha)(x) \in X_{k}\right).
\end{array}\]
With this definition, $\F(X)^n$ becomes a simplicial manifold.

Next, we show that  $\F(X)^\bu$ is indeed a functor $\F(X)^\bu: \Del \to \CMan^{\Del^{op}}$. In fact, for a morphism $\beta \in \Del( \lbrack n \rbrack,  \lbrack m \rbrack)$, define the natural transformation $\F(X)(\beta): \F(X)^n\to \F(X)^m$ of functors $\Del^{op}\to \CMan$ by the sequence of maps $\F(X)(\beta)_\ell: \F(X)^n_\ell\to \F(X)^m_\ell$,
\[ \begin{array}{rcl}%
\coprod\limits_{\sigma \in \Set([\ell], [n])} X_{\ell} & \xrightarrow{ \F(X)(\beta)_{\ell}}&  \coprod\limits_{\tau \in \Set([\ell], [m])} X_{\ell} \\
& & \\
\left( \sigma  \in \Set([\ell], [n]), x \in X_{\ell} \right) & \mapsto & \left( \beta \circ \sigma  \in \Set([\ell], [m] ),  x \in X_{\ell}\right).
\end{array}\]
Since the composition $\F(X)^n_\ell\stackrel {\F(X)^n(\al)}\longrightarrow \F(X)^n_k \stackrel {\F(X)(\beta)_k}\longrightarrow \F(X)^m_k$ is equal to the composition $\F(X)^n_\ell\stackrel {\F(X)(\beta)_\ell}\longrightarrow \F(X)^m_\ell \stackrel {\F(X)^m(\alpha)}\longrightarrow \F(X)^m_k$, this shows that $\F(X)(\beta)$ is indeed a natural transformation, and thus $\F(X):\Del \to \CMan^{\Del^{op}}$ is a functor.

Now, to see that we have a functor $\F:  \CMan^{\Del^{op}} \to  \left(\CMan^{\Del^{op}}\right)^{\Del}$, we must assign to a natural transformation $\phi: X \to Y$ of simplicial manifolds $X, Y \in \Ob \left( \CMan^{\Del^{op}}  \right)$, a natural transformation $\F(\phi): \F(X) \to \F(Y)$.  In detail, $\F(\phi)^n_\ell:\F(X)^n_\ell\to \F(Y)^n_\ell$ is defined by $\left(  \sigma  \in \Set([\ell], [n]),x \in X_{\ell} \right)\mapsto \left( \sigma  \in \Set([\ell], [n]),\phi_{\ell} (x) \in Y_{\ell} \right)$, which makes $\F(\phi)^n:\F(X)^n\to \F(Y)^n$ into a natural transformation, since $\F(X)^n_\ell\stackrel {\F(X)^n(\al)}\longrightarrow \F(X)^n_k \stackrel {\F(\phi)^n_k}\longrightarrow \F(Y)^n_k$ equals $\F(X)^n_\ell\stackrel {\F(\phi)^n_\ell}\longrightarrow \F(Y)^n_\ell \stackrel {\F(Y)^n(\alpha)}\longrightarrow \F(Y)^n_k$, and it gives an equality of the composed natural transformations $\F(X)^n\stackrel{\F(X)(\beta)}\longrightarrow \F(X)^m\stackrel{\F(\phi)^m}\longrightarrow \F(Y)^m$ and $\F(X)^n\stackrel{\F(\phi)^n}\longrightarrow \F(Y)^n\stackrel{\F(Y)(\beta)}\longrightarrow \F(Y)^m$, which can be seen by applying it to an object $[\ell]\in \Del^{op}$.

Finally, to prove the stated condition, note that
\begin{multline*}
 \F\left( \CN \mc U \right)^n_\ell 
 = \coprod\limits_{\sigma\in \Set( [\ell] , [n])} \CN \mc U_\ell
 =\coprod\limits_{\sigma\in \Set( [\ell] , [n])} \left( \coprod\limits_{i_0, \ldots, i_\ell \in I} U_{i_0, \ldots, i_\ell} \right)
 \\
  =  \coprod\limits_{\sigma\in \Set( [\ell] , [n])} \left( \coprod\limits_{i_0^{(\sigma(0))}, \ldots, i_\ell^{(\sigma(\ell))}\in I^{[n]}} U_{i_0^{(\sigma(0))}, \ldots, i_\ell^{(\sigma(\ell))}}\right)  
  = \coprod\limits_{i_0^{(j_0)}, \ldots, i_\ell^{(j_\ell)}\in I^{[n]}} U_{i_0^{(j_0)}, \ldots, i_\ell^{(j_\ell)}}
  =  \CN \mc U^{[n]}_\ell.
  \end{multline*}
Furthermore, the action of $\al\in \Del([k], [\ell])$ comes from mapping $U_{i_0,\dots,i_\ell}\to U_{i'_0,\dots, i'_k}$ as stated in definition \ref{DEF:Cov+NU}, while the action of $\beta\in \Del([n], [m])$ comes from mapping $U_{i_0^{(j_0)},\dots,i_\ell^{(j_\ell)}}\to U_{ i_0^{( \beta(j_0))},\dots,i_\ell^{(\beta(j_\ell))}}$ as described in definition \ref{DEF:U^[n]}. Thus, this yields the stated result, i.e., that $\F(\CN \mc U)^\bu = \CN \mc U^{[\bullet]}$.
\end{proof}
\end{proposition}

Next, we define the simplicial manifold $BG$.
\begin{definition}\label{DEF:BG}
Let $G$ be a complex Lie group. We define a simplicial complex manifold $BG\in \Ob(\CMan^{\Del^{op}})$ (also denoted by $[*/G]$) by setting the $n$-simplices to be $BG_n=G^{\times n}$, i.e., we have $BG_0=\{*\}, BG_1=G, BG_2=G\times G, BG_3=G\times G\times G, \dots$. The face maps $d_j:G^{\times n}\to G^{\times (n-1)}$ for $0<j<n$ are $d_j(g_1,\dots, g_n)=(g_1,\dots, g_j\cdot g_{j+1},\dots, g_n)$, while $d_0(g_1,\dots, g_n)=(g_2,\dots, g_n)$ and $d_n(g_1,\dots, g_n)=(g_1,\dots, g_{n-1})$. The degeneracies $s_j:G^{\times (n-1)}\to G^{\times n}$ are given by $s_j(g_1,\dots, g_{n-1})=(g_1,\dots, g_j,1,g_{j+1},\dots, g_{n-1})$, where $0\leq j\leq n-1$.
\end{definition}
In the following, we will be mainly interested in the case $G=GL(n, \C)$. 

Since $BG$ and $\CN\mc  U^{[n]}$ (for fixed $n$) are simplicial manifolds, we can consider the set of morphisms between these simplicial manifolds, i.e., $\CMan^{\Del^{op}}(\CN\mc  U^{[n]},BG)$. Now, varying $n$, this becomes a simplicial set $\CMan^{\Del^{op}}(\CN\mc  U^{[\bu]},BG)\in \sSet$ by setting the $n$-simplices to be $\CMan^{\Del^{op}}(\CN\mc  U^{[n]},BG)$. We now describe these $n$-simplices more explicitly.
 
 \begin{lemma}\label{LEM:n-simplex-inCMAN(NU,BG)}
 A simplicial manifold map $\CMan^{\Del^{op}}(\CN\mc  U^{[n]},BG)$ is precisely given by $n+1$ many transition functions $g^{(0)}_{i,j}:U_{i,j}\to G, \dots, g^{(n)}_{i,j}:U_{i,j}\to G$, each satisfying the cocycle condition $g^{(p)}_{i,j}|_{U_{i,j,k}}\cdot g^{(p)}_{j,k}|_{U_{i,j,k}}=g^{(p)}_{i,k}|_{U_{i,j,k}}$ and $g^{(p)}_{i,i}=1$ for any $p=0,\dots, n$, together with $n$ maps $f^{1}_i:U_i\to G, \dots, f^{n}_i:U_i\to G$, each commuting with the transition functions via $f_i^{p}|_{U_{i,j}}\cdot g_{i,j}^{(p-1)}=g_{i,j}^{(p)}\cdot f_j^{p}|_{U_{i,j}}$.
 \end{lemma}
 \begin{proof}
 A simplicial manifold map $h\in\CMan^{\Del^{op}}(\CN\mc  U^{[n]},BG)$ is a map for each $k$-simplex, i.e., $(\CN\mc U^{[n]})_k\to BG_k$, or 
 $$
 \coprod_{i_0^{(j_0)},\dots, i_k^{(j_k)}\in I^{[n]}} h_{i_0^{(j_0)},\dots, i_k^{(j_k)}}:\coprod_{i_0^{(j_0)},\dots, i_k^{(j_k)}\in I^{[n]}}U_{i_0^{(j_0)},\dots, i_k^{(j_k)}}\to G^{\times k}.$$
 For $k=0$, this is vacuous, for $k=1$, we get $h_{i_0^{(j_0)}, i_1^{(j_1)}}: U_{i_0^{(j_0)}, i_1^{(j_1)}}\to G$, for $k=2$, we get $h_{i_0^{(j_0)}, i_1^{(j_1)}, i_2^{(j_2)}}: U_{i_0^{(j_0)}, i_1^{(j_1)}, i_2^{(j_2)}}\to G\times G$, etc. Since $h$ respects the face maps, we see that $h_{i_0^{(j_0)}, i_1^{(j_1)}, i_2^{(j_2)}}=(h_{i_0^{(j_0)}, i_1^{(j_1)}}|_V, h_{ i_1^{(j_1)}, i_2^{(j_2)}}|_V)$, where $V=U_{i_0^{(j_0)},i_1^{(j_1)},i_2^{(j_2)}}$, as well as $h_{i_0^{(j_0)}, i_1^{(j_1)}}|_V\cdot h_{ i_1^{(j_1)}, i_2^{(j_2)}}|_V=h_{ i_0^{(j_0)}, i_2^{(j_2)}}|_V$. This shows, in particular, that $h_{i_0^{(j_0)}, i_1^{(j_1)}, i_2^{(j_2)}}$ is determined by the $h_{i_0^{(j_0)}, i_1^{(j_1)}}$ and a similar arguments shows that furthermore all of the maps $h_{i_0^{(j_0)},\dots, i_k^{(j_k)}}=(h_{i_0^{(j_0)}, i_1^{(j_1)}}|_W, h_{i_1^{(j_1)}, i_2^{(j_2)}}|_W, \dots, h_{i_{k-1}^{(j_{k-1})}, i_k^{(j_k)}}|_W)$ are determined by the $h_{i_0^{(j_0)}, i_1^{(j_1)}}$ restricted to $W=U_{i_0^{(j_0)},\dots, i_k^{(j_k)}}$. Moreover, $h_{i_0^{(j_0)}, i_1^{(j_1)}}|_V\cdot h_{ i_1^{(j_1)}, i_2^{(j_2)}}|_V=h_{ i_0^{(j_0)}, i_2^{(j_2)}}|_V$ is the only condition that is imposed on the functions $h_{i_0^{(j_0)}, i_1^{(j_1)}}$ besides $h_{i^{(j)},i^{(j)}}=1$ coming from the degeneracy $\sigma_0:[1]\to [0]$.
 
 Now, for $0\leq p\leq n$, denote by $g^{(p)}_{i,j}:=h_{i^{(p)}, j^{(p)}}:U_{i,j}\to G$, and, for $1\leq p\leq n$, denote by $f_i^p:=h_{i^{(p)},i^{(p-1)}}:U_i\to G$. Then, $g^{(p)}_{i,j}|_{U_{i,j,k}}\cdot g^{(p)}_{j,k}|_{U_{i,j,k}}=g^{(p)}_{i,k}|_{U_{i,j,k}}$ and $f_i^{p}|_{U_{i,j}}\cdot g_{i,j}^{(p-1)}=h_{i^{(p)},i^{(p-1)}}|_{U_{i,j}}\cdot h_{i^{(p-1)},j^{(p-1)}}=h_{i^{(p)}j^{(p-1)}}=h_{i^{(p)},j^{(p)}}\cdot h_{j^{(p)},j^{(p-1)}}|_{U_{i,j}}=g_{i,j}^{(p)}\cdot f_j^{p}|_{U_{i,j}}$, so that these functions satisfy the stated conditions. On the other hand, all $h_{i_0^{(j_0)}, i_1^{(j_1)}}$ can be written as products of the $g^{(p)}_{i,j}$ and $f^p_i$ and their inverses, e.g. for $p<q$, we have $h_{i^{(p)},i^{(q)}}=(f^{p+1}_i)^{-1}\cdot(f^{p+2}_i)^{-1}\cdot \ldots\cdot (f^{q}_i)^{-1}$, etc.
 \end{proof}
 
We want to define a simplicial set map from $\CMan^{\Del^{op}}(\CN\mc  U^{[\bu]},BG)$ to $\Tot(\HVB(\CN\mc U))$. To this end, we will use the described $n$-simplices of $\Tot(\HVB(\CN\mc U))$ from proposition \ref{PROP:Tot(Hol(U))n}.

\begin{definition} \label{DEF:beta}
Let $G=GL(n,\C)$. Then, we define a map $$\beta:\CMan^{\Del^{op}}(\CN \mc U^{[\bu]},BG) \to\Tot(\HVB(\CN\mc U))$$ which assign to the data of an $n$-simplex in the domain, i.e., $g^{(0)}_{i,j}, \dots, g^{(n)}_{i,j}$ and $f_i^1,\dots, f_i^n$ from lemma \ref{LEM:n-simplex-inCMAN(NU,BG)}, the data of an $n$-simplex in the range from proposition \ref{PROP:Tot(Hol(U))n} as follows. Let $E_i^{(0)}=U_i\times \C^n\to U_i, \dots, E_i^{(n)}=U_i\times \C^n\to U_i$ be the product bundles with connections $\nabla_i^{(0)}=\partial, \dots, \nabla_i^{(n)}=\partial$, where $\partial=\sum_{\ell=1}^n dz_\ell\frac{\partial}{\partial z_\ell}$. This makes the $g^{(p)}_{i,j}$ maps of bundles $g^{(p)}_{i,j}:E^{(p)}_j|_{U_{i,j}}\to E^{(p)}_i|_{U_{i,j}}$ as well as the $f_i^p$ maps of bundles $f^p_i:E_i^{(p-1)}\to E_i^{(p)}$.
 \end{definition}
 Then, we claim:
 \begin{proposition}
$\beta:\CMan^{\Del^{op}}(\CN \mc U^{[\bu]},BG) \to\Tot(\HVB(\CN\mc U))$ is a map of simplicial sets.
 \end{proposition}
\begin{proof}
We need to show that $\beta$ commutes with the application of a morphism $\rho:[n]\to [m]$ of $\Del$.

In fact, a face map $\delta_j:[n-1]\to [n]$ in $\Del$ induces a simplicial set map $\CMan^{\Del^{op}}(\CN \mc U^{[n]},BG)\to \CMan^{\Del^{op}}(\CN \mc U^{[n-1]},BG)$ by forgetting the open sets $U_{i^{(j)}}$ of the $j$th component in $\mc U^{[n]}$, while degeneracies $\sigma_j:[n]\to [n-1]$ induce $\CMan^{\Del^{op}}(\CN \mc U^{[n-1]},BG)\to \CMan^{\Del^{op}}(\CN \mc U^{[n]},BG)$, which repeat the open sets $U_{i^{(j)}}$ of the $j$th component in $\mc U^{[n]}$ (with the unit $1$ for the transition function).

On the other hand, for the totalization (see definition \ref{DEF:Totalization}), a face map $\delta_j:[n-1]\to [n]$ maps $\prod_\ell \sSet(\Delt{\ell}\times \Delt{n},\HVB((\CN\mc U)_\ell))\to \prod_\ell \sSet(\Delt{\ell}\times \Delt{n-1},\HVB((\CN\mc U)_\ell))$ by pre-composing by $\Delt{n-1}\to \Delt{n}$, which under the interpretation from proposition \ref{PROP:Tot(Hol(U))n} forgets the $j$th bundles $E_i^{(j)}\to U_i$ (since the $E_i^{(p)}$ are the images for $\ell=0$, i.e., the images under the map $\Delt{0}\times \Delt{n}\to \HVB(\CN \mc U_0)$). Similarly, $\sigma_j:[n]\to [n-1]$ gives a map $\prod_\ell \sSet(\Delt{\ell}\times \Delt{n-1},\HVB((\CN\mc U)_\ell))\to \prod_\ell \sSet(\Delt{\ell}\times \Delt{n},\HVB((\CN\mc U)_\ell))$ by pre-composing with $\Delt{n}\to \Delt{n-1}$, which interpreted as in proposition \ref{PROP:Tot(Hol(U))n} (i.e., for $\ell=0$) repeats the $j$th bundle $E^{(j)}_i\to U_i$.

Since $\beta$ maps the $p$th component in the domain to the $p$th component in the range, and morphisms of $\Del$ act in the same way in the domain and range (forgetting the $j$ component for $\delta_j$, and repeating the $j$th component for $\sigma_j$), we see that $\beta$ is indeed a map of simplicial sets.
\end{proof}

 \begin{remark}\label{REM:compare-CMan(NU[.],BG)-Tot(HVB(NU))}
The image of $\beta$ does not give all $0$-simplices of $\Tot(\HVB(\CN\mc U))$, since, by construction (definition \ref{DEF:beta}), we only get trivial product bundles with connection $\partial$ on $U_i$. In fact, if we define $\HVB^{triv}$ to consist only of trivial product bundles with fiber $\C^n$ and connection $\partial$, then $\HVB^{triv}$ is a sub-simplicial presheaf of $\HVB$ so that $\beta:\CMan^{\Del^{op}}(\CN \mc U^{[\bu]},BG) \to\Tot(\HVB^{triv}(\CN\mc U))$ is an isomorphism.

However, every holomorphic vector bundle $E\to M$ together with a cover $\mc U$ and a choice of local trivializations over $\mc U$ can be represented as a $0$-simplex of $\CMan^{\Del^{op}}(\CN \mc U^{[\bu]},BG)$ via lemma \ref{LEM:n-simplex-inCMAN(NU,BG)} for $n=0$ (cf. lemma \ref{LEM:E-given-local} for $0$-simplices of $\Tot(\HVB(\CN\mc U))$). Therefore, diagram \eqref{EQU:alternative-Tot(CS)} will provide an alternative for calculating $\CS$ of $E$ with the choice of $\partial$ for the local connections.
\end{remark}

\subsection{A combinatorial integration over the fiber}

In order to define the left vertical map in equation \eqref{EQU:alternative-Tot(CS)}, we need an ``integration over the fiber'' for \v{C}ech cochains, i.e., a suitable map $\intfib{k}:\vC^\bu(\mc U^{[k]},\mc A)\to \vC^\bu(\mc U, \mc A)$, which we define in this section.

We start with some notation on indices. For $k\geq 0$, we ``split'' the set $\{0,\dots, q\}$ into $k+1$ levels by choosing positions $0\leq \stp_1\leq \stp_2\leq\dots\leq \stp_{k}\leq q$ where a step of a level occurs. More precisely, we make the following definition.
\begin{definition}
A $k$-step position of $\{0,\dots, q\}$ (or a $k$-step or simply a step) is defined to be a sequence of natural numbers $ 0\leq \stp_1\leq \stp_2\leq\dots\leq \stp_{k}\leq q$. The set of $k$-steps is denoted by
\[
S_k(q):=\{(\stp_1,\dots, \stp_{k})\in \N_0^{k}: 0\leq \stp_1\leq \stp_2\leq\dots\leq \stp_{k}\leq q\}.
\]
\end{definition}

Now, let $\mc U=\{U_i\}_{i\in I}$ be a cover of a manifold $M$, and consider a sequence of indices $(i_0,\dots, i_q)\in I^{q+1}$. (In all of the cases of interest below, these will be the indices applied to some element $c=\{c_{i_0,\dots, i_q}\}$ in some \v{C}ech complex.)  Using a $k$-step position $ 0\leq \stp_1\leq \stp_2\leq\dots\leq \stp_{k}\leq q$, we can split $(i_0,\dots, i_q)$ into $k+1$ subsequences
\begin{equation*}
(i_0,\dots, i_{\stp_1}), (i_{\stp_1},\dots, i_{\stp_2}), \dots, (i_{\stp_{k}},\dots, i_{q})
\end{equation*}

\begin{example}
Let $\mc A$ be a presheaf of non-negatively graded cochain complexes, such as e.g. the the sheaf of holomorphic functions $\mc A=\Ohol^\bu$. Let $\mc U=\{U_i\}_{i\in I}$ be a cover of a manifold $M$.
Recall the \v{C}ech complex $\vC^\bu(\mc U,\mc A)$ from definition \ref{DEF:Presheaf-Cech}. Given elements $c_0,\dots, c_k\in \vC^\bu(\mc U,\mc A)$, their product can be defined as
\begin{equation}\label{EQU:c1...ck}
(c_0\cdot \ldots\cdot c_k)_{i_0,\dots, i_q}
= \sum_{
\tiny \begin{matrix}
(\stp_1,\dots, \stp_{k})\\
\in S_{k}(q)
\end{matrix}
} (c_0)_{i_0,\dots, i_{\stp_1}}\cdot  (c_1)_{i_{\stp_1},\dots, i_{\stp_2}}\cdot\ldots\cdot (c_k)_{i_{\stp_{k}},\dots, i_{q}}.
\end{equation}
\end{example}

\begin{example}
Denote by $\mc U^{[k]}:=\{U_{i^{(j)}}\}_{i^{(j)}\in I^{[k]}}$ the $k$-fold cover from definition \ref{DEF:U^[n]}. For each sequence of indices $(i_0,\dots, i_q)$ of $I$, and for each choice of $k$-step positions $(\stp_1,\dots, \stp_k)\in S_k(q)$, there is an induced sequence of indices $(j_0,\dots,j_{q+k})$ of $I^{[k]}$ given by (cf. figure \ref{fig:IkIndices})
\begin{equation}\label{EQU:j's-from-i's}
\left\{
\begin{matrix}
 j_0=i_0^{(0)}, \dots, j_{\stp_1}=i_{\stp_1}^{(0)},\\
 j_{\stp_1+1}=i_{\stp_1}^{(1)}, \dots, j_{\stp_2+1}=i_{\stp_2}^{(1)},\\
  \dots \\
 j_{\stp_{m}+m}=i_{\stp_{m}}^{(m)}, \dots, j_{\stp_{m+1}+m}=i_{\stp_{m+1}}^{(m)},\\
 \dots \\
 j_{\stp_k+k}=i_{\stp_k}^{(k)}, \dots, j_{q+k}=i_{q}^{(k)}
 \end{matrix}
 \right.
\end{equation}

\begin{figure}[h]
\centering
\begin{tikzpicture}[scale=0.44,  every node/.style={scale=0.8}]



  \draw[thick][dotted] ($(-1,{(5-0)*2})$) -- ($(0,{(5-0)*2})$);
\node[left= 0.1cm] at ($(-1,{(5-0)*2})$) {0};
  \draw[thick][dotted] ($(-1,{(5-1)*2})$) -- ($(6,{(5-1)*2})$);
\node[left= 0.1cm] at ($(-1,{(5-1)*2})$) {1};
  \draw[thick][dotted] ($(-1,{(5-2)*2})$) -- ($(14,{(5-2)*2})$);
\node[left= 0.1cm] at ($(-1,{(5-2)*2})$) {2};
  \draw[thick][dotted] ($(-1,{(5-3)*2})$) -- ($(14,{(5-3)*2})$);
\node[left= 0.1cm] at ($(-1,{(5-3)*2})$) {3};
  \draw[thick][dotted] ($(-1,{(5-4)*2})$) -- ($(22,{(5-4)*2})$);
\node[left= 0.1cm] at ($(-1,{(5-4)*2})$) {4};
  \draw (-1,0) -- (30,0);
\node[left= 0.1cm] at (-1,0) {$k=5$};

\foreach \x in {0,1,...,2}
{ 
  \draw[thick][dotted] ($(2*\x,-2)$) -- ($(2*\x,10)$);
\node[below= 0.1cm] at ($(2*\x,-2)$) {${\x}$};}

\foreach \x in {3,4,5,6}
{ 
  \draw[thick][dotted] ($(2*\x,-2)$) -- ($(2*\x,8)$);
\node[below= 0.1cm] at ($(2*\x,-2)$) {${\x}$};}

\foreach \x in {7,8,9,10}
{ 
  \draw[thick][dotted] ($(2*\x,-2)$) -- ($(2*\x,4)$);
\node[below= 0.1cm] at ($(2*\x,-2)$) {${\x}$};}

\foreach \x in {11}
{ 
  \draw[thick][dotted] ($(2*\x,-2)$) -- ($(2*\x,2)$);
\node[below= 0.1cm] at ($(2*\x,-2)$) {${\x}$};}

\foreach \x in {12,13,14,14}
{ 
  \draw[thick][dotted] ($(2*\x,-2)$) -- ($(2*\x,0)$);
\node[below= 0.1cm] at ($(2*\x,-2)$) {${\x}$};}

  \draw[thick][dotted] ($(2*15,-2)$) -- ($(2*15,0)$);
\node[below= 0.1cm, xshift=0.3cm] at ($(2*15,-2)$) {$15 = q$};

\node[below= 0.1cm] at (6,-3) {$\stp_1 = 3$};
\node[below= 0.1cm] at (14,-3) {$\stp_2 = 7$};
\node[below= 0.1cm] at (14,-4) {$\stp_3 = 7$};
\node[below= 0.1cm] at (22,-3) {$\stp_4 = 11$};
\node[below= 0.1cm] at (24,-4) {$\stp_5 = 12$};

\draw[name path=s1, thick] (0,10) -- (6,10) -- (6,8) -- (14,8) -- (14, 4) -- (22,4) -- (22,2) -- (24,2) -- (24,0) -- (30,0);

\foreach \x in {0,1,..., 3}
{ 
\fill[black] ($(2*\x,10)$) circle (5pt);
\node[above right, xshift=0.3cm, yshift=-0.15cm, rotate=45] at ($(2*\x,10)$) {\contour{white}{\small $i^{(0)}_{\x} = j_{\x}$}};}

\fill[black] (6,8) circle (5pt);
\node[above right, xshift=0.4cm, yshift=-0.15cm, rotate=45]  at (6,8)  {\contour{white}{\small $i^{(1)}_{3} = j_{4}$}};

\foreach \x in {4,5,...,7}
{ 
 \pgfmathparse{int(round(\x+1) )}
       \let\y\pgfmathresult
       \fill[black] ($(2*\x,8)$) circle (5pt);
\node[above right, xshift=0.4cm, yshift=-0.15cm, rotate=45] at ($(2*\x,8)$) {\contour{white}{\small $i^{(1)}_{\x} = j_{\y}$}};}

\fill[black] (14,6) circle (5pt);
\node[above right, xshift=0.4cm, yshift=-0.15cm, rotate=45]  at (14,6)  {\contour{white}{$i^{(2)}_{7} = j_{9}$}};

\fill[black] (14,4) circle (5pt);
\node[above right, xshift=0.4cm, yshift=-0.15cm, rotate=45]  at (14,4)  {\contour{white}{$i^{(3)}_{7} = j_{10}$}};

\foreach \x in {8,9,10,11}
{ 
 \pgfmathparse{int(round(\x+3) )}
       \let\y\pgfmathresult
       \fill[black] ($(2*\x,4)$) circle (5pt);
\node[above right, xshift=0.4cm, yshift=-0.15cm, rotate=45]  at ($(2*\x,4)$) {\contour{white}{\small $i^{(3)}_{\x} = j_{\y}$}};}

\fill[black] (22,2) circle (5pt);
\node[above right, xshift=0.4cm, yshift=-0.15cm, rotate=45]  at  (22,2)  {\contour{white}{$i^{(4)}_{11} = j_{15}$}};

\fill[black] (24,2) circle (5pt);
\node[above right, xshift=0.4cm, yshift=-0.15cm, rotate=45]  at  (24,2)  {\contour{white}{$i^{(4)}_{12} = j_{16}$}};

\fill[black] (24,0) circle (5pt);
\node[above right, xshift=0.4cm, yshift=-0.15cm, rotate=45]  at  (24,0)  {\contour{white}{$i^{(5)}_{12} = j_{17}$}};

\foreach \x in {13,14,15}
{ 
 \pgfmathparse{int(round(\x+5) )}
       \let\y\pgfmathresult
       \fill[black] ($(2*\x,0)$) circle (5pt);
\node[above right, xshift=0.4cm, yshift=-0.15cm, rotate=45]  at ($(2*\x,0)$) {\contour{white}{\small $i^{(5)}_{\x} = j_{\y}$}};}

\end{tikzpicture}
\caption{A visual representation of the indices in $I^{[k]}$ induced by a $k$-step $(\stp_1,\dots, \stp_k)$ with $k=5$ and $q=15$} \label{fig:IkIndices}
\end{figure}
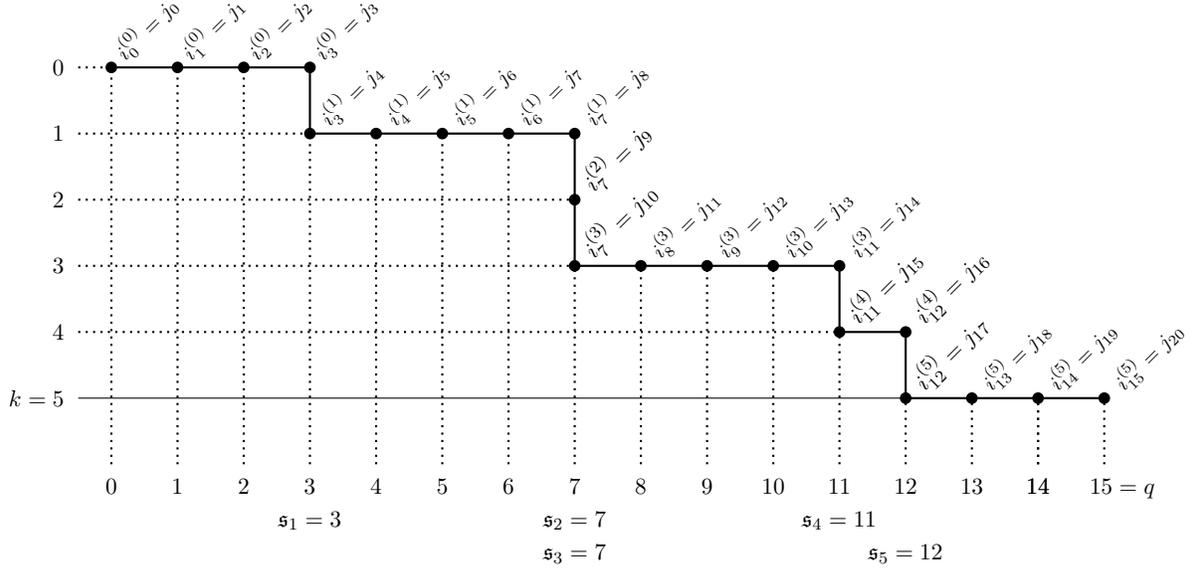

The set of all indices of $I^{[k]}$ obtained by splitting $(i_0,\dots, i_q)$ into $k+1$ levels described in the above way is denoted by
\begin{equation}
J_k(i_0,\dots,i_q):=\big\{(j_0,\dots,j_{q+k}): \exists(\stp_1,\dots,\stp_k)\in S_k(q)\text{ such that \eqref{EQU:j's-from-i's}  holds true} \big\}
\end{equation}

Note from \eqref{EQU:j's-from-i's}, that for $(j_0,\dots,j_{q+k})\in J_k(i_0,\dots,i_q)$ and $0\leq m\leq k$ the indices of the $m$th level occur exactly at $j_{\stp_m+m}, \dots, j_{\stp_{m+1}+m}$, and this information can always be recovered from $(j_0,\dots,j_{q+k})$. Thus, the step to and from the $m$th level occur exactly at $j_{\stp_m+m}$ and $j_{\stp_{m+1}+m}$. For our purposes, it is important to note that we do allow the special case where $\stp_m=\stp_{m+1}$, in which case there is only one index $j_{\stp_m+m}=j_{\stp_{m+1}+m}$ at the $m$th level. For $m=1,\dots, k-1$, we denote by $\hat J^m_k(i_0,\dots, i_q)$ those indices that come from $J_k(i_0,\dots, i_q)$ with either $j_{\stp_m+m}$ or $j_{\stp_{m+1}+m}$ removed:
\begin{align}\label{EQU:Jhat-m-k}
\hat J^{m,\leftarrow}_k(i_0,\dots, i_q)&:= \Big\{(j_0,\dots, \widehat{j_{\ell}}, \dots, j_{q+k}): \exists (j_0, \dots, j_{q+k})\in J_k(i_0,\dots, i_q) 
\\
\nonumber
&\hspace{2.74in}\text{such that $\ell={\stp_m+m}$}\Big\},
\\
\nonumber
\hat J^{m,\rightarrow}_k(i_0,\dots, i_q)&:= \Big\{(j_0,\dots, \widehat{j_{\ell}}, \dots, j_{q+k}): \exists (j_0, \dots, j_{q+k})\in J_k(i_0,\dots, i_q) 
\\
\nonumber &\hspace{1.5in} 
\text{such that $\ell={\stp_{m+1}+m}$ but $\ell\neq {\stp_m+m}$}\Big\},
\\
\nonumber
\hat J^m_k(i_0,\dots, i_q)&:= \hat J^{m,\leftarrow}_k(i_0,\dots, i_q) \sqcup \hat J^{m,\rightarrow}_k(i_0,\dots, i_q).
\end{align}
For $m=0$, respectively $m=k$, we only remove the index where the step occurs, but not $j_0$, respectively $j_q$. More precisely, we define
\begin{multline}\label{EQU:Jhat-0-k}
\hat J^0_k(i_0,\dots, i_q):=\hat J^{0,\rightarrow}_k(i_0,\dots, i_q)
\\
:= \Big\{(j_0,\dots, \widehat{j_{\ell}},\dots, j_{q+k}): \exists (j_0,\dots, j_{q+k})\in J_k(i_0,\dots, i_q)\text{ such that $\ell={\stp_{1}}$}\Big\},
\end{multline}
and
\begin{multline}\label{EQU:Jhat-k-k}
\hat J^k_k(i_0,\dots, i_q):=\hat J^{k,\leftarrow}_k(i_0,\dots, i_q)
\\
:= \Big\{(j_0,\dots, \widehat{j_{\ell}}, \dots, j_{q+k}): \exists (j_0,\dots,j_{q+k})\in J_k(i_0,\dots, i_q)\text{ such that $\ell={\stp_{k}+k}$}\Big\}.
\end{multline}
\end{example}

\begin{lemma}\label{LEM:J=Jhat+hatJ}
Fix a set of indices $i_0,\dots, i_q\in I$, and a $k\geq 0$. Then, the map 
\begin{align*}
f:J_k(i_0,\dots, i_q)\times\{0,\dots, q+k\}
 &\to \Big(\bigsqcup_{0\leq r\leq q} J_k(i_0,\dots, \widehat{i_r},\dots, i_q)\Big)\sqcup\Big(\bigsqcup_{0\leq m\leq k}\hat{J}_k^m(i_0,\dots, i_q)\Big)
\\
f:((j_0,\dots,j_{q+k}),\ell)&\mapsto (j_0,\dots,\widehat{j_\ell},\dots, j_{q+k}),
\end{align*}
which removes the $\ell$th index $j_{\ell}$, is a bijection.
\end{lemma}
\begin{proof}
First, note that the map $f$ is well-defined. If the removed index $j_\ell$ is either the beginning $j_{\stp_{m}+m}$ or the end $j_{\stp_{m+1}+m}$ index of a level (say the $m$th level), then $f((j_0,\dots,j_{q+k}),\ell)= (j_0,\dots,\widehat{j_\ell},\dots, j_{q+k})$ lands in $\hat{J}_k^m(i_0,\dots, i_q)$. Otherwise, $f$ removes one of the original indices, say $i_r$, in which case $f((j_0,\dots,j_q),\ell)$ lands in $J_k(i_0,\dots,\widehat{i_r},\dots, i_q)$. We can construct the inverse $f^{-1}$ by observing that for each $0\leq r\leq q$ and $(j'_0,\dots, j'_{q+k-1})\in J_{k}(i_0,\dots, \widehat{i_r},\dots,j_{q})$ there exists a unique $(j_0,\dots,j_{q+k})\in J_k(i_0,\dots, i_q)$ and $0\leq \ell\leq q+k$ so that $ (j_0,\dots,\widehat{j_\ell},\dots, j_{q+k})=(j'_0,\dots,j'_{q+k-1})$. Similarly, for each $0\leq m\leq k$ and $(j'_0,\dots, j'_{q+k-1})\in \hat{J}^m_{k}(i_0,\dots,j_{q})$ there exists a unique $(j_0,\dots,j_{q+k})\in J_k(i_0,\dots, i_q)$ and $0\leq \ell\leq q+k$ so that $ (j_0,\dots,\widehat{j_\ell},\dots, j_{q+k})=(j'_0, \dots,j'_{q+k-1})$ with $j_\ell$ on the $m$th level.
\end{proof}

We next define the integration over the fiber map.

\begin{definition}
Let $\mc U=\{U_i\}_{i\in I}$ be a cover of a complex manifold $M$, and let $\mc U^{[k]}$ be the $k$-fold cover coming from $\mc U$ from definition \ref{DEF:U^[n]}. For an element $\mu\in \vC^\bu(\mc U^{[k]}, \mc A)$ in the \v{C}ech complex, we define the following integration over the fiber map $\intfib{k}:\vC^\bu(\mc U^{[k]},\mc A)\to \vC^\bu(\mc U, \mc A)$, which maps the components $\intfib{k}:\vC^{q+k}(\mc U^{[k]},\mc A^r)\to \vC^{q}(\mc U, \mc A^r)$, by setting
\begin{equation}\label{EQU:Def-integration}
\Big(\intfib{k}\mu\Big)_{i_0,\dots,i_q}:=
\sum_{(j_0,\dots, j_{q+k})\in J_k(i_0,\dots, i_q)}
 (-1)^{\stp_1+\dots+\stp_k}\cdot \mu_{j_0,\dots, j_{q+k}}
\end{equation}
Note that the sign is well defined, since each $(j_0,\dots, j_{q+k})\in J_k(i_0,\dots, i_q)$ uniquely determines a $k$-step $(\stp_1,\dots,\stp_k)$.

Let $k>0$, and let $j\in \{0,\dots,k\}$. For the $j$th face map $\delta_j:[k-1]\to [k]$, there is a map of covers $\mc U^{[\bu]}(\delta_j)\in \Cov_M(\mc U^{[k-1]}, \mc U^{[k]})$ given by ignoring the open sets $U_{i^{(j)}}$ of the $j$th component of the cover $\mc U^{[k]}$. In particular, by definition \ref{DEF:Presheaf-Cech}, there is an induced map $\widetilde{\delta_j}:\vC^\bu(\mc U^{[k]}, \mc A)\to \vC^\bu(\mc U^{[k-1]}, \mc A)$, which forgets the $j$th open sets $U_{i^{(j)}}$, i.e., $\widetilde{\delta_j}(\mu)\in \vC^\bu(\mc U^{[k-1]}, \mc A)$ is the collection determined by $\mu\in \vC^\bu(\mc U^{[k]}, \mc A)$, which is only defined on indices not including any $i^{(j)}$ for $i\in I$.
\end{definition}
With this notation, we have the following integration over the fiber formulae.
\begin{proposition}\label{PROP:delta-tr=...}
The integration over the fiber commutes with the internal differential $d_{\mc A}$ of $\mc A$, i.e.,
\begin{equation}\label{EQU:integration-d}
d_{\mc A}\Big(\intfib{k} \mu \Big)  = \intfib{k} d_{\mc A}(\mu) 
\end{equation}
For the \v{C}ech differential $\delta$, we get the following identity,
\begin{equation}\label{EQU:integration-delta}
\intfib{k} \delta(\mu)  = (-1)^k\cdot \delta\Big(\intfib{k} \mu \Big) + \sum_{j=0}^k  (-1)^j \cdot \int_{\Delta^{k-1}} \widetilde{\delta_j}(\mu )
\end{equation}
\end{proposition}
\begin{proof}
For the first equation \eqref{EQU:integration-d} note that both sides of \eqref{EQU:Def-integration} are on the open set $U_{i_0,\dots, i_q}$, so that the same differential $d_{\mc A}$ of $\mc A(U_{i_0,\dots, i_q})$ is applied inside and outside the sum of \eqref{EQU:Def-integration}.

Next, we prove \eqref{EQU:integration-delta}. For fixed indices $i_0,\dots, i_q\in I$, we first calculate $\delta\big(\intfib{k} \mu \big)$ on $U_{i_0,\dots,i_q}$ to be
\begin{align*}
\Big(\delta\big(\intfib{k} \mu \big)\Big)_{i_0,\dots, i_q} =& \sum_{0\leq r\leq q} (-1)^{r}\cdot \big(\intfib{k} \mu \big)_{i_0,\dots, \widehat{i_r},\dots,i_q}\\
=&
\sum_{0\leq r\leq q}\quad
\sum_{(j'_0,\dots, j'_{q+k-1})\in J_k(i_0,\dots, \widehat{i_r},\dots,i_q)}
 (-1)^{r+\stp'_1+\dots+\stp'_k}\cdot \mu_{j'_0,\dots, j'_{q+k-1}}.
\end{align*}
Next, we calculate $\intfib{k} \delta(\mu) $ on $U_{i_0,\dots,i_q}$ to be
\begin{align*}
\Big( \intfib{k} \delta(\mu) \Big)_{i_0,\dots, i_q}=
&
\sum_{(j_0,\dots, j_{q+k})\in J_k(i_0,\dots, i_q)} (-1)^{\stp_1+\dots+\stp_k}\cdot \sum_{0\leq \ell\leq q+k} (-1)^\ell \mu_{j_0,\dots,\widehat{j_\ell},\dots, j_{q+k}}
\\
=&
\sum_{0\leq r\leq q}\quad
 \sum_{(j'_0,\dots, j'_{q+k-1})\in J_k(i_0,\dots,\widehat{i_r}, \dots, i_q)}
 (-1)^{r+k+\stp'_1+\dots+\stp'_k}\cdot \mu_{j'_0,\dots ,j'_{q+k-1}}
 \\  
 &+\sum_{0\leq m\leq k}\quad
  \sum_{(j'_0,\dots, j'_{q+k-1})\in \hat{J}^m_k(i_0, \dots, i_q)}
 (-1)^{r+m+\stp'_1+\dots+\stp'_k}\cdot \mu_{j'_0,\dots ,j'_{q+k-1}},
 \end{align*}
where we have used lemma \ref{LEM:J=Jhat+hatJ} in the last equality. (To see the sign in the upper line of the right hand side, note that if the removed index $j_\ell=i_{r}^{(\kappa)}$ occurs at $i_r$ at the $\kappa$th level, then $\stp_1=\stp'_1,\dots,\stp_\kappa=\stp'_\kappa$ while $\stp_{\kappa+1}=\stp'_{\kappa+1}-1,\dots,\stp_{k}=\stp'_{k}-1$, and $\ell=r+\kappa$, thus $\ell+\stp_1+\dots+\stp_k=(r+\kappa)+\stp'_1+\dots+\stp'_k-(k-\kappa)\equiv r+k+\stp'_1+\dots+\stp'_k (\text{mod }2)$. For the sign in the lower line of the right hand side, assume again that $j_\ell=i_r^{(\kappa)}$, and note that in this case the $\stp_1=\stp'_1, \dots, \stp_k =\stp'_k$ do not change, while $\ell=r+m$ is the number of indices before $j_\ell$.)

It therefore remains to show that $\sum_{j=0}^k \int_{\delta_j(\Delta^k)} \mu$ on $U_{i_0,\dots, i_q}$ can be written as
\begin{equation}\label{EQU:int-deltam-Deltak}
\Big(\sum_{j=0}^k (-1)^j \int_{\Delta^{k-1}} \widetilde{\delta_j}(\mu)\Big)_{i_0,\dots,i_q}=\sum_{0\leq m\leq k}  \,\,\sum_{ (j'_0,\dots, j'_{q+k-1})\in \hat J^m_k(i_0,\dots, i_q)} (-1)^{r+m+\stp'_1+\dots+\stp'_k}\cdot \mu_{j'_0,\dots, j'_{q+k}}.
\end{equation}

\begin{proof}[Proof of Equation \eqref{EQU:int-deltam-Deltak}]
We evaluate the right hand side of equation \eqref{EQU:int-deltam-Deltak}. First, we claim that the right hand side of \eqref{EQU:int-deltam-Deltak} vanishes except for the terms where
\begin{enumerate}
\item\label{ITEM:J0} either $\stp_1=0$ in $\hat{J}^0_k(i_0,\dots,i_q)$,
 \item\label{ITEM:Jm} or $\stp_m=\stp_{m+1}$ in $\hat{J}^m_k(i_0,\dots,i_q)$ for $m=1,\dots, k-1$,
 \item\label{ITEM:Jk} or $\stp_k=q$ in $\hat{J}^k_k(i_0,\dots,i_q)$. 
 \end{enumerate}
 Since we fixed $(i_0,\dots, i_q)$, we will simplify notation by writing $\hat{J}^m_k=\hat{J}^m_k(i_0,\dots,j_q)$.

To see \eqref{ITEM:J0}, if $\stp_1>0$, then the indices $(j_0,\dots, \widehat{j_{\stp_1}},\dots, j_{q+k})\in \hat{J}^{0,\rightarrow}_k$ coincide with the indices $(j_0,\dots, \widehat{j_{\stp'_1+1}},\dots, j_{q+k})\in \hat{J}^{1,\leftarrow}_k$ for the new steps $(\stp'_1,\stp'_2,\dots, \stp'_k)=(\stp_1-1,\stp_2,\dots,\stp_k)$, since $(j_{\stp_1-1},\widehat{j_{\stp_1}},j_{\stp_1+1})=(i^{(0)}_{\stp_1-1},\widehat{i^{(0)}_{\stp_1}},i^{(1)}_{\stp_1})=(i^{(0)}_{\stp_1-1},\widehat{i^{(1)}_{\stp_1-1}},i^{(1)}_{\stp_1})=(i^{(0)}_{\stp'_1},\widehat{i^{(1)}_{\stp'_1}},i^{(1)}_{\stp'_1+1})=(j_{\stp'_1},\widehat{j_{\stp'_1+1}},{j_{\stp'_1+2}})$. Thus, the same term appears twice, once from $\hat{J}^0_k$ with $(\stp_1,\dots,\stp_k)$, and once from $\hat{J}^1_k$ with $(\stp'_1,\dots,\stp'_k)$, and cancels as they have opposite signs (as the ``$r+m$'' part of the sign is the same for both, but $\stp'_1=\stp_1-1$).

Next, for \eqref{ITEM:Jm}, if $\stp_m<\stp_{m+1}$, we can either have the indices $(j_0,\dots, \widehat{j_{\stp_m+m}},\dots, j_{q+k})\in \hat{J}^{m,\leftarrow}_k$ or $(j_0,\dots, \widehat{j_{\stp_{m+1}+m}},\dots, j_{q+k})\in \hat{J}^{m,\rightarrow}_k$ appear in $\hat{J}^m_k$. In the first case, $(j_0,\dots, \widehat{j_{\stp_m+m}},\dots, j_{q+k})$ coincides with $(j_0,\dots, \widehat{j_{\stp'_{(m-1)+1}+(m-1)}},\dots, j_{q+k})\in \hat{J}^{m-1,\rightarrow}_k$ for the steps $(\stp'_1,\dots,\stp'_m, \dots, \stp'_k)=(\stp_1,\dots, \stp_m+1,\dots, \stp_k)$, since the indices coincide after removal of the index in question, i.e., $(j_{\stp_m+m-1},\widehat{j_{\stp_m+m}}, j_{\stp_m+m+1})=(i^{(m-1)}_{\stp_m},\widehat{i^{(m)}_{\stp_m}}, i^{(m)}_{\stp_m+1})=(i^{(m-1)}_{\stp_m},\widehat{i^{(m-1)}_{\stp_m+1}}, i^{(m)}_{\stp_m+1})=(i^{(m-1)}_{\stp'_m-1},\widehat{i^{(m-1)}_{\stp'_m}}, i^{(m)}_{\stp'_m})=(j_{\stp'_m+m-2},\widehat{j_{\stp'_m+m-1}},j_{\stp'_m+m})$. The two corresponding terms have opposite signs (since $\stp'_m=\stp_m+1$), and thus cancel. In the second case, $(j_0,\dots, \widehat{j_{\stp_{m+1}+m}},\dots, j_{q+k})$ coincides with $(j_0,\dots, \widehat{j_{\stp'_{(m+1)}+(m+1)}},\dots, j_{q+k})\in \hat{J}^{m+1,\leftarrow}_k$ for $(\stp'_1,\dots,\stp'_{m+1}, \dots, \stp'_k)=(\stp_1,\dots, \stp_{m+1}-1,\dots, \stp_k)$, since we have again coinciding indices $(j_{\stp_{m+1}+m-1},\widehat{j_{\stp_{m+1}+m}}, j_{\stp_{m+1}+m+1})=(i^{(m)}_{\stp_{m+1}-1},\widehat{i^{(m)}_{\stp_{m+1}}}, i^{(m+1)}_{\stp_{m+1}})=(i^{(m)}_{\stp_{m+1}-1},\widehat{i^{(m+1)}_{\stp_{m+1}-1}}, i^{(m+1)}_{\stp_{m+1}})
=(i^{(m)}_{\stp'_{m+1}},\widehat{i^{(m+1)}_{\stp'_{m+1}}}, i^{(m+1)}_{\stp'_{m+1}+1})
=(j_{\stp'_{m+1}+m},\widehat{j_{\stp'_{m+1}+m+1}},j_{\stp'_{m+1}+m+2})$. Again, these have opposite signs (since $\stp'_{m+1}=\stp_{m+1}-1$) and thus cancel.

Finally, for \eqref{ITEM:Jk}, if $\stp_k<q$, the indices $(j_0,\dots, \widehat{j_{\stp_k+k}},\dots, j_{q+k})\in \hat{J}^{k,\leftarrow}_k$ coincide with the indicies $(j_0,\dots, \widehat{j_{\stp'_{(k-1)+1}+(k-1)}},\dots, j_{q+k})\in \hat{J}^{k-1,\rightarrow}_k$ for $(\stp'_1,\dots,\stp'_{k-1},\stp'_k)=(\stp_1,\dots, \stp_{k-1},\stp_k+1)$, since removing the appropriate index yields $(j_{\stp_k+k-1},\widehat{j_{\stp_k+k}},j_{\stp_k+k+1})=(i^{(k-1)}_{\stp_k},\widehat{i^{(k)}_{\stp_k}}, i^{(k)}_{\stp_k+1})=(i^{(k-1)}_{\stp_k},\widehat{i^{(k-1)}_{\stp_k+1}}, i^{(k)}_{\stp_k+1})=(i^{(k-1)}_{\stp'_k-1},\widehat{i^{(k-1)}_{\stp'_k}}, i^{(k)}_{\stp'_k})=(j_{\stp'_k+k-2},\widehat{j_{\stp'_{k}+k-1}},j_{\stp'_k+k})$. As the corresponding terms have opposite signs (due to $\stp'_{k}=\stp_{k}+1$), they cancel.

Thus, the only remaining terms are as follows. For \eqref{ITEM:J0}, there are terms in $\hat{J}^0_k$ with $\stp_1=0$ and $j_0$ removed, i.e., we only have steps that skip the $0$th level altogether. For \eqref{ITEM:Jm}, we have terms in $\hat{J}^m_k$ with $\stp_m=\stp_{m+1}$ and $j_{\stp_m+m}$ removed, i.e., we only have steps that skip the $m$th level altogether. For \eqref{ITEM:Jk}, we have terms in $\hat{J}^k_k$ with $\stp_k=q$ and $j_{q+k}$ removed, i.e., we have steps that skip the $k$th level altogether.
We thus sum over steps that are in $J_{k-1}(i_0,\dots, i_q)$ where we skip over the $m$th level for $m=0,\dots, k$. Note conversely, that for any step in $J_{k-1}(i_0,\dots, i_q)$ and any $m=0,\dots,k$, we can add another level, which will be the $m$th level, so that the steps come from $\hat{J}^m_k$ via removing the $m$th level. This shows that 
\begin{align*}
\sum_{0\leq m\leq k}\,\, &
\sum_{(j'_0,\dots, j'_{q+k-1})\in \hat J^m_k(i_0,\dots, i_q)}
(-1)^{r+m+\stp'_1+\dots+\stp'_k}\cdot \mu_{j'_0,\dots,j'_{q+k-1}}
\\
&=
\sum_{0\leq m\leq k}(-1)^m 
\sum_{(j'_0,\dots, j'_{q+k-1})\in J_{k-1}(i_0,\dots, i_q)}
(-1)^{\stp'_1+\dots+\widehat{\stp'_m}+\dots+\stp'_k}
(\widetilde{\delta_m}(\mu))_{j'_0,\dots, j'_{q+k-1}}
\\
&=\Big(\sum_{m=0}^k (-1)^m \int_{\Delta^{k-1}} \widetilde{\delta_m}(\mu)\Big)_{i_0,\dots, i_q},
\end{align*}
where we have used that $\stp'_m=r$ in the first equality. This proves equation \eqref{EQU:int-deltam-Deltak}.
\end{proof}
This completes the proof of proposition \ref{PROP:delta-tr=...}.
\end{proof}

\subsection{Computing $\Tot(\CS(\CN \mc U))$ on product bundles}

We will define the left vertical map of \eqref{EQU:alternative-Tot(CS)} as a composition of two maps $\gamma$ and $\iota$,
\begin{equation}\label{EQU:gamma-iota}
\xymatrix{ 
\CMan^{\Del^{op}}(\CN \mc U^{[\bu]},BG)  \ar^{\gamma}[d]  
\\ 
(\vC^\bu(\mc U^{[\bu]},\Ohol^\bu))^{even}_{closed} \ar^{\iota}[d]
\\
\DKSet(\vC^\bu(\mc U, \Ohol^\bu)\ul) 
  }
\end{equation}
Here, $(\vC^\bu(\mc U^{[\bu]},\Ohol^\bu))^{even}_{closed}$ denotes the simplicial set, whose $k$-simplices are $\delta$-closed elements of $\vC^\bu(\mc U^{[k]},\Ohol^\bu)$ which are of even total degree. 

\begin{definition}\label{DEF:gamma}
Assume that $G=GL(n,\C)$. For a cover $\mc V=\{V_j\}_{j\in J}\in \Cov_M$ of $M$, we define the map $\gamma_{\mc V}:\CMan^{\Del^{op}}(\CN \mc V,BG)\stackrel {\gamma_{\mc V}}\to (\vC^\bu(\mc V,\Ohol^\bu))^{even}_{closed}$ as follows. By lemma \ref{LEM:n-simplex-inCMAN(NU,BG)} (for $\mc U=\mc V$ and $n=0$), an element $h\in\CMan^{\Del^{op}}(\CN \mc V,BG)$ is given by transition functions $g_{i,j}:V_{i,j}\to G\subseteq \C^{n,n}$. Then, define $\gamma_{\mc V}(h)$ on the open set $V_{j_0,\dots, j_p}$ to be
\begin{equation}\label{DEF:gamma_V(h)}
(\gamma_{\mc V}(h))_{j_0,\dots,j_p}:=\tr\Big(g_{j_p,j_{0}}^{-1}\cdot \partial (g_{j_p,j_{p-1}})\cdot \ldots \cdot \partial (g_{j_2,j_{1}})\cdot \partial (g_{j_1,j_{0}}) \Big)
\end{equation}
Note that $(\gamma_{\mc V}(h))_{j_0,\dots,j_p}\in \Ohol^p(V_{j_0,\dots, j_p})$ is of \v{C}ech degree $p$ and form degree $p$, and thus of even total degree $2p$ in $\vC^\bu(\mc V,\Ohol^\bu)$. The collection of all these $\{(\gamma_{\mc V}(h))_{j_0,\dots,j_p}\}_{j_0,\dots, j_p\in J}$ is $\delta$-closed in $\vC^\bu(\mc V,\Ohol^\bu)$, since
\begin{align*}
(\delta(\gamma_{\mc V}(h)))_{j_0,\dots,j_p}
=&
\tr\Big(g_{j_{p},j_{1}} ^{-1}  
  \partial(g_{j_{p},j_{p-1}})   \dots  \partial(g_{j_2,j_1})\Big)
 \\
& +\sum_{\ell=1}^{p-1}(-1)^\ell\cdot  \tr\Big(g_{j_{p},j_{0}} ^{-1}  
  \partial(g_{j_{p},j_{p-1}})   \dots \partial(g_{j_{\ell+1},j_{\ell-1}}) \dots \partial(g_{j_1,j_0})\Big)
\\
& +(-1)^p \cdot \tr\Big(g_{j_{p-1},j_{0}} ^{-1}  
  \partial(g_{j_{p-1},j_{p-2}})   \dots  \partial(g_{j_1,j_0})\Big)=0
\end{align*}
vanishes, just as in the proof of theorem \ref{THM:CS-natural-transformation} (using the Leibniz property of $\partial$ and the cyclicity of the trace).

Now, the simplicial set map $\gamma:\CMan^{\Del^{op}}(\CN \mc U^{[\bu]},BG)\to (\vC^\bu(\mc U^{[\bu]},\Ohol^\bu))^{even}_{closed}$ from \eqref{EQU:gamma-iota} in simplicial degree $n$ is defined as $\gamma_n:=\gamma_{\mc U^{[n]}}:\CMan^{\Del^{op}}(\CN \mc U^{[n]},BG)\to (\vC^\bu(\mc U^{[n]},\Ohol^\bu))^{even}_{closed}$. Note that $\gamma$ respects morphisms in $\Del$, since the simplicial structure in the domain and range of $\gamma$ comes from the cosimplicial cover $\mc U^{[\bu]}:\Del\to \Cov_M$. 
\end{definition}

Next, we define the map $\iota$ from equation \eqref{EQU:gamma-iota}.

\begin{definition}\label{DEF:iota}
$\iota:(\vC^\bu(\mc U^{[\bu]},\Ohol^\bu))^{even}_{closed}\to \DKSet(\vC^\bu(\mc U, \Ohol^\bu)\ul) $ is a map, which assigns to an $n$-simplex $c\in (\vC^\bu(\mc U^{[n]},\Ohol^\bu))^{even}_{closed}$ an $n$-simplex in $\DKSet(\vC^\bu(\mc U, \Ohol^\bu)\ul)_n={\Chain^-}(N(\Z\Delt{n}),\vC^\bu(\mc U, \Ohol^\bu)\ul)$, i.e., a chain map from the chains on the standard $n$-simplex to $\vC^\bu(\mc U, \Ohol^\bu)\ul$. If $e_{i_0,\dots, i_\ell}$ with $0\leq i_0<\dots<i_\ell\leq n$ is a generator of $N(\Z\Delt{n})_{-\ell}$ as in example \ref{REM:NZDeltak}, then denote by $\lambda:[\ell]\to [n]$ the map $\lambda(j):=i_j$, which induces a map $\widetilde\lambda:(\vC^\bu(\mc U^{[n]},\Ohol^\bu))^{even}_{closed}\to (\vC^\bu(\mc U^{[\ell]},\Ohol^\bu))^{even}_{closed}$. When $c$ concentrated in homogenious total degree $|c|$, then we define $\iota(c)$  by
\begin{equation}\label{EQU:iota(c)}
\iota(c):N(\Z\Delt{n})\to \vC^\bu(\mc U, \Ohol^\bu)\ul, \quad e_{i_0,\dots, i_\ell}\mapsto (-1)^{\frac{\ell(\ell-1)}{2}}\cdot u^{|c|/2}\cdot \int_{\Delta^\ell} \widetilde{\lambda}(c)
\end{equation}
Note that since the degree of $c$ is $|c|$, the degree $|\widetilde{\lambda}(c)|=|c|$, so that $|\int_{\Delta^{\ell}}\widetilde{\lambda}(c)|=|c|-\ell$, and $\big|u^{|c|/2}\cdot \int_{\Delta^\ell} \widetilde{\lambda}(c)\big|=-\ell$.
\end{definition}
\begin{proposition}
$\iota:(\vC^\bu(\mc U^{[\bu]},\Ohol^\bu))^{even}_{closed}\to \DKSet(\vC^\bu(\mc U, \Ohol^\bu)\ul) $ from definition \ref{DEF:iota} is a well-defined map of simplicial sets.
\end{proposition}
\begin{proof}
First, we show that $\iota(c)$ as defined in \eqref{EQU:iota(c)} is indeed a chain map:
\begin{align*}
\delta(\iota(c)(e_{i_0,\dots, i_\ell}))
&=(-1)^{\frac{\ell(\ell-1)}{2}}\cdot u^{|c|/2}\cdot \delta\Big(\int_{\Delta^\ell} \widetilde{\lambda}(c)\Big)
\\
&\stackrel{\eqref{EQU:integration-delta}}= (-1)^{\frac{\ell(\ell-1)}{2}}\cdot u^{|c|/2}\cdot (-1)^\ell\cdot \left[ \int_{\Delta^\ell} \delta(\widetilde{\lambda}(c))-\sum_{j=0}^\ell (-1)^j \cdot \int_{\Delta^{\ell-1}} \widetilde{\delta_j}(\widetilde{\lambda}(c))\right]
\\
&= \sum_{j=0}^\ell (-1)^j \cdot (-1)^{\frac{(\ell-1)(\ell-2)}{2}}\cdot u^{|c|/2}\int_{\Delta^{\ell-1}} \widetilde{\delta_j}(\widetilde{\lambda}(c))
\\
&=\sum_{j=0}^\ell (-1)^j\cdot \iota(c)(e_{i_0,\dots,\widehat{i_j},\dots, i_\ell})
\\
&=\iota(c)(d(e_{i_0,\dots, i_\ell})),
\end{align*}
where we used that $\delta(\widetilde \lambda(c))=\widetilde\lambda(\delta(c))=0$. To see that $\iota$ is a map of simplicial sets, let $\rho:[n]\to [m]$, and $\widetilde \rho :(\vC^\bu(\mc U^{[m]},\Ohol^\bu))^{even}_{closed}\to (\vC^\bu(\mc U^{[n]},\Ohol^\bu))^{even}_{closed}$ the induced map. Then, for $c\in (\vC^\bu(\mc U^{[m]},\Ohol^\bu))^{even}_{closed}$ and $e_{i_0,\dots, i_\ell}$ a generator of $N(\Z\Delt{n})_{-\ell}$ and $\lambda:[\ell]\to [n]$ as before, we get
\begin{align*}
\iota(\widetilde\rho(c))(e_{i_0,\dots, i_\ell})&= (-1)^{\frac{\ell(\ell-1)}{2}}\cdot u^{|c|/2}\cdot \int_{\Delta^\ell} \widetilde{\lambda}(\widetilde \rho(c))= (-1)^{\frac{\ell(\ell-1)}{2}}\cdot u^{|c|/2}\cdot \int_{\Delta^\ell} \widetilde{(\rho\circ \lambda)}(c)\\
&=\iota(c)(e_{\rho(\lambda(0)),\dots, \rho(\lambda(\ell))}) = (\iota(c))(\rho^\sharp(e_{i_0,\dots, i_\ell})),
\end{align*}
where $\rho^\sharp:\Norm(\Z\Delt{n})\to \Norm(\Z\Delt{m})$ is the induced map by post-composition with $\rho$ in $\Delt{n}=\Del(.,[n])$. Thus, $\iota(\widetilde\rho(c))=(\text{pre-compostion with }\rho^\sharp)\circ (\iota(c))$ as maps $(\vC^\bu(\mc U^{[m]},\Ohol^\bu))^{even}_{closed}\to \Chain^-(\Norm(\Z\Delt{n}),\vC^\bu(\mc U, \Ohol^\bu)\ul)$, which shows that $\iota$ is a map of simplicial sets.

This completes the proof of the proposition.
\end{proof}

We can now state the main theorem of this section.
\begin{theorem} 
The following is a commutative diagram of simplicial sets,
\[
\xymatrix{ 
\CMan^{\Del^{op}}(\CN \mc U^{[\bu]},BG) \ar^{\quad \beta}[r] \ar_{\gamma}[d] & \Tot(\HVB(\CN\mc U)) \ar^{\Tot(\CS(\CN\mc U))}[dd]  
\\ 
(\vC^\bu(\mc U^{[\bu]},\Ohol^\bu))^{even}_{closed} \ar_{\iota}[d] &
\\
\DKSet(\vC^\bu(\mc U, \Ohol^\bu)\ul)  & \Tot(\OM(\CN\mc U)) \ar^{}[l] }
\]
\end{theorem}
\begin{proof}
We calculate $\iota(\gamma(h))$ for an $n$-simplex $h\in \CMan^{\Del^{op}}(\CN \mc U^{[n]},BG)$. By definition \ref{DEF:gamma}, for indices $i_0^{(j_0)},\dots,i_r^{(j_r)}\in I^{[n]}$,
\begin{equation*}
(\gamma_n(h))_{i_0^{(j_0)},\dots,i_r^{(j_r)}}=\tr\Big(h_{i_r^{(j_r)},i_0^{(j_{0})}}^{-1}\cdot \partial (h_{i_r^{(j_r)},i_{r-1}^{(j_{r-1})}})\cdot \ldots \cdot \partial (h_{i_2^{(j_2)},i_1^{(j_{1})}})\cdot \partial (h_{i_1^{(j_1)},i_0^{(j_{0})}}) \Big),
\end{equation*}
which is of total degree $2r$ (i.e., \v{C}ech degree $r$ and form degree $r$). By definition \ref{DEF:iota} this becomes the map $\iota(\gamma_n(h)):N(\Z\Delt{n})\to \vC^\bu(\mc U, \Ohol^\bu)\ul$, $e_{j_0,\dots, j_k}\mapsto (-1)^{\frac{k(k-1)}{2}}\cdot u^{\text{degree}/2}\cdot \int_{\Delta^k} \widetilde\lambda (\gamma_n(h))$, where $\lambda:[k]\to [n], \lambda(p)=j_p$ is as in definition \ref{DEF:iota}. 
Since $\widetilde\lambda:(\vC^\bu(\mc U^{[n]},\Ohol^\bu))^{even}_{closed}\to (\vC^\bu(\mc U^{[k]},\Ohol^\bu))^{even}_{closed}$ forgets all but the levels $j_0,\dots, j_k$, this in turn becomes, in component $i_0,\dots, i_q\in I$,
\begin{multline*}
\big(\iota(\gamma_n(h))(e_{j_0,\dots, j_k})\big)_{i_0,\dots, i_q}\\
= (-1)^{\frac{k(k-1)}{2}}\cdot u^{2(q+k)/2}\cdot \sum_{\big(a_0^{(b_0)},\dots, a_{q+k}^{(b_{q+k})}\big)\in J_k(i_0,\dots, i_q)} (-1)^{\stp_1+\dots+\stp_k} \cdot(\gamma_n(h))_{a_0^{(\lambda(b_0))},\dots, a_{q+k}^{(\lambda (b_{q+k}))}}\\
= (-1)^{\frac{k(k-1)}{2}}\cdot u^{q+k} \cdot \sum_{\big(a_0^{(b_0)},\dots, a_{q+k}^{(b_{q+k})}\big)\in J_k(i_0,\dots, i_q)} (-1)^{\stp_1+\dots+\stp_k} \cdot \tr\Big(h_{a_{q+k}^{(\lambda(b_{q+k}))},a_0^{(\lambda(b_{0}))}}^{-1}\hspace{1.4cm}\\
\cdot \partial (h_{a_{q+k}^{(\lambda(b_{q+k}))},a_{q+k-1}^{(\lambda(b_{q+k-1}))}})\cdot \ldots \cdot \partial (h_{a_1^{(\lambda(b_1))},a_0^{(\lambda(b_{0}))}}) \Big)
\end{multline*}
Note that by the definition of $J_k(i_0,\dots, i_q)$, that any adjacent indices $a_q^{(b_q)}, a_{q+1}^{(b_{q+1})}$ appearing in the above sum are either of the form $b_q=b_{q+1}$ or $a_q=a_{q+1}$. Thus, the only $h_{a_{q+1}^{(\lambda(b_{q+1}))}, a_q^{ (\lambda(b_q))}}$ that appear above are (in the notation of proposition \ref{PROP:Tot(Hol(U))n}) either $g^{(b_q)}_{a_q,a_{q+1}}$ or $f^{(b_{q+1},b_{q})}_{a_q}=f^{b_{q+1}}_{a_q}\circ\ldots \circ f^{b_{q}+1}_{a_q}:E^{(b_{q})}_{a_q}\to E^{(b_{q+1})}_{a_q}$.

Next, the outcome of going around the diagram from the theorem is the other way is described in proposition \ref{PROP:Tot-CS-n}, which we see to coincide with the above, since the $\beta$ map assigns the connections $\nabla=\partial$ to all bundles.
\end{proof}

\begin{example}
Consider the case of a $2$-simplex $h\in \CMan^{\Del^{op}}(\CN \mc U^{[2]},BG)$, where we assume again that $G=GL(n,\C)$. Then $\iota(\gamma(h))$ is a mapping $N(\Z\Delt{2})\to \vC^\bu(\mc U, \Ohol^\bu)\ul, e_{j_0,\dots, j_k}\mapsto c^{(j_0,\dots, j_k)}$
\[
\scalebox{1}{
\begin{tikzpicture}[scale=1]
\node (E0) at (0, 0) {}; \fill (E0) circle (2pt) node[left] {$c^{(0)}$};
\node (E1) at (2, 1.5) {}; \fill (E1) circle (2pt) node[above] {$c^{(1)}$};
\node (E2) at (4, 0) {}; \fill (E2) circle (2pt) node[right] {$c^{(2)}$};
\draw [->] (E0) -- node[left] {$c^{(0,1)}$} (E1);
\draw [->] (E1) -- node[above] {\quad\quad$c^{(1,2)}$} (E2);
\draw [->] (E0) -- node[below] {$c^{(0,2)}$} (E2);
\node (E3) at (2,0.6) {$c^{(0,1,2)}$}; 
\end{tikzpicture}
}
\]
where, for $j\in \{0,1,2\}$ and $(j',j'')\in \{(0,1),(0,2),(1,2)\}$,
\begin{align*}
(c^{(j)})_{i_0,\dots, i_q}=& u^q\cdot \tr\big((g^{(j)}_{i_q,i_0})^{-1}\partial(g^{(j)}_{i_q,i_{q-1}})\dots \partial(g^{(j)}_{i_1,i_{0}}) \big), \\
(c^{(j',j'')})_{i_0,\dots, i_q}=& u^{q+1}\cdot\sum_{0\leq \stp\leq q} (-1)^{\stp}\cdot\tr\Big[\big(g^{(j'')}_{i_q,i_{q-1}}\dots g^{(j'')}_{i_{\stp+1},i_{\stp}}f^{(j'',j')}_{i_\stp}g^{(j')}_{i_\stp,i_{\stp-1}}\dots g^{(j')}_{i_1,i_{0}}\big)^{-1}\\
&\hspace{0.8cm}\partial(g^{(j'')}_{i_q,i_{q-1}})\dots \partial(g^{(j'')}_{i_{\stp+1},i_{\stp}})\partial(f^{(j'',j')}_{i_\stp})\partial(g^{(j')}_{i_\stp,i_{\stp-1}})\dots \partial(g^{(j')}_{i_1,i_{0}}) \Big], \\
(c^{(0,1,2)})_{i_0,\dots, i_q}=& -u^{q+2}\cdot \sum_{0\leq \stp_1\leq \stp_2\leq q}(-1)^{\stp_1+\stp_2}\cdot\tr\Big[
\big(g^{(2)}_{i_q,i_{q-1}}\dots f^{(2,1)}_{i_{\stp_2}}\dots f^{(1,0)}_{i_{\stp_1}} \dots g^{(0)}_{i_1,i_{0}}\big)^{-1}\\
&\hspace{1.5cm}\partial(g^{(2)}_{i_q,i_{q-1}})\dots \partial(f^{(2,1)}_{i_{\stp_2}})\dots \partial(f^{(1,0)}_{i_{\stp_1}}) \dots \partial(g^{(0)}_{i_1,i_{0}})\Big].
\end{align*}
In the lowest case, this is interpreted as $(c^{(j)})_{i_0}=u^0\cdot \tr(id_{\C^n})=\dim(\C^n)=n$.
\end{example}

In the remainder of this section, we want to give an alternative description of $\gamma$ from \eqref{EQU:gamma-iota} via the universal Chern form on $BG$.

\begin{definition}
Assume that $G=GL(n,\C)$. Applying holomorphic forms to the simplicial manifold $BG$ from definition \ref{DEF:BG}, we obtain a cosimplicial non-negatively graded cochain complex $\Ohol^\bu(BG):\Del\to \Chain^+$ with $\Ohol^\bu(BG)_k= \Ohol^\bu(G^{\times k})$. There is a closed and even element $\Chern$ in the totalization, $\Chern\in \tot(\Ohol^\bu(BG))=\prod_\ell \Ohol^\ell(G^{\times \ell})[\ell]$, given by the following sequence of forms, 
\begin{equation}\label{DEF:Ch-form}
\Chern:=(n, \tr( g\partial (g^{-1})), \tr( g_1g_2 \partial (g_2^{-1})\partial (g_1^{-1})),\dots).
\end{equation}
\end{definition}

If $h\in \CMan^{\Del^{op}}(\CN \mc V,BG)$, there in an induced map $\Ohol^\bu(h):\Ohol^\bu(BG)\to \Ohol^\bu(\CN \mc V)$, and, thus a map on the total complex $\tot(\Ohol^\bu(h)):\tot(\Ohol^\bu(BG))\to \tot(\Ohol^\bu(\CN \mc V))\stackrel{\eqref{EQU:Tot(NU)=Cech}}\cong \vC^\bu(\mc V,\Ohol^\bu)$. Then, we claim the following.

\begin{proposition}
The map $\gamma:\CMan^{\Del^{op}}(\CN \mc U^{[\bu]},BG)\to (\vC^\bu(\mc U^{[\bu]},\Ohol^\bu))^{even}_{closed}$ from definition \ref{DEF:gamma} can be expressed via the Chern character $\Chern$ by
\[
\gamma_n(h)=\tot(\Ohol^\bu(h))(\Chern)\quad\quad \in  (\vC^\bu(\mc U^{[n]},\Ohol^\bu))^{even}_{closed}
\]
 \end{proposition}
 \begin{proof}
Just as in definition \ref{DEF:gamma}, let $h\in \CMan^{\Del^{op}}(\CN \mc V,BG)$ be given by transition functions $g^{\mc V}_{i,j}:V_{i,j}\to G\subseteq \C^{n,n}$ from lemma \ref{LEM:n-simplex-inCMAN(NU,BG)}. As shown in the proof of lemma \ref{LEM:n-simplex-inCMAN(NU,BG)}, these $g^{\mc V}_{i,j}$ induces all higher maps $V_{j_0,\dots, j_\ell}\to G^{\times \ell}$ via $V_{j_0,\dots, j_\ell} \ni x\mapsto (g^{\mc V}_{j_0,j_1}(x),\dots,g^{\mc V}_{j_{\ell-1},j_\ell}(x))$. Thus, under the pullback of $h$, the $\ell$-form $\tr(g_1\dots g_{\ell}\cdot \partial (g_\ell^{-1})\dots\partial (g_1^{-1}))\in \Ohol^\ell(G^{\times \ell})$, which is the $\ell$-component of $\Chern$, gets pulled back to
\[
\tr\Big(g_{j_0,j_{1}}^{\mc V}\dots g_{j_{\ell-1},j_{\ell}}^{\mc V}\cdot \partial (g^{\mc V}_{j_\ell,j_{\ell-1}})\dots \partial (g^{\mc V}_{j_1,j_{0}}) \Big) \quad \in \Ohol^\ell(V_{i_0,\dots, i_\ell}).
\]
Now, since the $g^{\mc V}_{i,j}$ satisfy the cocycle condition $g^{\mc V}_{j_0,j_1}...g^{\mc V}_{j_{\ell-1},j_\ell}=g^{\mc V}_{j_0,j_\ell}$ by lemma \ref{LEM:n-simplex-inCMAN(NU,BG)}, this is precisely the expression obtained for the definition of $\gamma_{\mc V}(h)$ in equation \eqref{DEF:gamma_V(h)}.

Applying this to the covers $\mc V=\mc U^{[n]}$ for all $n$ yields the claim of the proposition.
 \end{proof}

\section{Holomorphic vector bundles with group action}\label{SEC:Group-action}

We give an application of the previous sections by considering a complex manifold with group action. We first generalize definition \ref{DEF:BG}.
\begin{definition}\label{DEF:M/G}
Let $M$ be a complex manifold, and $G$ be a (possibly discrete) complex Lie group together with a right action on $M$. We define a simplicial complex manifold $[M/G]\in \Ob(\CMan^{\Del^{op}})$ by setting the $n$-simplices to be $[M/G]_n=  M \times G^{\times n}$:
\[\begin{tikzcd}
  M \arrow[from=r,  transform canvas={yshift=-0.8ex}]
         \arrow[from=r, transform canvas={yshift=0.8ex}]
          \arrow[r]
    & M \times G
        		 \arrow[r,  transform canvas={yshift=-1.6ex}]
        \arrow[from=r, transform canvas={yshift=-0.8ex}]
        		 \arrow[r]
        \arrow[from=r, transform canvas={yshift=0.8ex}]
            	  \arrow[r,  transform canvas={yshift=1.6ex}]
& M \times G \times G 
 \arrow[r, swap, "s_i",  transform canvas={yshift=-0.8ex}]
         \arrow[from=r, swap, "d_j",  transform canvas={yshift=0.8ex}]
        & \cdots
\end{tikzcd}\]
The face maps $d_j:   M\times G^{\times n} \to    M \times G^{\times (n-1)}$ for $0< j<n$ are $d_j( x, g_1,\dots, g_n)=( x, g_0,\dots, g_j\cdot g_{j+1},\dots, g_n,)$, while $d_n(x, g_1,\dots, g_n)=(x, g_1,\dots, g_{n-1} )$ and $d_0(x, g_1,\dots, g_n)=(x \cdot g_1, g_2 , \dots,  g_n )$. The degeneracies $s_j:M  \times G^{\times (n-1)}  \to M \times G^{\times n}   $ are given by $s_j(x, g_1,\dots, g_{n-1})=(x, g_1,\dots, g_j,1,g_{j+1},\dots, g_{n-1} )$, where $0\leq j\leq n-1$.
\end{definition}

By section \ref{SEC:Prestacks}, $\HVB( [M/G]):\Del\stackrel{[M/G]} \longrightarrow \CMan^{op}\stackrel{\HVB} \longrightarrow \lsSet$ and $\OM( [M/G]):\Del\stackrel{[M/G]} \longrightarrow \CMan^{op}\stackrel{\OM} \longrightarrow \sSet$ are cosimplicial simplicial sets, and $\CS([M/G]):\HVB([M/G])\to \OM([M/G])$ is a map of cosimplicial simplicial sets. By applying the totalization, we obtain an induced map as follows.

Note that the above gives rise to a map of simplicial sets
\begin{equation}\label{EQU:Tot(CS(M/G))}
\Tot(\CS([M/G])):\Tot(\HVB([M/G]))\to \Tot(\OM([M/G])).
\end{equation}

In order to interpret the above map, we briefly review the notion of a $G$-equivariant bundle.
\begin{definition}
Given a $G$-manifold, $M$, with action given by $\rho : M \times G \to M$, a bundle $E \xrightarrow{\pi} M$ is a $G$-equivariant bundle over $M$ if there is a $G$-action on $E$, $\varphi: E \times G \to E$, such that the diagram
\[ \begin{tikzcd}
E \times G \arrow[r, "\varphi"] \arrow[d, "\pi \times id"] & E \arrow[d, "\pi"]\\
M \times G \arrow[r, "\rho"] & M
\end{tikzcd} \]
commutes.
\end{definition}

With this, we can now describe $\Tot(\HVB([M/G]))$ more explicitly.
\begin{proposition}\label{PROP:Tot(HVB(M/G))}
The simplices of $\Tot(\HVB([M/G]))$ have the following interpretation.
\begin{enumerate}
\item
\label{PROP:Tot(HVB(M/G))_0}
A $0$-cell in $\Tot(\HVB([M/G]))$ consists precisely of a $G$-equivariant bundle, $E$, with connection, $\nabla$, where $\nabla$ is not required to satisfy any condition with respect to the $G$-action.
\item
\label{PROP:Tot(HVB(M/G))_n}
An $n$-cell in $\Tot(\HVB([M/G]))$ consists precisely of a sequence of $G$-equivariant bundles, $E^{(0)},\dots, E^{(n)}$,  and $G$-equivariant maps, $\al_0,\dots, \al_{n-1}$,
\[ E^{(0)} \xrightarrow{\alpha_0} E^{(1)} \xrightarrow{\alpha_1} \dots \xrightarrow{\alpha_{n-1}} E^{(n)}\] 
where each bundle $E_i \to M$ has a connection $\nabla_i$, which are not required to satisfy any conditions with respect to the $G$-action or the bundle maps.
\end{enumerate}
\end{proposition}
\begin{proof}
First we prove part \eqref{PROP:Tot(HVB(M/G))_0}. 
Similar to the proof of Proposition \ref{PROP:Tot(Hol(U))0}, a 0-simplex, $\omega$, in $\Tot(\HVB([M/G]))$ is given by a sequence of simplicial set maps, $\Delt{\ell}\times \Delt{0} \xrightarrow{\omega_{\ell}} \HVB([M/G])$ for $\ell=0,1,2,\dots$.  Note then, for $\ell = 0$, we have the image under $\omega_{0}$ of a vertex which is given by a vector bundle, $E_0^0 := E$, with connection, $\nabla$, over $M$. Over $M \times G$, $\omega_{1}$ gives a pair of bundles, 
$$\begin{tikzcd}
E^1_0\arrow[rr, "\phi"]  \arrow[dr, lightgray]&  &  E^1_1 \arrow[dl, lightgray]\\
&M \times G & 
\end{tikzcd}$$
satisfying $E^1_1 = d_0^*(E) $ and $E^1_0 = d_1^*(E) = E \times G$ with approprately induced pullback connections.  Over the point $( m, g) \in M \times G$, this map sends the fibers, $$\phi_{(m,g)}: \left( \left( E^1_0\right)_{(m,g)} =  E_m \right)  \mapsto   \left(  \left( E^1_1\right)_{(m, g)}=E_{ m \cdot g}  \right) $$
From this data, we can define a map of manifolds for each $g \in G$, $\varphi_g : E \to E$, defined by $\varphi_g|_{(E\times G)_{(m,g)}}:= \phi_{(m,g)}: E_m \mapsto E_{ m \cdot g}$.  Furthermore, for each $g \in G$, this map commutes with $\pi$, i.e. $\pi \circ \varphi_g  = R_g \circ \pi$, where $R_g=\rho(.,g)$ is the right multiplication of $G$ on $M$.  Note, however, that we still have to check that this a $G$-equivariant map, i.e., that $\varphi_{g'} \circ \varphi_{g} = \varphi_{g \cdot g'}$.  This relation requires higher simplicial data.  Our sequence of simplicial set maps, $\omega$, also provides a $2$-simplex given by $\omega_2$:
$$\begin{tikzcd}
  E^2_1\arrow[from=rr, "\phi_{1,0}", swap]\arrow[ddr, lightgray] & & E^2_0\arrow[ddl, lightgray]\\
  &  E^2_2  \arrow[from=ul, "\phi_{2,1}"]    \arrow[from=ur, swap, "\phi_{2,0}"] \arrow[d, lightgray]\\
&  M \times G \times G& \\
\end{tikzcd}$$
Here, by the definition of totalization, $E^2_0= d_1^*\left( d_1^* E \right)  = d_2^*\left( d_1^* E \right)$, $E^2_1= d_0^*\left( d_1^* E \right)  = d_2^*\left( d_0^* E \right)$, and $E^2_2= d_1^*\left( d_0^* E \right)  = d_0^*\left( d_0^* E \right)$.  Similarly, note that $\phi_{1,0} = d_2^* (\phi)$, $\phi_{2,0} = d_1^* (\phi)$, and $\phi_{2,1} = d_0^* (\phi)$.  Since the above diagram commutes, the proof is concluded after unpacking the equation given by the above commutative triangle,
$$\phi_{2,1} \circ \phi_{1,0} = \phi_{2,0}.$$
Since $E^2_0= d_1^*\left(  E \times G \right)  = \left( d_1 \circ d_1 \right)^*(E)$, the composition of maps governing the pullback is given by $d_1 \circ d_1: (m, g, g') \mapsto m$.  Similarly, $E^2_1= \left( d_0 \circ d_2 \right)^*(E)$ is given by $d_0 \circ d_2: ( m, g, g') \mapsto (m\cdot g)$ and $E^2_2= \left( d_0 \circ d_0 \right)^*(E)$ is given by $d_0 \circ d_0: (m, g, g') \mapsto ( m \cdot  g \cdot g' )$.  Thus, the maps act accordingly on the fibers: $\left(\phi_{1,0} \right)_{(m, g, g')}: E_{m} \mapsto E_{m \cdot g }$, $\left(\phi_{2,1} \right)_{(m, g, g')} : E_{m \cdot g} \mapsto E_{ m \cdot g \cdot g' }$, and $\left(\phi_{2,0} \right)_{(m, g, g')} : E_{ m} \mapsto E_{m \cdot g \cdot g'  }$.  Therefore, the above commutative diagram shows that $\varphi_g' \circ \varphi_{g} = \varphi_{g \cdot g'}$, which concludes the proof that $E$ is a $G$-equivariant bundle.

Now we turn to part \eqref{PROP:Tot(HVB(M/G))_n}. 
Similar to the proof of Proposition \ref{PROP:Tot(Hol(U))n}, an $n$-simplex is given by a sequence of simplicial set maps $\Delt{\ell}\times \Delt{n} \xrightarrow{\omega_{\ell}} \HVB([M/G]_\ell)$ for $\ell=0,1,2,\dots$ satisfying certain conditions.  The $0$-simplices of this $n$-simplex are precisely the data for a $G$-equivariant bundle, $(E^{(i)}, \varphi^{(i)})$, for $i=0,\dots, n$, as described in part \eqref{PROP:Tot(HVB(M/G))_0}.  Over the base manifold $M$, we can write the image of the maximal non-degernate $n$-simplex of $\Delt{0}\times \Delt{n}$, under the map $\omega_0$, in $\HVB([M/G]_0)$ as a sequence of $G$-equivariant bundles,
$$\begin{tikzcd}
  E^{(0)} \arrow[drrr, lightgray] \arrow[rr, "\alpha_0"]& &  E^{(1)}\arrow[rr, "\alpha_1"]  \arrow[dr, lightgray]  & &  \dots \arrow[rr, "\alpha_{n-1}"]  \arrow[dl, lightgray] & &  E^{(n)} \arrow[dlll, lightgray]  \\
  &  & &M & & & 
\end{tikzcd}$$
such that each $\alpha_i: E^{(i)} \to E^{(i+1)}$, as a morphism in $\HVB(M)$ maps fibers $\left(E^{(i)} \right)_m$ to $\left(E^{(i+1)} \right)_m$.  To see that these maps $\alpha_i$ respect the $G$-actions $\varphi^{(i)}$, we note that $\omega_1$ offers us the following commutative diagram, 
$$\begin{tikzcd}
  d_0^*(E^{(0)}) \arrow[d, "\phi^{(0)}"]\arrow[ddrrr, lightgray] \arrow[rr, " d_0^*(\alpha_0)"]& &   d_0^*(E^{(1)})  \arrow[d, "\phi^{(1)}"] \arrow[rr, " d_0^*(\alpha_1)"]  \arrow[ddr, lightgray]  & &  \dots \arrow[rr, " d_0^*(\alpha_{n-1})"]  \arrow[ddl, lightgray] & &   d_0^*(E^{(n)})  \arrow[d, "\phi^{(n)}"] \arrow[ddlll, lightgray]  \\
     d_1^*(E^{(0)}) \arrow[drrr, lightgray] \arrow[rr, "d_1^*(\alpha_0)"]& &  d_1^*(E^{(1)})\arrow[rr, "d_1^*(\alpha_1)"]  \arrow[dr, lightgray]  & &  \dots \arrow[rr, "d_1^*(\alpha_{n-1})"]  \arrow[dl, lightgray] & &  d_1^*(E^{(n)}) \arrow[dlll, lightgray]  \\
  &  & &M  \times G& & & 
\end{tikzcd}$$
which on the fiber over a point $( m, g) \in  M \times G $, induces a commutative diagram of maps of fibers, 
$$\begin{tikzcd}
\left( E^{(i)}\right)_m \arrow[d, "(\varphi^{(i)})_g"]   \arrow[r, "\alpha_i"] & \left( E^{(i+1)}\right)_m \arrow[d, "(\varphi^{(i+1)})_g"] \\
\left( E^{(i)}\right)_{m \cdot g} \arrow[r, "\alpha_i"] & \left( E^{(i+1)}\right)_{ m \cdot g}
\end{tikzcd}$$
and thus each $\alpha_i$ is a $G$-equivariant map of $G$-equivariant bundles.  Since there are no higher relations, this concludes the proof.
\end{proof}

For a $G$-equivariant bundle $E \xrightarrow{\pi} M$, where $\varphi:E\times G\to E$ denotes the lift of the $G$-action $\rho:M\times G\to M$ on $M$ on the base, we note that there is an induced map of bundles, $\phi:E \times G \to \rho^*(E)$ over $M\times G$, which we may interpret as a section of $Hom(E \times G , \rho^*(E))$ over $M\times G$, i.e., as a $0$-form $\phi\in \Ohol^0(M\times G, Hom(E \times G , \rho^*(E)))$. Now assume furthermore, that $E\to M$ is equipped with a holomorphic connection $\nabla$. Pulling back $\nabla$ under the projection $pr_1:M\times G\to M$ gives a connection $\nabla^{E\times G}$ on $E\times G$, while pulling back $\nabla$ under $\rho:M\times G\to M$ gives an induced connection $\nabla^{\rho^*(E)}$ on $\rho^*(E)$, and thus we get an induced connection $\nabla^{Hom(E \times G , \rho^*(E))}$ of $Hom(E \times G , \rho^*(E))$, which we simply denote by $\nabla$ again. Thus, we may apply $\nabla$ to $\phi$, which is given by pre- and post-composing with $\nabla^{E\times G}$ and $\nabla^{\rho^*(E)}$, respectively, 
\begin{equation*}
\nabla( \phi)  :=  \nabla^{\rho^*(E)} \circ \phi - \phi \circ  \nabla^{E\times G}  \quad\in \Ohol^1(M\times G, Hom(E \times G , \rho^*(E))).
\end{equation*}
\begin{definition}\label{DEF:EquivNabla}
A connection $\nabla$ on a $G$-equivariant bundle $(E, M, \pi, \rho, \varphi)$, is $G$-invariant if  $\nabla( \phi) = 0$.
\end{definition}

The next corollary states that we can use the map $\Tot(\CS([M/G]))$ from equation \eqref{EQU:Tot(CS(M/G))} as a measure for the connection $\nabla$ to be $G$-invariant.

\begin{corollary}
Let $(E, M, \pi, \rho, \varphi)$ be a $G$-equivariant bundle with connection $\nabla$, which, by proposition \ref{PROP:Tot(HVB(M/G))}\eqref{PROP:Tot(HVB(M/G))_0}, we may interpret as a $0$-simplex in $\Tot(\HVB([M/G]))_0$. If the connection $\nabla$ is $G$-invariant, then $\Tot(\CS([M/G]))$ applied to this is zero in all positive holomorphic form degrees.
\end{corollary}
\begin{proof}
Since $\nabla(\phi)=0$, it follows that $\tr(\phi^{-1} \nabla(\phi))\cdot u=0$, which is the form-degree $1$ part of $\Tot(\CS([M/G]))_0$. Similarly, the higher form degrees vanish; for example, in the notation of the proof of proposition \ref{PROP:Tot(HVB(M/G))}\eqref{PROP:Tot(HVB(M/G))_0}, the form-degree $2$ part is $\tr(\phi_{2,0}^{-1} \nabla_{2,1}(\phi_{2,1})\nabla_{1,0}(\phi_{1,0}))\cdot u^2=\tr(\phi_{2,0}^{-1} d_0^*(\nabla(\phi))d_2^*(\nabla(\phi)))\cdot u^2=0$.
\end{proof}

$\Tot(\CS([M/G]))$ measures the extent to which a holomorphic connection is $G$-invariant.

\appendix

\section{Small and large simplicial sets}

In this appendix we recall some notation of small and large simplicial set, see e.g. \cite{GJ} and \cite{J}.

\begin{definition}\label{DEF:simplicial}
Let $\Del$ be the category whose objects are $[n]=\{0,\dots,n\}$ for $n=0,1,2,\dots$, and morphisms $\rho:[n]\to [m]$ are non-decreasing maps. We have face maps $\delta_j:[n-1]\to [n]$ skipping $j$ (for $j=0,\dots, n$), and degeneracies $\sigma_j:[n]\to [n-1]$ repeating $j$ (for $j=0,\dots, n-1$). If $\mc C$ is a category, then a simplicial (respectively cosimplicial) object in $\mc C$ is a functor $X=X_\bu:\Del^{op}\to \mc C$ (respectively a functor $X=X^\bu:\Del\to \mc C$), where we denote $X_n:=X([n])$ (respectively $X^n:=X([n])$), as usual.

For example, $\Delt{n}:\Del^{op}\to \Set$ is the simplicial set of the standard $n$-simplex, given by setting $\Delt{n}_k:=\Del([k],[n])$. Moreover, $\Delt{\bu}:\Del\to \sSet$ is a cosimplicial simplicial set.
\end{definition}

We remark on the size of the categories that we study. In this paper we will consider small, large, and extra large categories, as well as small and large simplicial sets. We recall some notation from \cite{J}.
\begin{definition}
Fix three Grothendieck universes ${\bf U_s}$, ${\bf U_l}$, ${\bf U_{el}}$, with ${\bf U_s}\in {\bf U_l}\in {\bf U_{el}}$, whose elements are called small sets, large sets, and extra-large sets, respectively. In this paper, we assume that certain sets, such as the underlying set of a complex manifold or a holomorphic vector bundle, are elements of ${\bf U_s}$.

A category $\mc C$ is called small (respectively large, or extra-large), if both the set of objects $\Ob(\mc C)$ and the set of morphism $\Mor(\mc C)=\coprod_{E,E'\in \Ob(\mc C)}\mc C(E,E')$ are small sets (respectively large sets, or extra-large sets). An example of a small category is the simplicial category $\Del$ from above. Examples of large categories are the category $\Set$ of small sets, the category $\Chain$ of cochain complexes, the category $\CMan$ of complex manifolds, and the category $\Cat$ of small categories. Examples of extra-large categories are the category $\lSet$ of large sets, and the category $\lCat$ of large categories.

A simplicial object $X:\Del^{op}\to \mc C$ in a category $\mc C$ is called small (respectively large), if all $X_n=X([n])$ are small sets (respectively large sets), and similarly for a cosimplicial object $X:\Del\to \mc C$. Denote by $\sSet$ the category of all small simplicial sets, which is a large category. Denote by $\lsSet$ the category of large simplicial sets, which is an extra-large category. The nerve $\Nerve(\mc C)$ of a category is the simplicial set whose set of $0$-simplices $\Nerve(\mc C)_0=\Ob(\mc C)$ are the objects of $\mc C$, and $k$-simplices for $k\geq 1$ are $k$ composable morphisms $E_0\stackrel {f_1}\to E_1\stackrel {f_2}\to \dots \stackrel {f_k}\to E_k$, i.e., $\Nerve(\mc C)_k= \coprod_{E_0,\dots, E_k\in \Ob(\mc C)} \mc C(E_0,E_1)\times \dots\times \mc C(E_{k-1},E_k)$. If $\mc C$ is a small category (respectively large category), then $\Nerve(\mc C)$ is a small simplicial set (respectively large simplicial set). Moreover, the nerve is a functor $\Nerve:\Cat\to \sSet$ or $\Nerve:\lCat\to \lsSet$.

In section \ref{SEC:Totalization}, we will consider cosimplicial simplicial sets, i.e.,  functors of the form $X:\Del\to \sSet$, or, more generally, functors $X:\Del\to \lsSet$, where, for the latter, $X^n:\Del^{op}\to \lSet$, so that $(X^n)_m$ is a large set.

Sometimes, we may not comment on the size and just refer to categories or simplicial sets without any size reference. Note, however, that all structures in this paper are, in particular, large structures.
\end{definition}

\section{Cochain complexes and the Dold-Kan functor}

We will frequently consider cochain complexes in this paper that are concentrated in non-negative or non-positive degrees. The next definitions provide more details.

\begin{definition}
Denote by $\Chain$ the category of $\Z$-graded cochain complexes.  So an object in $\Chain$ is a pair $(C^\bullet, d)$, with $d:C^\bullet \rightarrow C^{\bullet +1}$.  Note that if $(C_\bullet, \partial)$ is a graded \emph{chain} complex, i.e.,  $\partial: C_\bullet \rightarrow C_{\bullet-1}$, then we think of it as an object of $C^\bu\in \Chain$ by setting $C^{n}:=C_{-n}$ in degree $n$.

Let $\Chain^+$ be the category of non-negatively graded cochain complexes and $\Chain^-$ be the category of non-positively graded cochain complexes.
\end{definition}

\begin{example} \label{REM:NZDeltak}
If $A:\Del^{op}\to \Ab$ is a simplicial abelian group, then define the negatively graded chain complex $\Norm(A)\in \Chain^-$ to be the normalized chains of $A$, i.e., in degree $-k\leq 0$ we set it to be $A_k$ modulo degeneracies, $$N(A)_{-k}:=A_k/(\sum_{j=0}^{k-1} \image (s_j:A_{k-1}\to A_{k})).$$ The differential $d:N(A)_{-k}\to N(A)_{-k+1}$ is induced by the alternating sum of the face maps $\sum_{j=0}^{k} (-1)^j d_j :A_{k}\to A_{k-1}$.

In particular, the free abelian group of the standard $n$-simplex $\Z\Delt{n}$ is a simplicial abelian group. The cochain complex $\Norm(\Z\Delt{n})\in \Chain^-$ has the following explicit representation. In degree $-\ell\leq 0$, $\Norm(\Z\Delt{n})_{-\ell}$ is the free abelian group with generators $e_{i_0,\dots, i_\ell}$ for all $0\leq i_0<\dots<i_\ell\leq n$ corresponding to the non-degenerate $\ell$-simplex $i:[\ell]\to[n],i:k\mapsto i_k$ of $\Delt{n}$. The generator $e_{i_0,\dots, i_\ell}$ can be thought of labeling an $\ell$-cell of the topological $n$-simplex $|\Delt{n}|$. The differential $d:\Norm(\Z\Delt{n})_{-\ell}\to \Norm(\Z\Delt{n})_{-\ell+1}$ is given by $d(e_{i_0,\dots,i_\ell})=\sum_{j=0}^\ell (-1)^j e_{i_0,\dots,\widehat{i_j},\dots, i_\ell}$.

Similarly, we can describe generators of $\Norm(\Z(\Delt{n}\times\Delt{m}))$ as follows. Let $j:[p]\to [n], j:k\mapsto j_k$ be a non-degenerate $p$-simplex of $\Delt{n}$, and $i:[q]\to [m], i:k\mapsto i_k$ be a non-degenerate $q$-simplex of $\Delt{m}$. In order to obtain a non-degenerate $r$-simplex of $\Delt{n}\times \Delt{m}$ with $p\leq r$ and $q\leq r$, choose numbers $0\leq \mu_1<\mu_2<\dots <\mu_{r-q}\leq r-1$ and $0\leq \nu_1<\nu_2<\dots <\nu_{r-p}\leq r-1$ with $\{\mu_1,\dots,\mu_{r-q}\}\cap \{\nu_1,\dots,\nu_{r-p}\}=\varnothing$. Then $(j\circ \sigma_{\nu_1}\circ\dots \circ \sigma_{\nu_{r-p}}:[r]\to [n],i\circ \sigma_{\mu_1}\circ\dots \circ \sigma_{\mu_{r-q}}:[r]\to [m])$ is a non-degenerate $r$-simplex of $\Delt{n}\times\Delt{m}$, and we denote the corresponding generator of $\Norm (\Z(\Delt{n}\times\Delt{m}))$ by $(s_{\nu_{r-p}}\dots s_{\nu_{1}}(e_{j_0,\dots,j_p}), s_{\mu_{r-q}}\dots s_{\mu_{1}}(e_{i_0,\dots,i_q}))$. In particular, for $r=p+q$, $\{\mu_1,\dots, \mu_p,\nu_1,\dots,\nu_q\}$ determines a permutation of $\{0,\dots, p+q-1\}$, and those $(\mu,\nu)$ are then called $(p,q)$-shuffles. Denote the sign of this permutation by $\sgn(\mu,\nu)$, which is calculated as 
\begin{equation}\label{EQU:sgn(mu,nu)}
\sgn(\mu,\nu)=(-1)^{\mu_1+(\mu_2-1)+(\mu_3-2)+\dots+(\mu_p-p+1)}.
\end{equation}
\end{example}

The Dold-Kan construction makes the normalization into an equivalence of categories; see for example \cite{D}, \cite{K}, \cite[section 2]{G}, or \cite[chapter III.2]{GJ}.
\begin{theorem}[Dold-Kan]\label{THM:Dold-Kan}
Let $\sAb$ be the category of simplicial abelian groups and $\Chain^-$ the category of non-positively graded cochain complexes. There is an adjoint pair of functors $\Norm\dashv \DK$, which is an equivalence, where $\Norm:\sAb\to \Chain^-$ is the normalization, and $\DK:\Chain^-\to \sAb$. 

The functor $\DK$ can be defined as follows. For a non-positively graded chain complex $C^{\bu\leq 0}\in \Ob(\Chain^-)$ define $DK(C^{\bu\leq 0})\in \sAb$ to be the simplicial abelian group, which in simplicial degree $k$ consists of the cochain maps from normalized cells of the standard simplex $\Delt{k}$ to $C^{\bu\leq 0}$, i.e., we set 
\begin{equation}\label{EQU:DK-def}
DK(C^{\bu\leq 0})_k:={\Chain^-}(N(\Z\Delt{k}),C^{\bu\leq 0}).
\end{equation}
\end{theorem}

The Dold-Kan functor $\DK:\Chain^-\to \sAb$ can be composed with the forgetful map $\mc F:\sAb\to \sSet$, which we denote by $\DKSet=\mc F\circ \DK:\Chain^-\to \sSet$.

We will need to use functors between $\Chain$, $\Chain^-$, and $\Chain^+$.  

\begin{definition}\label{DEF:q-functor}
The truncation functor $t$ is a functor $t:\Chain \rightarrow \Chain^+$ defined by $t(C^\bullet) = C^{\bullet \geq 0}$.

The quotient functor $q$ is a functor $q: \Chain \rightarrow \Chain^-$ defined by $q(C^\bullet) = {C^\bullet}/ C^{\bullet \geq 0}$.
\end{definition}

\begin{definition}\label{DEF:Chain-Cat}
There is an adjoint pair of functors 
\begin{eqnarray}
\trunc: \Chain^- \leftrightarrows \Chain^+: \quot
\end{eqnarray}   

To define $\trunc$, let $\Z[v]$ be the cochain complex of polynomials in a formal variable $v$ of degree $|v|=+2$ with differential $d = 0$.  For an object $C=C^{\bullet\leq 0} \in \Ob(\Chain^-)$, the tensor product $C \otimes \Z[v]$ is a $\Z$-graded cochain complex, and then $\trunc(C)\in \Chain^+$ is defined as the truncation of $C \otimes \Z[v]$ to non-negative degrees, 
\[
\trunc(C)^{\bu\geq 0}:=(C \otimes \Z[v])^{\bu\geq 0}.
\]
Thus, elements of $\trunc(C^{\bullet \leq 0})$ in even degree $2k\geq 0$ are polynomials $c_{0}v^k+c_{-2}v^{k+1}+\dots$, and elements in odd degree $2k+1\geq 0$ are polynomials $c_{-1}v^{k+1}+c_{-3}v^{k+2}+\dots$, where each $c_j\in C^j$.

To define $\quot$, let $\Z[u]$ be the cochain complex of polynomials in a formal variable $u$ of degree $|u|=-2$ with differential $d = 0$.  For an object $C=C^{\bullet\geq 0} \in \Ob(\Chain^+)$, the tensor product $C \otimes \Z[u]$ is a $\Z$-graded cochain complex, and $\quot(C)\in \Chain^-$ is the quotient of $C \otimes \Z[u]$ by all the positively graded components of $C \otimes \Z[u]$,
\[
\quot(C)^{\bu\leq 0}:=(C \otimes \Z[u])^\bu/(C\otimes \Z[u])^{\bu> 0}.
\]
We sometimes abuse notation and simply write $\quot(C)=C[u]^{\bu\leq 0}$. Thus, elements of $\quot(C^{\bullet \geq 0})$ in even degree $-2k\leq 0$ are represented by polynomials $c_{0}u^k+c_{2}u^{k+1}+\dots$, and elements in odd degree $-2k-1\leq 0$ are polynomials $c_{1}u^{k+1}+c_{3}u^{k+2}+\dots$, where each $c_j\in C^j$.
\end{definition}

\section{Simplicial Model Categories}

To take the totalization of a cosimplicial object, the category $\mc M$ is assumed to be a simplicial model category.  This means that $\mc M$ is a model category enriched over simplicial sets.  That is, given any two objects $X, Y$ in $\mc M$, there is a simplicial set, denoted $Map(X,Y)$ with $Map(X,Y)_0 = \mc M(X,Y)$, and a composition map $Map(X,Y) \times Map(Y,Z) \rightarrow Map(X,Z) $, satisfying the usual associativity axioms.  Given $X \in \mc M$ and a simplicial set $K_\bullet$, we also need to define objects $X \otimes K_\bullet$ and $X^{K_\bullet}$ in $\mc M$, satisfying some compatibility relations with the model structure and with the enrichment over simplicial sets.  The reader can find the axioms in \cite[chapter 9.1]{H}.  

\begin{example}[\cite{H}, ex. 9.1.13]\label{example:simplicialsets}
The category $\sSet$ of simplicial sets is a simplicial model category, with the following simplicial model category structure.
\begin{enumerate}
\item{$f: X_\bullet \rightarrow Y_\bullet$ is a weak equivalence if the induced map on the geometric realization, $|f|:|X_\bullet| \rightarrow |Y_\bullet|$, is a quasi-isomorphism, i.e., it induces isomorphisms between the homotopy groups of $|X_\bullet|$ and $|Y_\bullet|$,}
\item{$f:X_\bullet \rightarrow Y_\bullet$ is a fibration if it is a Kan fibration,}
\item{$f: X_\bullet \rightarrow Y_\bullet$ is a cofibration if it has the left lifting property with respect to trivial fibrations,}
\item{for simplicial sets $X_\bullet$ and $Y_\bullet$, let $Map(X_\bullet,Y_\bullet)$ be the simplicial set whose $n$-simplices are given by simplicial set maps $X_\bu \times \Delt{n}_\bu \rightarrow Y_\bu$,}
\item{for a simplicial set $X_\bullet$ and simplicial set $K_\bullet$, let $X_\bullet \otimes K_\bullet$ be the simplicial set $X_\bullet \times K_\bullet$}
\item{for a simplicial set $X_\bullet$ and simplicial set $K_\bullet$, let $X_\bullet^{K_\bullet}$ be the simplicial set $Map(K_\bullet, X_\bullet)$.}
\end{enumerate}
\end{example}

\begin{example} \label{example:simplicialabeliangroups}The category $\sAb$ of simplicial abelian groups is a simplicial model category with the following simplicial model category structure (cf. \cite[Ch. III, Proposition 2.11]{GJ}).
\begin{enumerate}
\item{$f: A_\bullet \rightarrow B_\bullet$ is a weak equivalence if the induced map on geometric realization $|f|: |A_\bullet| \rightarrow |B_\bullet|$ is a quasi-isomorphism,}
\item{$f: A_\bullet \rightarrow B_\bullet$ is a  fibration if it is a Kan fibration,}
\item{$f: A_\bullet \rightarrow B_\bullet$ is a cofibration if it has the left lifting property with respect to trivial fibrations,}
\item{for simplicial abelian groups $A_\bullet$ and $B_\bullet$, let $Map(A_\bullet, B_\bullet)$ be the simplicial set whose $n$-simplices are given by simplicial abelian group maps $A_\bullet \otimes_\Z \Z\Delt{n}_\bullet \rightarrow B_\bullet$,}
\item{for a simplicial abelian group $A_\bullet$ and simplicial set $K_\bullet$, $A_\bullet \otimes K_\bullet$ is $A_\bullet \otimes_\Z \Z K_\bullet$, where $\Z K_\bullet$ is the free simplicial abelian group on $K_\bullet$,}
\item{for a simplicial abelian group $A_\bullet$ and simplicial set $K_\bullet$, $A_\bullet^{K_\bullet}$ is the simplicial set $Map( \Z K_\bullet, A_\bullet)$ defined in (4), with the group structure inherited by the group structure on $A_\bullet$. }
\end{enumerate}
\end{example}

The Dold-Kan correspondence can be used to transfer the simplicial model category structure on $\sAb$ to $\Chain^-$.  To define the simplicial model category structure on $\Chain^-$, let $Hom^\bullet(C^\bullet, D^\bullet)$ be the cochain complex of graded maps between cochain complexes $C^\bullet$ and $D^\bullet$.  An element of degree $n$ is a graded map $f:C^\bullet \rightarrow D^{\bullet +n}$, and $d( f) = d_D\circ  f - (-1)^n f \circ d_C$ is an element in degree $n+1$.  

\begin{example}\label{example:negativechaincomplex}
We use the following simplicial model category structure on $\Chain^-$:
\begin{enumerate}
\item{$f: C^\bullet \rightarrow D^\bullet$ is a weak equivalence if $f$ induces an isomorphism on cohomology,}
\item{$f:C^\bullet \rightarrow D^\bullet$ is a fibration if it is a degree-wise surjection for $n <0$,}
\item{$f:C^\bullet \rightarrow D^\bullet$ is a cofibration if it has the left lifting property for every trivial fibration,}
\item{for $C^\bullet$ and $D^\bullet$ in $\Chain^-$, let the $n$-simplices of the simplicial set $Map(C^\bullet, D^\bullet)$ be $Map(C^\bullet, D^\bullet)_n := \Chain^-(C^\bullet \otimes \Norm (\Z\Delt{n}) , D^\bullet)$,}
\item{for $C^\bullet$ in $\Chain^-$ and $K$ a simplicial set, let $C^\bullet \otimes K$ be the cochain complex $C^\bullet \otimes \Norm(\Z K)$,}
\item\label{ITEM:6-}{for $C^\bullet$ in $\Chain^-$ and $K$ a simplicial set, let $(C^\bullet)^K$ be the cochain complex $q(Hom^\bullet(\Norm(\Z K), C^\bullet))$, where $q$ is the quotient functor from definition \ref{DEF:q-functor}.} 
\end{enumerate}
\end{example}

\section{Totalization of a cosimplicial object in a simplicial model category}\label{APPX:totalization}

\begin{definition}[\cite{H}, Definition 18.6.3]\label{DEF:Totalization}
Let $X:\Del\to \mc C$ be a cosimplicial object in a simplicial model category $\mc C$, i.e., each $X^n\in \Ob(\mc C)$. (For example, this applies to $\mc C=\sSet$, in which case $X$ is a cosimplicial simplicial set.) Then, the totalization $\Tot (X)$ of $X$ is defined as the object in $\mc C$, which is the equalizer of the maps
\begin{equation}\label{DEF:Tot}
\prod_{[\ell]\in \Ob(\Del)} (X^\ell)^{\Delt{\ell}} \stackrel[\psi]{\phi}{\rightrightarrows} \prod_{\rho:[n]\to [m]} (X^m)^{\Delt{n}}
\end{equation}
Here, to a morphism $\rho:[n] \rightarrow [m]$, the map $\phi$ sends a factor of $(X^m)^{\Delt{m}}$ to the factor $((X^m)^{\Delt{n}})_\rho$ using the induced map $(\Delt{\bu})(\rho):\Delt{n} \rightarrow \Delt{m}$.  The map $\psi$ sends a factor of $(X^n)^{\Delt{n}}$ to the factor $((X^m)^{\Delt{n}})_\rho$ using the induced map $X(\rho): X^n \rightarrow X^m$.

Furthermore, for a morphism $F$ of cosimplicial objects $X$ and $Y$ in $\mc C$, i.e., a natural transformation $F:X\to Y$, there is an induced map $\Tot(F):\Tot(X)\to \Tot(Y)$ defined as follows.  The natural transformation $F: X \rightarrow Y$ induces maps $(X^\ell)^{\Delt{\ell}} \rightarrow (Y^\ell)^{\Delt{\ell}}$ and $(X^m)^{\Delt{n}} \rightarrow (Y^m) ^{\Delt{n}}$, which defines a diagram 
 \begin{equation}
 \Tot(X) \rightarrow \prod_{[\ell]\in \Ob(\Del)} (Y^\ell)^{\Delt{\ell}} \stackrel[\psi]{\phi}{\rightrightarrows} \prod_{\rho:[n]\to [m]} (Y^m)^{\Delt{n}}
 \end{equation}  
 Using the universal property of $\Tot(Y)$ then gives us a map $\Tot(F): \Tot(X) \rightarrow \Tot(Y)$.
 \end{definition}

We can also define an algebraic analogue of totalization for cosimplicial non-negatively graded cochain complexes, called the total complex and denoted $\tot$.  Let $K^{\bullet, \bullet}: \Del \rightarrow \Chain^+, [n]\mapsto K^{n,\bu}$, be a cosimplicial non-negatively graded cochain complex.  Then $K^{\bullet,\bullet}$ can be made into a bicomplex where the differential $\delta: K^{\bullet, \bullet} \rightarrow K^{\bullet+1, \bullet}$ is defined by taking the alternating sums of the maps induced by the coface maps $[n] \rightarrow [n+1]$, and the differential $d_K:K^{\bullet, \bullet} \rightarrow K^{\bullet, \bullet+1}$ is given by the differentials $d_n: K^{n,\bullet} \rightarrow K^{n,\bullet+1}$ of the cochain complexes.  We obtain an ordinary cochain complex in $\Chain^+$ in two ways.  One way is by taking the total complex of $K^{\bullet,\bullet}$, by defining 
\begin{equation}\label{EQU:Def:tot}
\tot (K) := \bigoplus_{n\geq 0} K^{n, \bullet}[n],
\end{equation}
where $K^{n, \bullet}[n]$ denotes $K^{n, \bullet}$ shifted up by $n$, with differential $d$, applied to $c\in \tot(K)$ of degree $|c|$, given by
\begin{equation}\label{EQU:d-in-tot}
d(c):=d_K(c)-(-1)^{|c|}\cdot \delta (c).
\end{equation}
The second way is by taking the following equalizer
\begin{equation}\label{EQU:Def:eq-for-CH+}
eq:\prod_{[\ell]} Hom^\bullet(N(\Z\Delt{\ell}) , K^{\ell,\bullet}) \rightrightarrows \prod_{[m]\rightarrow[n]} Hom^\bullet(N(\Z \Delt{m}), K^{n, \bullet}) 
\end{equation}
with differential $d((f_\ell:N(\Z\Delt{\ell}) \to K^{\ell,\bullet})_\ell)=( d_K\circ f_\ell-(-1)^{|f_\ell|}\cdot f_\ell \circ d_{N(\Z \Delt{\ell})})_\ell$. The following lemma shows that the two cochain complexes are equal.

\begin{lemma}\label{lemma:TotComplex}
Let $K^{\bullet,\bullet}: \Del \rightarrow \Chain^+$ be a cosimplicial non-negatively graded cochain complex.  Then the total complex $\tot(K)$ from \eqref{EQU:Def:tot} is isomorphic to the equalizer from \eqref{EQU:Def:eq-for-CH+}.
\end{lemma}

\begin{proof}
An element of degree $k$ in $\prod_{\ell} K^{\ell, \bullet}[\ell]$ is a collection of elements $c^{0,k} \in K^{0,k}[0]$, $c^{1,k-1} \in K^{1,k-1}[1], \cdots c^{k,0} \in K^{k,0}$.  

An element of degree $k$ in $Hom^\bullet(N (\Z \Delt{\ell}), K^{\ell, \bullet})$ is a collection of maps
\begin{eqnarray*}
f^{k,0}_\ell : N(\Z \Delt{\ell})_0 &\rightarrow& K^{\ell, k} \\
f^{k, 1}_\ell: N(\Z \Delt{\ell})_{-1} &\rightarrow& K^{\ell, k-1} \\
&\vdots & \\
f^{k,k}_\ell: N(\Z\Delt{\ell})_{-k} & \rightarrow & K^{\ell,0} 
\end{eqnarray*}
An element of degree $k$ in the product $\prod_{[\ell]} Hom^\bullet(N (\Z \Delt{\ell}), K^{\ell, \bullet})$ is then a collection of these maps over all $[\ell]$, $f^{k, \bullet}_\ell: N(\Z \Delt{\ell})_{-\bullet} \rightarrow K^{\ell, k-\bullet }$.  To be in the equalizer, the maps $\{f^{k,\bullet}_\bullet \}$ must fit in commutative diagrams
$$
\xymatrix{N(\Z \Delt{m})_{-i} \ar^{f^{k,i}_m} [r] \ar[d] &  K^{m,  k-i}   \ar[d]\\ 
N(\Z \Delt{n})_{-i} \ar^{f^{k,i}_n} [r] & K^{n, k-i} }
$$
for every map map $[m] \rightarrow [n]$.   The maps $\{f^{k,\bullet}_\bullet \}$ in the equalizer then are determined by $f^{0,k}_0 : N(\Z \Delt{0})_0 \rightarrow K^{0, k}$, $f^{1,k-1}_1: N(\Z\Delt{1})_{-1} \rightarrow K^{1, k-1}, \cdots, f^{k,0}_k:N(\Z \Delt{k})_{-k} \rightarrow K^{k,0}$. 

Using the Hom-Tensor adjunction, we can identify $f^{i,k}_i$ with $c^{i, k-i} \in K^{i, k-i}[i]$. 
\end{proof}

\end{document}